\documentclass[11pt,reqno]{amsart}
\usepackage{hyperref}

\allowdisplaybreaks[4]
\usepackage{amssymb}
\usepackage{mathrsfs}
\usepackage{amsthm}
\usepackage{amsmath}
\numberwithin{equation}{section}
\usepackage{bm}
\usepackage{bbm}
\usepackage{graphicx}
\usepackage{booktabs}
\usepackage{float}
\usepackage{subfig}
\usepackage{geometry}
\usepackage{color}
\theoremstyle{plain}
\newtheorem{theo}{Theorem}[section]
\newtheorem{lemm}[theo]{Lemma}

\newtheorem{coro}[theo]{Corollary}
\newtheorem{prop}[theo]{Proposition}
\newtheorem{assumption}[theo]{Assumption}

\theoremstyle{remark}
\newtheorem{rema}[theo]{Remark}
\theoremstyle{definition}
\numberwithin{equation}{section}

\newcommand{\R}{{\mathbb R}}

\renewcommand{\H}{{\mathbb H}}
\newcommand{\E}{{\mathbb E}}

\newcommand{\Ee}{{\mathbf E}}
\newcommand{\Hh}{{\mathbf H}}

\newcommand{\dd}{{\rm d}}

\allowdisplaybreaks[4]

\begin{document}
\title[ergodic approximations for stochastic Maxwell equations]{Ergodic numerical approximations for stochastic Maxwell equations}

\author{Chuchu Chen}
\address{LSEC, ICMSEC, Academy of Mathematics and Systems Science, Chinese Academy of Sciences, Beijing 100190, China; School of Mathematical Sciences, University of Chinese Academy of Sciences, Beijing 100049, China}
\curraddr{}
\email{chenchuchu@lsec.cc.ac.cn}
\thanks{}

\author{Jialin Hong}
\address{LSEC, ICMSEC, Academy of Mathematics and Systems Science, Chinese Academy of Sciences, Beijing 100190, China; School of Mathematical Sciences, University of Chinese Academy of Sciences, Beijing 100049, China}
\curraddr{}
\email{hjl@lsec.cc.ac.cn}
\thanks{}

\author{Lihai Ji}
\address{Institute of Applied Physics and Computational Mathematics, Beijing 100094, China}
\email{jilihai@lsec.cc.ac.cn}

\author{Ge Liang}
\address{LSEC, ICMSEC, Academy of Mathematics and Systems Science, Chinese Academy of Sciences, Beijing 100190, China; School of Mathematical Sciences, University of Chinese Academy of Sciences, Beijing 100049, China}
\email{liangge2020@lsec.cc.ac.cn}
	\thanks{The research of this work was supported by the National key R\&D Program of China under Grant NO. 2020YFA0713701, the National Natural Science Foundation of China (NOs. 11971470, 11871068, 12022118, 12031020, 12171047, 11971458) and the Youth Innovation Promotion Association CAS}
\begin{abstract}
 
 In this paper, we propose a novel kind of numerical approximations to inherit the ergodicity of stochastic Maxwell equations.
 The key to proving the ergodicity lies in the uniform regularity estimates of the numerical solutions with respect to time, which are established by analyzing some important physical quantities. By introducing an auxiliary process, we show that the mean-square convergence order of the ergodic discontinuous Galerkin full discretization is $\frac{1}{2}$ in the temporal direction and $\frac{1}{2}$ in the spatial direction, which provides the convergence order of the numerical invariant measure to the exact one in $L^2$-Wasserstein distance.
 
\end{abstract}
\maketitle
\section{introduction}
Stochastic Maxwell equations play an important role in many fields, including statistical radiophysics and stochastic electromagnetism, whose stochasticity may come from the random medium or the stochastic source (see e.g. \cite{RKT1989}).  
In this paper, we consider the following stochastic Maxwell equations in an isotropic conductive medium
{\small\begin{equation}\label{SMEs with PML}
		\begin{cases}
			\dd\Ee(t)={\rm curl\,} \Hh(t)\dd t -\sigma\Ee(t)\dd t-\lambda_1\Hh(t)\circ\dd W_1(t)
			+{\bm\lambda_2}\dd W_2(t)
			, &(t,{\bf x})\in\R_+\times D,\\
			\dd\Hh(t)=-{\rm curl\,} \Ee(t)\dd t-\sigma\Hh(t)\dd t +\lambda_1\Ee(t)\circ\dd W_1(t)
			+{\bm\lambda_2}\dd W_2(t), &(t,{\bf x})\in\R_+\times D,\\
			\Ee(0,{\bf x})=\Ee_0({\bf x}),~\Hh(0,{\bf x})=\Hh_0({\bf x}), &{\bf x}\in D,\\
			{\bf n}\times \Ee=0, ~{\bf n}\cdot \Hh=0,  &(t,{\bf x})\in\R_+\times \partial D,\\
		\end{cases}
\end{equation}	}
where $D=(x_L,x_R)\times(y_L,y_R)\times(z_L,z_R)\subset\R^3$ is a cuboid, the notation $\circ$ means Stratonovich integral, $\Ee=(E_1,E_2,E_3)^{\top}$ and $\Hh=(H_1,H_2,H_3)^{\top}$ are the electromagnetic field,  $\lambda_{1}\in\mathbb{R}$ and  ${\bm\lambda_2}=(\lambda_2^{(1)}, \lambda_2^{(2)},\lambda_2^{(3)})^{\top}\in\mathbb{R}^3$. 
Here, $\{W_1(t)\}_{t\geq 0}$ and $\{W_2(t)\}_{t\geq 0}$ are two independent Wiener processes with respect to a filtered probability space $(\Omega,\mathcal{F},\{\mathcal{F}_t\}_{t\geq 0},\mathbb{P})$, which characterize the randomness from the medium and the source, respectively.
The damping terms $\sigma\Ee$ and $\sigma\Hh$ may be induced by conductivity of the medium or by the perfectly matched layer technique (see e.g. \cite{JCL2019,SHL2016}).
It is assumed that $\sigma\in W^{1,\infty}(D)$ and $\sigma\geq\sigma_0>0$ for a constant $\sigma_0$.  	 
Due to the existence of these damping terms, the properties of stochastic Maxwell equations change tremendously. 
The main aim of this paper is to construct and analyze numerical approximations for \eqref{SMEs with PML} to inherit the intrinsic properties of the original system.

As we all know, ergodicity is an important longtime property of stochastic partial differential equations (SPDEs). There have been many works concentrating on the study of their ergodicity and ergodic numerical approximations; see e.g., \cite{HW2019} for the stochastic nonlinear Schr\"odinger equation and \cite{CHS2021} for parabolic SPDEs. However, to our best knowledge, there is no result on the ergodicity of stochastic Maxwell equations. The main difficulty lies in the uniform estimates of the exact solution with respect to time.
By analyzing some important physical quantities of \eqref{SMEs with PML}, we prove that the solution is bounded uniformly in $L^2(\Omega,H^1(D)^6)$ in time.
This result ensures that the $H^1(D)^6$-norm of the solution is a proper choice for Lyapunov function, which leads to the existence of the invariant measure for \eqref{SMEs with PML}. Furthermore, we obtain the continuous dependence of the solution on the initial data with exponential decay rate. As a consequence, \eqref{SMEs with PML} possesses a unique invariant measure $\pi^*$ which is ergodic and exponentially mixing. Moreover, we show that the phase flow of \eqref{SMEs with PML} possesses the stochastic conformal multi-symplectic structure.

It is meaningful and important to design structure-preserving numerical discretizations since they have remarkable superiority in the longtime computation. In recent years, various symplectic and multi-symplectic numerical methods have been proposed and analyzed for stochastic Maxwell equations without damping terms.
We refer the interested readers to \cite{CHZ2016,HJZ2014,HJZC2017} for the stochastic multi-symplectic methods, to \cite{CHJ2019b} for the symplectic Runge--Kutta method, to \cite{CCC2021,sun2022discontinuous,SQW2022} for the symplectic and multi-symplectic discontinuous Galerkin methods, to \cite{CHJ2019a,CHCS2020} for the semi-implicit and exponential Euler methods. For stochastic Maxwell equations with damping terms, we do not find any relative works on constructing numerical discretizations to inherit the ergodicity up to now.

To this end, we first propose a novel temporal semi-discretization for \eqref{SMEs with PML} which is a modification of the midpoint method. This temporal semi-discretization is specially constructed to preserve the ergodicity and the stochastic conformal multi-symplecticity simultaneously.
Compared with the continuous case, the proof of the ergodicity of the temporal semi-discretization is more complicated since the curl and divergence of the numerical solution must be estimated together. By establishing the uniform boundedness of the numerical solution in $L^2(\Omega,H^1(D)^6)$ with respect to time, we show that the temporal semi-discretization is ergodic with a unique invariant measure $\pi^{\Delta t}$.        
The mean-square convergence order of the temporal semi-discretization is shown to be $\frac{1}{2}$, which provides the convergence order of the numerical invariant measure $\pi^{\Delta t}$  to the exact one $\pi^*$ in $L^2$-Wasserstein distance.

Further, we apply the discontinuous Galerkin (dG) method to discretize the temporal semi-discretization in space. The ergodicity of the dG numerical solution is obtained by establishing the uniform boundedness of the numerical solution in $L^2(\Omega,\H)$. The mean-square convergence analysis of the dG full discretization is more challenging due to the low regularity of the numerical solution. To solve this problem, we introduce an auxiliary process in our convergence analysis, which allows us to take full advantage of the $ H^1(D)^6$-regularity of the exact solution.
We show that the mean-square convergence order of the dG full discretization is $\frac{1}{2}$ both in the temporal and spatial directions. As a byproduct, the $L^2$-Wasserstein distance between the numerical invariant measure $\pi^{\Delta t,h}$ and the exact one $\pi^*$ is estimated.   
We remark that there exist many alternative choices of the spatial discretization. For example, we also use a finite difference method to discretize the temporal semi-discretization in space. 
We prove that this full discretization preserves the ergodicity by deriving the uniform boundedness of the averaged discrete energy, and meanwhile it possesses the discrete stochastic conformal multi-symplectic conservation law. 

The rest of this paper is organized as follows. In Section 2, we first introduce some notations and focus on studying the uniformly boundedness, ergodicity and stochastic conformal multi-symplecticity of \eqref{SMEs with PML}. In Section 3, an ergodic modified midpoint temporal semi-discretization of \eqref{SMEs with PML} is proposed. We establish the mean-square convergence of the numerical scheme. Section 4 is devoted to designing ergodic full discretizations.
The details of the proof to a prior estimates of some operators are provided in Appendices.

\section{Properties of stochastic Maxwell equations}
In this section, we investigate the regularities, ergodicity and the stochastic conformal multi-symplecticity of the solution of \eqref{SMEs with PML}, which make preparations for the numerical approximations in the rest sections. 
Throughout this paper, $C$ will be used to denote a generic positive constant independent of time. 

Let $W^{k,p}(D)$ be the standard Sobolev space. Especially, we denote $H^k(D):=W^{k,2}(D)$. Denote the Euclidean norm in $\mathbb{R}^6$ by $|\cdot|$. We define the Maxwell operator by
\begin{equation*}
	M=\begin{pmatrix}
		0 & \nabla\times\\
		-\nabla\times &0
	\end{pmatrix}
\end{equation*}
with ${\mathcal D}(M):=H_{0}({\rm curl},D)\times H({\rm curl},D)$, which is skew-adjoint on $\H:=L^2(D)^3\times L^2(D)^3$ (see e.g., \cite{CHJ2019a}).
Denote by $HS(U,H)$ the Banach space of all Hilbert--Schmidt operators from one separable Hilbert space $U$ to another separable Hilbert space $H$, equipped with the norm
$$\|\Gamma\|_{HS(U,H)}=\Big(\sum_{j=1}^{\infty}\|\Gamma q_j\|^2_{H}\Big)^{\frac{1}{2}}\quad \forall\,\Gamma\in HS(U,H),$$
where $\{q_j\}_{j\in\mathbb{N}_+}$ is an orthonormal basis of $U$.
For the Wiener processes, we give the following assumption.
\begin{assumption}\label{ass}
	
	For $i=1,2$, assume that $W_i(t)$ is a $Q_i$-Wiener process on the filtered probability space $(\Omega,\mathcal{F},\{\mathcal{F}_t\}_{t\geq0},\mathbb{P})$, which can be represented as  $$W_i(t)=\sum_{k=1}^{\infty}\sqrt{\eta_k^{(i)}}q_k\beta^{(i)}_k(t),\quad t\geq0,$$ 
	where $Q_iq_k=\eta_{k}^{(i)} q_k$ with $\eta_{k}^{(i)}\geq 0$ and $\{q_k\}_{k\in\mathbb{N}_+}$ being the orthonormal basis of $L^2(D)$.
	In addition, assume that $Q_{i}^{\frac12}\in HS(L^2(D),H^{\gamma_{i}}(D))=:\mathcal{L}_2^{\gamma_{i}}$ for some $\gamma_{i}\geq 0$.
	
\end{assumption}

Let $u=(\Ee^{\top},\Hh^{\top})^{\top}$ and  $u_0=(\Ee_0^{\top},\Hh_0^{\top})^{\top}$. We can rewrite  \eqref{SMEs with PML} as a stochastic evolution equation
\begin{equation}\label{SEE_S}
	\left\{
	\begin{split}
		&\dd u(t)=\Big(M u(t)-\sigma u(t)\Big)\dd t+\lambda_1 J u(t)\circ \dd W_1(t)+\widetilde{\bm{\lambda}}_2\dd W_2(t),\quad t>0,\\
		&u(0)=u_0,
	\end{split}\right.
\end{equation}
where $\widetilde{\bm{\lambda}}_2=(\bm{\lambda}_2^{\top},\bm{\lambda}_2^{\top})^{\top}$, $J=\begin{pmatrix} 0&-I_3\\I_3 &0 \end{pmatrix}$ with $I_3$ being the identity matrix on $\mathbb{R}^{3\times3}$.
The equivalent It\^o formulation of  \eqref{SEE_S} reads as
\begin{equation}\label{SEE_I}
	\dd u(t)=\Big(M u(t)-\sigma u(t)-\frac12\lambda_1^2F_{Q_1} u(t)\Big)\dd t+\lambda_1 J u(t)\dd W_1(t)+\widetilde{\bm{\lambda}}_2\dd W_2(t)
\end{equation}
for $t>0$, where $F_{Q_1}=\sum_{k=1}^{\infty}\eta_k^{(1)} (q_k)^2$.

The following proposition states the well-posedness and the uniform boundedness in $L^2(\Omega,\H)$ of the solution of \eqref{SEE_I}.
\begin{prop}\label{H-boundedness}
	Let Assumption \ref{ass} hold with $\gamma_1\geq0$ and $\gamma_2\geq 0$, and let $u_0\in L^2(\Omega,\H)$. Then  the system \eqref{SEE_I} is well-posed.
	Moreover, there exists a positive constant $C_1:=C_1(\sigma_0,\widetilde{{\bm{\lambda}}}_2,{\rm tr}(Q_2))$  such that  
	\begin{equation*}
		\E\big[\|u(t)\|_{\H}^2\big]\leq e^{-2\sigma_0t}\mathbb{E}\big[\|u_0\|_{\H}^2\big]+C_1\big(1-e^{-2\sigma_0t}\big).
	\end{equation*}
\end{prop}
\begin{proof}
	The well-posedness of \eqref{SEE_I} follows similarly to \cite[Theorem 2.1]{CHJ2019a}.
	By applying the It\^o formula to $\|u\|_{\H}^2$, we have
	\begin{align}\label{qwe}
		\dd\|u(t)\|_{\H}^2
		=&-2\langle u(t), \sigma u(t)\rangle_{\H}\dd t+|\widetilde{{\bm{\lambda}}}_2|^2{\rm tr}(Q_2)\dd t+2\langle u(t),  \widetilde{{\bm{\lambda}}}_2\dd W_2(t)\rangle_{\H}\notag\\
		\leq&-2\sigma_0\|u(t)\|^2_{\H}\dd t+|\widetilde{{\bm{\lambda}}}_2|^2{\rm tr}(Q_2)\dd t+2\langle u(t),  \widetilde{{\bm{\lambda}}}_2\dd W_2(t)\rangle_{\H}
	\end{align}
	due to the skew-adjointness of $M$.
	By taking the expectation on both sides of \eqref{qwe} and using the Gronwall inequality, we have
	\begin{align*}
		\E\big[\|u(t)\|_{\H}^2\big]&\leq e^{-2\sigma_0 t} \mathbb{E}\big[\|u_0\|_{\H}^2\big]+(1-e^{-2\sigma_0t})\frac{|\widetilde{{\bm{\lambda}}}_2|^2 {\rm tr}(Q_2)}{2\sigma_0}\\
		&=: e^{-2\sigma_0t}\mathbb{E}\big[\|u_0\|_{\H}^2\big]+(1-e^{-2\sigma_0t})C_1.
	\end{align*}
	
	Thus we finish the proof.
\end{proof}

\subsection{Ergodicity}
In this part, we investigate the ergodicity of \eqref{SMEs with PML}, that is, the existence and uniqueness of the invariant measure.
Let $P_t\varphi(x):=\mathbb{E}[\varphi(u(t))],t\geq 0$ which is the Markov transition semigroup associated to the solution $u$ of \eqref{SMEs with PML}. Denote by $\mathcal{P}(\H)$ the space of all Borel probability measures on $\H$. A probability measure $\pi\in\mathcal{P}(\H)$ is said to be invariant for $u$ if
$
\int_{\H}P_t\varphi(x)\pi(\dd x)=\int_{\H}\varphi\dd \pi=:\pi(\varphi)$
for any Borel bounded mapping $\varphi$ and $t\geq0$. Further, $u$ is said to be ergodic on $\H$ if 
$
\lim_{T\rightarrow\infty}\frac{1}{T}\int_{0}^{T}\mathbb{E}[\varphi(u(t))]\dd t=\pi(\varphi)\,\text{in }L^2(\H,\pi)
$
for all $\varphi\in L^2(\H,\pi)$; $u$ is said to be exponentially mixing on $\H$ if there exist a positive constant $\rho$ and a positive function $C(\cdot)$ such that for any bounded Lipschitz continuous function $\varphi$ on $\H$, all $t>0$ and all $u_0\in \H$,
$\big|P_t\varphi(u_0)-\pi(\varphi)\big|\leq C(u_0)L_\varphi e^{-\rho t}$
with $L_{\varphi}$ being the Lipschitz constant of $\varphi$.

In order to obtain the ergodicity, we first derive the uniform boundedness of the solution of \eqref{SMEs with PML} in $L^2(\Omega,H^1(D)^6)$ in the following lemma.
\begin{lemm}\label{regularity in H1}
	Let Assumption \ref{ass} hold with $\gamma_1:=1+\gamma>\frac52$ and $\gamma_2\geq1$, let $u_0\in L^2(\Omega,H^1(D)^6)$ and $F_{Q_1}\in W^{1,\infty}(D)$. Then the solution of \eqref{SMEs with PML} is bounded uniformly in time, i.e.,
	\begin{align}\label{H1bdd}
		\E\big[\|u(t)\|_{H^1(D)^6}^2\big]\leq C_2e^{-\sigma_0t}	\E\big[\|u_0\|_{H^1(D)^6}^2\big]+C_3,
	\end{align}
	where the positive constants $C_2$ depends on $|D|$, and $C_3$ depends on  $\sigma_0, \|\sigma\|_{W^{1,\infty}(D)}$, $\|F_{Q_1}\|_{W^{1,\infty}(D)}$,
	$C_1,$
	$\mathbb{E}\big[\|u_0\|_{\H}^2\big],$
	$ \lambda_{1}$ ,$\widetilde{{\bm{\lambda}}}_2$ ,$\|Q_1^{\frac{1}{2}}\|_{\mathcal{L}_2^{\gamma_1}}$ and $\|Q_2^{\frac{1}{2}}\|_{\mathcal{L}_2^{\gamma_2}}$. 
\end{lemm}
\begin{proof}
	{\em Step 1. Uniform boundedness of the curl of the solution.}
	
	Applying the It\^o formula to $\|Mu\|_{\H}^2$, we have
	\begin{align}\label{ito1}
		\dd\mathbb{E}\big[\|Mu(t)\|_{\H}^2\big]=&\,-2\mathbb{E}\big[\langle Mu(t),M\big((\sigma+\frac12\lambda_1^2F_{Q_1}) u(t)\big)\rangle_{\H}\big]\dd t\notag\\
		&+ \lambda_1^2\mathbb{E}\Big[\sum_{k=1}^{\infty}\big\|M(Ju(t)Q_1^{\frac12}q_k)\big\|_{\H}^2\Big]\dd t+\sum_{k=1}^{\infty}\big\|M(\widetilde{{\bm{\lambda}}}_2Q_2^{\frac12}q_k)\big\|_{\H}^2\dd t.
	\end{align}
	We note that 
	\begin{align}\label{yuy}
		-2\E\Big[\Big\langle Mu,M\big((\sigma+\frac12\lambda_1^2F_{Q_1}) u\big)\Big\rangle_{\H}\Big]
		&=-2\E\Big[\Big\langle Mu,(\sigma+\frac12\lambda_1^2F_{Q_1}) Mu\Big\rangle_{\H}\Big]\notag\\
		&\quad-2\E\Big[\Big\langle Mu, \begin{pmatrix} \nabla(\sigma+\frac12\lambda_1^2F_{Q_1})\times\Hh \\
			-\nabla(\sigma+\frac12\lambda_1^2F_{Q_1})\times\Ee \end{pmatrix} \Big\rangle_{\H}\Big]\\
		&\leq -\frac{3}{2}\sigma_0\E\big[\|Mu\|^2_{\H}\big]-\lambda_1^2\E\big[\langle Mu,F_{Q_1}Mu\rangle_{\H}\big]+C,\notag
	\end{align}
	where we use the fact that
	\begin{equation*}\label{yuuu}
		\begin{split}
			&-2\E\Big[\Big\langle Mu, \begin{pmatrix} \nabla(\sigma+\frac12\lambda_1^2F_{Q_1})\times\Hh \\
				-\nabla(\sigma+\frac12\lambda_1^2F_{Q_1})\times\Ee \end{pmatrix} \Big\rangle_{\H}\Big]\\
			&\leq 
			4\E\big[\|Mu\|_{\H}\|u\|_{\H}\big]\|\nabla(\sigma+\frac12\lambda_1^2F_{Q_1})\|_{L^\infty(D)^3}\leq \frac{1}{2}\sigma_0\E[\|Mu\|^2_{\H}]+C
		\end{split}
	\end{equation*}
	due to the Young inequality, the assumption $\sigma\in W^{1,\infty}(D,\mathbb{R})$ and Proposition \ref{H-boundedness}.

	By using the Sobolev embedding $H^{\gamma}(D)\hookrightarrow L^{\infty}(D)$ for $\gamma>\frac{3}{2}$, the H\"older inequality, the Young inequality and Proposition \ref{H-boundedness}, it holds that
	\begin{align}\label{poi}
		&\lambda_1^2\E\Big[\sum_{k=1}^{\infty}\|M(JuQ_1^{\frac12}q_k)\|_{\H}^2\Big]+\sum_{k=1}^{\infty}\|M(\widetilde{{\bm{\lambda}}}_2Q_2^{\frac12}q_k)\|_{\H}^2\notag\\
		&=\lambda_1^2\E[\langle Mu,F_{Q_1}Mu\rangle_{\H}]
		+\lambda_{1}^2\E\Bigg[\sum_{k=1}^{\infty}\Bigg{\|}
		\left(          
		\begin{array}{ccc}  
			\nabla(Q_1^\frac{1}{2}q_k)\times\Ee\\
			\nabla(Q_1^\frac{1}{2}q_k)\times\Hh\\
		\end{array}
		\right)
		\Bigg{\|}^2_{\H}\Bigg]\notag
		\\
		&\quad+2\lambda_{1}^2\E\Bigg[\sum_{k=1}^{\infty}\Bigg{\langle}
		\left(          
		\begin{array}{ccc}  
			(Q_1^\frac{1}{2}q_k)\nabla\times\Ee\\
			(Q_1^\frac{1}{2}q_k)\nabla\times\Hh\\
		\end{array}
		\right),
		\left(          
		\begin{array}{ccc}  
			\nabla(Q_1^\frac{1}{2}q_k)\times\Ee\\
			\nabla(Q_1^\frac{1}{2}q_k)\times\Hh\\
		\end{array}
		\right)
		\Bigg{\rangle}_\mathbb{H}\Bigg]	\\
		&\quad+2\sum_{k=1}^{\infty}\|
		{\bm\lambda_2}\times\nabla(Q_2^{\frac{1}{2}}q_k)
		\|^2_{L^2(D)^3}\notag\\
		&\leq \lambda_1^2\E\big[\langle Mu,F_{Q_1}Mu\rangle_{\H}\big]
		+\frac12 \sigma_0\E\big[\|Mu\|^2_{\H}\big]+C.\notag
	\end{align}
	
	Substituting \eqref{yuy} and \eqref{poi} into \eqref{ito1}, we obtain
	\begin{equation*}\label{ito2}
		\begin{split}
			\dd\mathbb{E}\big[\|Mu(t)\|_{\H}^2\big]\leq -\sigma_0\E\big[\|Mu(t)\|^2_{\H}\big]\dd t+C\dd t,
		\end{split}
	\end{equation*}
	which combining the Gronwall inequality yields
	$$\mathbb{E}\big[\|Mu(t)\|_{\H}^2\big]\leq e^{-\sigma_0t}\mathbb{E}\big[\|Mu_0\|_{\H}^2\big]+C.$$

	{\em Step 2. Uniform boundedness of the divergence of the solution.}
	
	Applying the It\^o formula to $\|\nabla\cdot\Ee\|^2_{L^2(D)}$ and taking the expectation, we arrive at 
	\begin{equation}\label{uyu2}
		\begin{split}
			\dd \mathbb{E}[\|\nabla\cdot\Ee(t)\|^2_{L^2(D)}]&=-2\mathbb{E}\Big[\Big\langle \nabla\cdot\Ee(t), \nabla\cdot\Big((\sigma+\frac12\lambda_1^2 F_{Q_1})\Ee(t)\Big)\Big\rangle_{L^2(D)}\Big]\dd t\\
			&\quad+\lambda_1^2\mathbb{E}\Big[\sum_{k=1}^{\infty}\|\nabla\cdot(\Hh(t) Q_1^{\frac12}q_k)\|_{L^2(D)}^2\Big]\dd t\\
			&\quad+\sum_{k=1}^{\infty}\|\nabla\cdot({\bm \lambda}_2Q_2^{\frac12}q_k)\|_{L^2(D)}^2\dd t.
		\end{split}
	\end{equation}
	By the Sobolev embedding $H^{\gamma}(D)\hookrightarrow L^{\infty}(D)$ for $\gamma>\frac{3}{2}$, the H\"older inequality, the Young inequality, the assumption $\sigma,F_{Q_1}\in W^{1,\infty}(D)$ and Proposition \ref{H-boundedness}, we have
	\begin{align}\label{uyu1}
		&-2\E\Big[\Big\langle \nabla\cdot\Ee, \nabla\cdot\big((\sigma+\frac12\lambda_1^2 F_{Q_1})\Ee\big)\Big\rangle_{L^2(D)}\Big]\notag\\
		&=-2\E\Big[\Big\langle \nabla\cdot\Ee, \nabla\big(\sigma+\frac12\lambda_1^2 F_{Q_1}\big)\cdot\Ee\Big\rangle_{L^2(D)}\Big]\notag\\
		&\quad-2\E\Big[\Big\langle \nabla\cdot\Ee, (\sigma+\frac12\lambda_1^2 F_{Q_1})\nabla\cdot\Ee\Big\rangle_{L^2(D)}\Big]\\
		&\leq -\frac32 \sigma_0 \E[ \|\nabla\cdot\Ee\|^2_{L^2(D)}]-\lambda_1^2\E\Big[\langle \nabla\cdot\Ee, F_{Q_1}\nabla\cdot\Ee\rangle_{L^2(D)}\Big]+C\notag
	\end{align}
	and
	\begin{equation}\label{uyu}
		\begin{split}
			&\lambda_{1}^2\E\Big[\sum_{k=1}^{\infty}\|\nabla\cdot(\Hh Q_1^{\frac12}q_k)\|_{L^2(D)}^2\Big]\\
			&=\lambda_{1}^2\E\Big[\Big\langle \nabla\cdot\Hh, F_{Q_1}\nabla\cdot\Hh\Big\rangle_{L^2(D)}\Big]+\lambda_{1}^2\E\Big[\sum_{k=1}^{\infty}\|\Hh\cdot\nabla(Q_1^{\frac12} q_k)\|^2_{L^2(D)}\Big]
			\\
			&\quad+2\lambda_{1}^2\E\Big[\sum_{k=1}^{\infty}\Big\langle Q_1^{\frac12} q_k\nabla\cdot\Hh, \Hh\cdot \nabla(Q_1^{\frac12} q_k)\Big\rangle_{L^2(D)}\Big]\\
			&\leq \lambda_{1}^2\E\big[\langle \nabla\cdot\Hh, F_{Q_1}\nabla\cdot\Hh\rangle_{L^2(D)}\big]
			+\frac12\sigma_0\E\big[\|\nabla\cdot\Hh\|^2_{L^2(D)}\big]+C.
		\end{split}
	\end{equation}
	
	Using the fact that $\sum_{k=1}^{\infty}\|\nabla\cdot({\bm \lambda}_2Q_2^{\frac12}q_k)\|_{L^2(D)}^2\leq C(|{\bm\lambda_2}|,\|Q_2^{\frac{1}{2}}\|_{\mathcal{L}_2^{\gamma_2}})$ and substituting \eqref{uyu1}--\eqref{uyu} into \eqref{uyu2}, we have
	\begin{align*}
		\dd \E[\|\nabla\cdot\Ee(t)\|^2_{L^2(D)}]
		\leq&\, -\frac32 \sigma_0 \E [\|\nabla\cdot\Ee(t)\|^2_{L^2(D)}]\dd t
		+\frac12\sigma_0\E[\|\nabla\cdot\Hh(t)\|^2_{L^2(D)}]\dd t\\
		&-\lambda_1^2\E\Big[\Big\langle \nabla\cdot\Ee(t), F_{Q_1}\nabla\cdot\Ee(t)\Big\rangle_{L^2(D)}\Big]\dd t\\
		&+\lambda_1^2\E\Big[\Big\langle \nabla\cdot\Hh(t), F_{Q_1}\nabla\cdot\Hh(t)\Big\rangle_{L^2(D)}\Big]\dd t
		+C\dd t.
	\end{align*}
	Similarly, one gets
	\begin{align*}
		\dd \E[\|\nabla\cdot\Hh(t)\|^2_{L^2(D)}]
		\leq&\,-\frac32 \sigma_0 \E[\|\nabla\cdot\Hh(t)\|^2_{L^2(D)}]\dd t
		+\frac12\sigma_0\E[\|\nabla\cdot\Ee(t)\|^2_{L^2(D)}]\dd t\\
		&-\lambda_1^2\E\Big[\Big\langle \nabla\cdot\Hh(t), F_{Q_1}\nabla\cdot\Hh(t)\Big\rangle_{L^2(D)}\Big]\dd t
		\\
		&+\lambda_1^2\E\Big[\Big\langle \nabla\cdot\Ee(t), F_{Q_1}\nabla\cdot\Ee(t)\Big\rangle_{L^2(D)}\Big]\dd t
		+C\dd t.
	\end{align*}
	Combining them together, we have
	\begin{align*}
		&\dd\E\Big[\|\nabla\cdot\Ee(t)\|^2_{L^2(D)}+\|\nabla\cdot\Hh(t)\|^2_{L^2(D)}\Big]\\
		&\leq\, -\sigma_0\E\Big[\|\nabla\cdot\Ee(t)\|^2_{L^2(D)}+\|\nabla\cdot\Hh(t)\|^2_{L^2(D)}\Big]\dd t+C\dd t,
	\end{align*}
	which by the Gronwall inequality implies
	$$\E\Big[\|\nabla\cdot\Ee(t)\|^2_{L^2(D)}+\|\nabla\cdot\Hh(t)\|^2_{L^2(D)}\Big]\leq e^{-\sigma_0t}\mathbb{E}\Big[\|\nabla\cdot\Ee_0\|^2_{L^2(D)}+\|\nabla\cdot\Hh_0\|^2_{L^2(D)}\Big]+C.$$
	
	{\em Step 3. Proof of }\eqref{H1bdd}.
	
	By utilizing the fact that $v\in H({\rm curl},D)\cap H({\rm div},D)$ belongs to $H^1(D)^3$ if ${\bf n}\times v|_{\partial D}=0$ or ${\bf n}\cdot v|_{\partial D}=0$, we can get the time-independent $H^1(D)^6$-regularity of the solution of \eqref{SMEs with PML}. Combining {\em Step 1}, {\em Step 2} and Proposition \ref{H-boundedness}, we obtain
	\begin{align*}
		\E\big[\|u(t)\|_{H^1(D)^6}^2\big]&\leq C\Big(\E\big[\|u(t)\|_{\H}^2\big]+\E\big[\|\nabla\times\Ee(t)\|^2_{L^2(D)^3}\big]+\E\big[\|\nabla\cdot\Ee(t)\|^2_{L^2(D)}\big]\\
		&\quad\quad\quad+\E\big[\|\nabla\times\Hh(t)\|^2_{L^2(D)^3}\big]+\E\big[\|\nabla\cdot\Hh(t)\|^2_{L^2(D)}\big]\Big)\\
		&\leq C_2e^{-\sigma_0t}	\E\big[\|u_0\|_{H^1(D)^6}^2\big]+C_3.
	\end{align*}
	Thus we finish the proof.
\end{proof}

By applying the It\^o formula to $\|u(t)\|^{2p}_{\H}$ and $\|Mu(t)\|^{2p}_{\H}$, respectively, we can similarly obtain the uniform boundedness of $\mathbb{E}[\|u(t)\|^{2p}_{\mathcal{D}(M)}]$ in the following proposition. 
\begin{prop}\label{obv}
	Set $p\in\mathbb{N}_+$. Let Assumption \ref{ass} hold with $\gamma_1:=1+\gamma>\frac52$ and $\gamma_2\geq1$, let $u_0\in L^{2p}(\Omega,\mathcal{D}(M))$ and $F_{Q_1}\in W^{1,\infty}(D)$.
	Then the solution of \eqref{SMEs with PML} is bounded uniformly in time, i.e.,
	\begin{align}\label{H1bdd}
		\E\big[\|u(t)\|_{\mathcal{D}(M)}^{2p}\big]\leq \widetilde{C}_{2,p}e^{-\sigma_0t}	\E\big[\|u_0\|_{\mathcal{D}(M)}^{2p}\big]+\widetilde{C}_{3,p},
	\end{align}
	where the positive constants $\widetilde{C}_{2,p}$ depends on $p,|D|$, and $\widetilde{C}_{3,p}$ depends on  $p,\sigma_0$, $\|\sigma\|_{W^{1,\infty}(D)}$, $\|F_{Q_1}\|_{W^{1,\infty}(D)}$,
	$C_1,$
	$\mathbb{E}\big[\|u_0\|_{\H}^{2p}\big],$
	$ \lambda_{1}$ ,$\widetilde{{\bm{\lambda}}}_2$ ,$\|Q_1^{\frac{1}{2}}\|_{\mathcal{L}_2^{\gamma_1}}$ and $\|Q_2^{\frac{1}{2}}\|_{\mathcal{L}_2^{\gamma_2}}$. 
\end{prop}

Let $\pi\in\mathcal{P}(\H)$ and denote the transpose operator of $P_t$ by $P_t^*$. Based on Lemma \ref{regularity in H1}, we obtain the ergodicity of stochastic Maxwell equations and the convergence of $P_t^*\pi$ towards the invariant measure in the $L^2$-Wasserstein distance as $t\rightarrow\infty$ in the following proposition.

\begin{prop}\label{SMEErgo}
	Under the conditions in Lemma \ref{regularity in H1}, the following statements hold.\\
	{\rm (i)} The solution $u$ of \eqref{SMEs with PML} possesses a unique invariant measure $\pi^{*}\in\mathcal{P}_2(\H)$, where $\mathcal{P}_2(\H)=\{\mu\in\mathcal{P}(\H):\int_{\H}\|x\|_{\H}^2\mu(\dd x)<\infty\}$. Thus $u$ is ergodic. Moreover, $u$ is exponentially mixing.\\
	{\rm (ii)} For any distribution $\pi\in\mathcal{P}_2(\H)$, 
	\begin{equation*}
		\mathcal{W}_2(P^*_t\pi,\pi^{*})\leq e^{-\sigma_0t}\mathcal{W}_2(\pi,\pi^*).
	\end{equation*}
\end{prop}
\begin{proof}
	(i)	Let $u$ and $\widetilde{u}$ be solutions of \eqref{SMEs with PML} with initial data $u_0$ and $\widetilde{u}_0$, respectively. Similarly to Proposition \ref{H-boundedness},
	we get
	$$ \mathbb{E}\big[\|u(t)-\widetilde{u}(t)\|^2_{\H}\big]\leq\mathbb{E}\big[\|u(0)-\widetilde{u}(0)\|^2_{\H}\big]-2\sigma_0\int_{0}^{t}\mathbb{E}\big[\|u(s)-\widetilde{u}(s)\|^2_{\H}\big]\dd s$$
	which by the Gronwall inequality yields
	\begin{equation}\label{mixing}
		\E\big[\|u(t)-\widetilde{u}(t)\|^2_{\H}\big]\leq e^{-2\sigma_0t}\E\big[\|u_0-\widetilde{u}_0\|^2_{\H}\big].
	\end{equation}
	By Lemma \ref{regularity in H1}, we can choose $\|\cdot\|_{H^1(D)^6}$ as a proper Lyapunov function. 
	We note that $H^1(D)^6$ is compactly embedded in $\H$. Therefore, the level set $K_{\alpha}:=\{v\in\H:\|v\|_{H^1(D)^6}\leq \alpha\}$ is compact for any constant $\alpha>0$. By using the Krylov--Bogoliubov theorem (see \cite[Proposition 7.10]{DaPrato2008}), the general Harris' theorem (see \cite[Theorem 4.8]{MartingHairer2011}) and \eqref{mixing}, we get the existence and the uniqueness of invariant measure $\pi^*$. By Proposition \ref{H-boundedness} and \cite[Proposition 4.24]{hairer2006ergodic}, it holds that
	\begin{equation}\label{ert}
		\int_{\H}\|\textbf{r}\|^2_{\H}\pi^*(\dd \textbf{r})<C_1,
	\end{equation}
	which implies that $\pi^*\in\mathcal{P}_2(\H).$
	
	For any bounded Lipschitz continuous function $\varphi$ on $\H$ and all $u_0\in \H$,
	\begin{align*}
		&\big|P_t\varphi(u_0)-\pi^*(\varphi)\big|
		=\Big|\int_{\H}\mathbb{E}\big[\varphi(u(t,u_0))-\varphi(u(t,\textbf{r}))\big]\pi^*(\dd \textbf{r})\Big|\\
		&\leq L_{\varphi}\int_{\H}\mathbb{E}\big[\|u(t,u_0)-u(t,\textbf{r})\|\big]\pi^*(\dd \textbf{r})
		\leq L_{\varphi}\int_{\H}\big(\mathbb{E}\big[\|u(t,u_0)-u(t,\textbf{r})\|^2_{\H}\big]\big)^{\frac{1}{2}}\pi^*(\dd \textbf{r})\\
		&\leq L_{\varphi}e^{-\sigma_0t}\int_{\H}\|u_0-\textbf{r}\|_{\H}\pi^*(\dd\textbf{r})
		\leq \Big(\|u_0\|_{\H}+\int_{\H}\|\textbf{r}\|_{\H}\pi^*(\dd \textbf{r})\Big)L_{\varphi}e^{-\sigma_0t}\\
		&\leq (\|u_0\|_{\H}+C_1)L_{\varphi}e^{-\sigma_0t}
	\end{align*}
	due to the H\"older inequality, \eqref{mixing} and \eqref{ert}. Hence the exponentially mixing property is proved.
	
	(ii) We fix initial values $(u_0,\widetilde{u}_0)\in\H\times\H$ and denote initial distributions by $\delta_{u_0}$ and $\delta_{\widetilde{u}_0}$.  Let $u$ and $\widetilde{u}$ be solutions of \eqref{SMEs with PML} with initial data $u_0$ and $\widetilde{u}_0$, respectively. Denote the joint distribution of $(u(t),\widetilde{u}(t))$ by $\mathcal{J}(P^*_t\delta_{u_0},P^*_t\delta_{\widetilde{u}_0})$. By \eqref{mixing}, we have
	\begin{align*}\label{tgb}
		&\mathcal{W}_2(P^*_t\delta_{u_0},P_t^*\delta_{\widetilde{u}_0})
		\leq \Bigg(\int_{\H\times\H}\|\textbf{r}_1-\textbf{r}_2\|_{\H}^2\mathcal{J}(P_t^*\delta_{u_0},P_t^*\delta_{\widetilde{u}_0})(\dd\textbf{r}_2,\dd\textbf{r}_2)\Bigg)^{\frac{1}{2}}\notag\\
		&=\Bigg(\int_{\Omega}\|u(t)-\widetilde{u}(t)\|_{\H}^2\dd \mathbb{P}\Bigg)^{\frac{1}{2}}\leq e^{-\sigma_0t}\|u_0-\widetilde{u}_0\|_{\H}=e^{-\sigma_0t}\mathcal{W}_2(\delta_{u_0},\delta_{\widetilde{u}_0}).
	\end{align*}
	Combining with the convexity of $L^2$-Wasserstein distance (see e.g., \cite[Theorem 4.8]{optimaltransport}) and the H\"older inequality, we obtain that for any coupling $\gamma$ of $\pi$ and $\pi^*$,
	\begin{align*}
		&\mathcal{W}_2(P^*_t\pi,P^*_t\pi^*)\leq \int_{\H\times\H}\mathcal{W}_2(P^*_t\delta_{\textbf{r}_1},P^*_t\delta_{\textbf{r}_2})\gamma(\dd \textbf{r}_1,\dd \textbf{r}_2)\\
		&\leq e^{-\sigma_0t}\int_{\H\times\H}\mathcal{W}_2(\delta_{\textbf{r}_1},\delta_{\textbf{r}_2})\gamma(\dd \textbf{r}_1,\dd \textbf{r}_2)
		\leq e^{-\sigma_0t}\Big(\int_{\H\times\H}\|\textbf{r}_1-\textbf{r}_2\|_{\H}^2\gamma(\dd\textbf{r}_1,\dd \textbf{r}_2)\Big)^{\frac{1}{2}}.
	\end{align*} 
	Thus we finish the proof.
\end{proof}

\subsection{Stochastic conformal multi-symplecticity}
This part is devoted to studying the stochastic conformal multi-symplecticity of \eqref{SMEs with PML}.
We set
$S_1(u)=\frac{\lambda_1}{2}\big(|\Ee|^2+|\Hh|^2\big)$, $S_2(u)={\bm\lambda_2}\cdot \Ee-{\bm\lambda_2}\cdot\Hh$ and
$$
F=\begin{pmatrix}
	0 & I_3 \\ -I_3 &0
\end{pmatrix},\quad
K_s=\begin{pmatrix}
	{\mathcal D}_s & 0 \\ 0 & {\mathcal D}_s
\end{pmatrix},~
s=1,2,3
$$
with
\[
{\mathcal D}_1=\begin{pmatrix}
	0 & 0&0\\ 0&0&-1\\ 0&1&0
\end{pmatrix},\quad
{\mathcal D}_2=\begin{pmatrix}
	0 & 0&1\\ 0&0&0\\ -1&0&0
\end{pmatrix},\quad
{\mathcal D}_3=\begin{pmatrix}
	0 & -1&0\\1&0&0\\ 0&0&0
\end{pmatrix}.
\]	
Then,  \eqref{SMEs with PML} can be rewritten into the following form of damped stochastic Hamiltonian PDE:
\begin{equation*}
	F\dd u+K_1 \partial_x u\dd t+K_2\partial_y u\dd t+K_3\partial_z u \dd t=-\sigma Fu\dd t+\nabla_u S_1(u)\circ \dd W_1(t)+\nabla_u S_2(u)\circ \dd W_2(t).
\end{equation*}

\begin{theo}
	The system \eqref{SMEs with PML} possesses the stochastic conformal multi-symplectic conservation law
	\begin{equation*}\label{eq_1.8}
		\dd w+\partial_x\kappa_1\dd t+\partial_y\kappa_2\dd t+\partial_z\kappa_3\dd t=-2\sigma w\dd t,\qquad\mathbb{P}\text{-}a.s.,
	\end{equation*}
	i.e.,
	\begin{align*}
		&\int_{z_{0}}^{z_{1}}\int_{y_{0}}^{y_{1}}\int_{x_0}^{x_1}\omega(t_{1},x,y,z){\rm d}x{\rm d}y{\rm d}z+
		\int_{z_{0}}^{z_{1}}\int_{y_{0}}^{y_{1}}\int_{t_0}^{t_1}\kappa_{1}(t,x_{1},y,z){\rm d}t{\rm d}y{\rm d}z\\
		&+\int_{z_{0}}^{z_{1}}\int_{x_{0}}^{x_{1}}\int_{t_0}^{t_1}\kappa_{2}(t,x,y_{1},z){\rm d}t{\rm d}x{\rm d}z+
		\int_{y_{0}}^{y_{1}}\int_{x_{0}}^{x_{1}}\int_{t_0}^{t_1}\kappa_{3}(t,x,y,z_{1}){\rm d}t{\rm d}x{\rm d}y\\
		&-\int_{z_{0}}^{z_{1}}\int_{y_{0}}^{y_{1}}\int_{x_0}^{x_1}\omega(t_{0},x,y,z){\rm d}x{\rm d}y{\rm d}z-
		\int_{z_{0}}^{z_{1}}\int_{y_{0}}^{y_{1}}\int_{t_0}^{t_1}\kappa_{1}(t,x_{0},y,z){\rm d}t{\rm d}y{\rm d}z\\
		&-\int_{z_{0}}^{z_{1}}\int_{x_{0}}^{x_{1}}\int_{t_0}^{t_1}\kappa_{2}(t,x,y_{0},z){\rm d}t{\rm d}x{\rm d}z
		-\int_{y_{0}}^{y_{1}}\int_{x_{0}}^{x_{1}}\int_{t_0}^{t_1}\kappa_{3}(t,x,y,z_{0}){\rm d}t{\rm d}x{\rm d}y\\
		&=-2\int_{t_0}^{t_1}\int_{z_{0}}^{z_{1}}\int_{y_{0}}^{y_{1}}\int_{x_0}^{x_1}\sigma\omega(t,x,y,z)\dd x\dd y\dd z\dd t,
	\end{align*}
	where $\omega(t,x,y,z)=\frac{1}{2}du\wedge
	Fdu$, $\kappa_{s}(t,x,y,z)=\frac{1}{2}du\wedge
	K_{s}du$ $(s=1,2,3)$ are the differential 2-forms
	associated with the skew-symmetric matrices $F$ and
	$K_{s}$, respectively, and
	$(t_{0},t_{1})\times(x_{0},x_{1})\times(y_{0},y_{1})\times(z_{0},z_{1})$
	is the local domain of $u(t,x,y,z)$.
\end{theo}	
\begin{proof}
	The proof is similar to that of \cite[Theorem 1]{stochasticKGequation}, thus we omit it here.
\end{proof}

\section{Ergodic temporal semi-discretization}
In this section, we propose a temporal semi-discretization to inherit the ergodicity and stochastic conformal multi-symplecticity of \eqref{SMEs with PML}. Moreover, the mean-square convergence order of the temporal semi-discretization is derived.

By introducing a uniform partition in time interval $[0,+\infty)$ with time step-size $\Delta t$, we propose a modified midpoint method to discretize \eqref{SEE_S} in the temporal direction
\begin{equation}\label{ns}
	u^{n+1}-e^{-\sigma\Delta t}u^n=\Delta tM\frac{u^{n+1}+e^{-\sigma\Delta t}u^n}{2}+\lambda_{1}J\frac{u^{n+1}+e^{-\sigma\Delta t}u^n}{2}\Delta \overline{W}_1^n+\widetilde{{\bm{\lambda}}}_2\Delta W_2^n
\end{equation}
for $n\in\mathbb{N}$. Here, $\Delta W_2^n:=W_2(t_{n+1})-W_2(t_n)$. Since the diffusion term of \eqref{ns} is implicit and the noise could be unbounded for arbitrary small time step-size, we truncate the noise $\Delta W_1^n$ by another random variable $$\Delta\overline{W}_1^n:=\sqrt{\Delta t}\sum_{i=1}^{\infty}\zeta^{(1),n}_iQ_1^{\frac{1}{2}}q_i,$$
where
\begin{center}
	$\zeta^{(1),n}_i=\left\{
	\begin{aligned}
		A_{\Delta t}&\quad&\xi_i^{(1),n}>A_{\Delta t},\\
		\xi_i^{(1),n}&\quad&|\xi_i^{(1),n}|\leq A_{\Delta t},\\
		-A_{\Delta t}&\quad&\xi_i^{(1),n}<-A_{\Delta t}
	\end{aligned}
	\right.$
\end{center}
with $\{\xi_i^{(1),n}\}_{i\in\mathbb{N}_+}$ being a family of independent standard normal random variables and $A_{\Delta t}=\sqrt{2b|\ln\Delta t|},\hphantom{=}b\geq 4$.
Similarly to \cite{CH2016}, it holds that 
\begin{align*}
	\mathbb{E}\big[|\zeta_i^{(1),n}-\xi_i^{(1),n}|^{2p}\big]\leq \mathbb{E}\big[|\xi_i^{(1),n}|^{2p}\big]\Delta t^b\qquad\forall p\in\mathbb{N}_+,
\end{align*}
which implies that for $Q_1^{\frac{1}{2}}\in\mathcal{L}_2^{\gamma_1}$, 
\begin{align}\label{2ptruncate}
	\mathbb{E}\big[\|\Delta\overline{W}_1^n-\Delta W_1^n\|^{2p}_{H^{\gamma_1}(D)}\big]\leq C\Delta t^{b+p}\qquad \forall p\in\mathbb{N}_+
\end{align}
and 
\begin{align}\label{wfzq}
	\mathbb{E}\big[\|(\Delta\overline{W}_1^k)^2-(\Delta W_1^k)^2\|^2_{H^{\gamma_1}(D)}\big]\leq C\Delta t^{b+1}.
\end{align}

The well-posedness and the uniform boundedness of the numerical solution of \eqref{ns} in $L^p(\Omega,\H)$ are stated in the following proposition. 
\begin{prop}\label{uniformH}
	Set $p\geq 1$. Let Assumption \ref{ass} hold with $\gamma_1\geq1$ and $\gamma_2\geq 1$, and let $u_0\in L^{2p}(\Omega,\H)$. Then for sufficiently small $\Delta t>0$, there uniquely exists a family of $\H$-valued and $\{\mathcal{F}_{t_n}\}_{n\in\mathbb{N}}$-adapted solution $\{u^n\}_{n\in\mathbb{N}}$ of \eqref{ns}, which satisfies 
	\begin{equation}\label{unH}
		\mathbb{E}\big[\|u^n\|_{\H}^{2p}\big]\leq e^{-(2p-1)\sigma_0n\Delta t}\mathbb{E}\big[\|u_0\|_{\H}^{2p}\big]+C_{4,p}\quad\forall\, n\in\mathbb{N},
	\end{equation}
	where the positive constant $C_{4,p}:=C\big(p,\sigma_0,\lambda_{1},\widetilde{{\bm{\lambda}}}_2,\mathbb{E}\big[\|u_0\|_{\H}^{2p}\big]$,
	$\|Q_1^{\frac{1}{2}}\|_{\mathcal{L}_2^{\gamma_1}},
	\|Q_2^{\frac{1}{2}}\|_{\mathcal{L}_2^{\gamma_2}}\big)$.
\end{prop}
\begin{proof}
	The proof of the existence and uniqueness of the numerical solution is similar to that of \cite[Lemma 4.1]{CHJ2019a}, so we omit it here. In the following we focus on proving the assertion \eqref{unH}.
	\par {\em Step 1: Case $p=1$.} 
	
	We apply $\langle\cdot,u^{n+1}+e^{-\sigma\Delta t}u^n\rangle_{\H}$ on both sides of \eqref{ns} to get
	\begin{align}\label{p2}
		\|u^{n+1}\|_{\H}^2-e^{-2\sigma_0\Delta t}\|u^n\|_{\H}^2\leq \langle u^{n+1}+e^{-\sigma\Delta t}u^n,\widetilde{{\bm{\lambda}}}_2\Delta W_2^n\rangle_{\H}
	\end{align}
	due to the skew-adjointness of the Maxwell operator $M$.
	By taking the expectation and using the fact that $\Delta W_2^n$ is independent of $\mathcal{F}_{t_n}$, we have	
	\begin{equation}\label{tyt}
		\begin{split}
			\mathbb{E}\big[\|u^{n+1}\|_{\H}^2\big]&\leq 
			e^{-2\sigma_0\Delta t}\mathbb{E}\big[\|u^n\|^2_{\H}\big]+\mathbb{E}\big[\langle u^{n+1}-e^{-\sigma\Delta t}u^n,\widetilde{{\bm{\lambda}}}_2\Delta W_2\rangle_{\H}\big].
		\end{split}
	\end{equation}
	Substituting \eqref{ns} into the second term of the right side of \eqref{tyt} leads to
	\begin{align}\label{yyy}
		\mathbb{E}\big[\|u^{n+1}\|_{\H}^2\big]
		&\leq e^{-2\sigma_0\Delta t}\mathbb{E}\big[\|u^n\|^2_{\H}\big]+\frac{\Delta t}{2}\mathbb{E}\big[\langle Mu^{n+1},\widetilde{\bm{\lambda}}_2\Delta W_2^n\rangle_{\H}\big]\notag\\
		&\hphantom{=}+\frac{\lambda_{1}}{2}\mathbb{E}\big[\langle Ju^{n+1}\Delta \overline{W}_1^n,\widetilde{{\bm{\lambda}}}_2\Delta W_2^n\rangle_{\H}\big]+\mathbb{E}\big[\|\widetilde{{\bm{\lambda}}}_2\Delta W_2^n\|^2_{\H}\big].
	\end{align}
	For the second term on the right side of \eqref{yyy}, using the skew-adjointness of $M$ and the Young inequality, we have
	\begin{equation*}
		\begin{split}
			\frac{\Delta t}{2}\mathbb{E}\big[\langle Mu^{n+1},\widetilde{\bm{\lambda}}_2\Delta W_2^n\rangle_{\H}\big]&\leq 
			\frac{\sigma_0}{8}\Delta t\mathbb{E}\big[\|u^{n+1}\|^2_{\H}\big]
			+C(\sigma_0)\Delta t\mathbb{E}\big[\|M(\widetilde{{\bm{\lambda}}}_2\Delta W_2^n)\|_{\H}^2\big]\\
			&\leq \frac{\sigma_0}{8}\Delta t\mathbb{E}\big[\|u^{n+1}\|^2_{\H}\big]+C\Delta t^2.
		\end{split}
	\end{equation*}
	Based on the H\"older inequality, the Young inequality and the Sobolev embedding $H^{1}(D)\hookrightarrow L^{4}(D)$, the third term of the right side of \eqref{yyy} is estimated as follows
	\begin{equation}\label{yyy2}
		\begin{split}
			&\frac{\lambda_{1}}{2}\mathbb{E}\big[\langle Ju^{n+1}\Delta \overline{W}_1^n,\widetilde{{\bm{\lambda}}}_2\Delta W_2^n\rangle_{\H}\big]\\
			&\leq\frac{\sigma_0}{8}\Delta t\mathbb{E}\big[\|u^{n+1}\|^2_{\H}\big]+\frac{C(\sigma_0,\lambda_1,\widetilde{{\bm{\lambda}}}_2)}{\Delta t}\mathbb{E}\big[\|\Delta \overline{W}_1^n\|^2_{H^{1}(D)}\|\Delta W_2^n\|^2_{H^1(D)}\big]\\
			&\leq \frac{\sigma_0}{8}\Delta t\mathbb{E}\big[\|u^{n+1}\|_{\H}^2\big]+C\Delta t.
		\end{split}
	\end{equation}
	Noting that $\mathbb{E}[||\widetilde{{\bm{\lambda}}}_2\Delta W_2^n||^2_{\H}]\leq C(\widetilde{{\bm{\lambda}}}_2,\text{tr}(Q_2))\Delta t$ and combining \eqref{yyy}--\eqref{yyy2}, we get
	\begin{align*}
		&\mathbb{E}\big[\|u^{n+1}\|_{\H}^2\big]
		\leq e^{-2\sigma_0\Delta t}\mathbb{E}\big[\|u^n\|^2_{\H}\big]+\frac{\sigma_0}{4}\Delta t\mathbb{E}\big[\|u^{n+1}\|_{\H}^2\big]+C\Delta t. 
	\end{align*}
	There exists a $\Delta t^*>0$ (e.g., $\Delta t^*=\frac{3}{\sigma_0}$) ,
	such that for any $\Delta t\in(0, \Delta t^*]$, 
	\begin{equation*}
		\frac{1}{1-\sigma_0\Delta t/4}\leq 1+\sigma_0\Delta t\leq e^{\sigma_0\Delta t}.
	\end{equation*}
	Then the Gronwall inequality implies the assertion for the case $p=1$.

	\par {\em Step 2: Case $p\geq 2$.}
	
	We give the proof of the case $p=2$ since the proofs for the cases $p>2$ are similar. By multiplying $\|u^{n+1}\|^2_{\H}$ on both sides of \eqref{p2} and taking the expectation, it yields
	\begin{align}\label{uuu}
		&\frac{1}{2}\mathbb{E}\big[\|u^{n+1}\|^4_{\H}\big]-\frac{1}{2}e^{-4\sigma_0\Delta t}\mathbb{E}\big[\|u^n\|_{\H}^4\big]+\frac{1}{2}\mathbb{E}\big[(\|u^{n+1}\|_{\H}^2-e^{-2\sigma_0\Delta t}\|u^n\|^2_{\H})^2\big]\notag\\
		&\leq\mathbb{E}\big[\langle u^{n+1}-e^{-\sigma\Delta t}u^n,\widetilde{{\bm{\lambda}}}_2\Delta W_2^n\rangle_{\H}\|u^{n+1}\|_{\H}^2\big]\\
		&\quad+2\mathbb{E}\big[\langle e^{-\sigma\Delta t}u^n,\widetilde{{\bm{\lambda}}}_2\Delta W_2^n\rangle_{\H}\|u^{n+1}\|_{\H}^2\big]\notag\\
		&=:I+II.\notag
	\end{align}	
	For the first term $I$, we substitute \eqref{ns} into it, and use the H\"older inequality and the Young inequality to get
	\begin{align*}
		I&=\frac{\Delta t}{2}\mathbb{E}\big[\langle M(u^{n+1}+e^{-\sigma\Delta t}u^n),\widetilde{{\bm{\lambda}}}_2\Delta W_2^n\rangle_{\H}\|u^{n+1}\|^2_{\H}\big]\notag\\
		&\quad+\frac{\lambda_{1}}{2}\mathbb{E}\big[\langle J(u^{n+1}+e^{-\sigma\Delta t}u^n)\Delta\overline{W}_1^n,\widetilde{{\bm{\lambda}}}_2\Delta W_2^n\rangle_{\H}\|u^{n+1}\|^2_{\H}\big]+\mathbb{E}\big[\|\widetilde{{\bm{\lambda}}}_2\Delta W_2^n\|^2_{\H}\|u^{n+1}\|^2_{\H}\big]\notag\\
		&\leq \frac{\Delta t}{2}\mathbb{E}\big[\|u^{n+1}+e^{-\sigma\Delta t}u^n\|_{\H}\|M(\widetilde{{\bm{\lambda}}}_2\Delta W_2^n)\|_{\H}\|u^{n+1}\|^2_{\H}\big]\notag\\
		&\quad+C\mathbb{E}\big[\|u^{n+1}+e^{-\sigma\Delta t}u^n\|_{\H}\|\Delta\overline{W}_1^n\|_{L^4(D)}\|\Delta W^n_2\|_{L^4(D)}\|u^{n+1}\|^2_{\H}\big]\\
		&\quad+\mathbb{E}\big[\|\widetilde{{\bm{\lambda}}}_2\Delta W_2^n\|^2_{\H}\|u^{n+1}\|^2_{\H}\big]\notag\\
		&\leq \frac{\sigma_0}{8}\Delta t\mathbb{E}\big[\|u^{n+1}\|_{\H}^4\big]+\frac{\sigma_0}{16}\Delta te^{-4\sigma_0\Delta t}\mathbb{E}\big[\|u^n\|_{\H}^4\big]+C\Delta t,
	\end{align*}	
	where in the last step we use the Sobolev embedding  $H^{1}(D)\hookrightarrow L^{4}(D)$.
	For the second term $II$, it yields
	\begin{equation*}\label{00}
		\begin{split}
			II&=2\mathbb{E}\big[\langle e^{-\sigma\Delta t}u^n,\widetilde{{\bm{\lambda}}}_2\Delta W_2^n\rangle_{\H}(\|u^{n+1}\|_{\H}^2-e^{-2\sigma_0\Delta t}\|u^n\|_{\H}^2)\big]\\
			&\leq\frac{1}{2}\mathbb{E}\big[(\|u^{n+1}\|^2_{\H}-e^{-2\sigma_0\Delta t}\|u^n\|_{\H}^2)^2\big]+\frac{\sigma_0}{16}\Delta te^{-4\sigma_0\Delta t}\mathbb{E}\big[\|u^n\|^4_{\H}\big]+C\Delta t.
		\end{split}
	\end{equation*}	
	Combining \eqref{uuu} and the estimates of $I$ and $II$, we obtain   
	\begin{align*}
		\mathbb{E}\big[\|u^{n+1}\|_{\H}^4\big]\leq& e^{-4\sigma_0\Delta t}\mathbb{E}\big[\|u^n\|_{\H}^4\big]+\frac{\sigma_0}{4}\Delta t\mathbb{E}\big[\|u^{n+1}\|_{\H}^4\big]+\frac{\sigma_0}{4}\Delta te^{-4\sigma_0\Delta t}\mathbb{E}\big[\|u^n\|_{\H}^4\big]+C\Delta t.
	\end{align*}
	There exists a $\Delta t^*>0$ (e.g., $\Delta t^*=\frac{2}{\sigma_0})$, 
	such that for any $\Delta t\in(0, \Delta t^*]$, 
	\begin{equation*}
		\frac{1+\sigma_0\Delta t/4}{1-\sigma_0\Delta t/4}\leq 1+\sigma_0\Delta t\leq e^{\sigma_0\Delta t}.
	\end{equation*}
	By the Gronwall inequality, \eqref{unH} holds for the case $p=2$.
	Thus we complete the proof.
\end{proof}

\subsection{Properties of the temporal semi-discretization}
We are now in the position to study the ergodicity and stochastic multi-symplecticity of \eqref{ns}.
\subsubsection{Ergodicity} In order to show the ergodicity of \eqref{ns}, we first give the uniform boundedness of the numerical solution in $L^2(\Omega,H^1(D)^6)$.
\begin{lemm}\label{uniformdivM}
	Set $\gamma>\frac{3}{2}$. Let Assumption \ref{ass} hold with $\gamma_1:=2+\gamma$ and $\gamma_2\geq 2$ and let $\sum_{i=1}^{\infty}\|Q_1^{\frac{1}{2}}e_i\|_{H^{1+\gamma}(D)}<\infty$, and $u_0\in L^2(\Omega,H^1(D)^6)$. Then there exists positive constants $C_5$ and $C_6$ such that for sufficiently small $\Delta t>0$,
	\begin{equation}\label{div}
		\sup_{n\geq 0}\mathbb{E}\big[\|u^n\|^2_{H^1(D)^6}\big]\leq C_5e^{-\sigma_0n\Delta t}\mathbb{E}\big[\|u_0\|^2_{H^1(D)^6}\big]+C_6,
	\end{equation}	
	where $C_5$ depends on $|D|$, and $C_6$ depends on $\sigma_0,C_{4,1},C_{4,2},\lambda_{1},\widetilde{{\bm{\lambda}}}_2,\mathbb{E}\big[\|u_0\|^4_{\H}\big],\|Q_1^{\frac{1}{2}}\|_{\mathcal{L}_2^{\gamma_1}}$ and $\|Q_2^{\frac{1}{2}}\|_{\mathcal{L}_2^{\gamma_2}}$.
\end{lemm}
\begin{proof}
	By utilizing the fact that $v\in H({\rm curl},D)\cap H({\rm div},D)$ belongs to $H^1(D)^3$ if ${\bf n}\times v|_{\partial D}=0$ or ${\bf n}\cdot v|_{\partial D}=0$, it is sufficient to show
	\begin{align}\label{div13}
		&\mathbb{E}\big[\|Mu^n\|^{2}_{\H}+\|\nabla\cdot \Ee^n\|^{2}_{L^2(D)}+\|\nabla\cdot \Hh^n\|^{2}_{L^2(D)}\big]\notag\\
		&\leq e^{-\sigma_0t}\mathbb{E}\big[\|Mu_0\|^{2}_{\H}+\|\nabla\cdot \Ee_0\|^{2}_{L^2(D)}+\|\nabla\cdot \Hh_0\|^{2}_{L^2(D)}\big]+C.
	\end{align}
	We prove \eqref{div13} in the following two steps.
	
	{\em Step 1. Estimate of $\,\mathbb{E}\big[\|\nabla\cdot \Ee^n\|^{2}_{L^2(D)}+\|\nabla\cdot \Hh^n\|^{2}_{L^2(D)}\big]$.}
	
	It follows from \eqref{ns} that
	\begin{align}
		\Ee^{n+1}-e^{-\sigma\Delta t}\Ee^n&=\frac{\Delta t}{2}\nabla\times\big(\Hh^{n+1}+e^{-\sigma\Delta t}\Hh^n\big)\notag\\
		&\qquad-\frac{\lambda_{1}}{2}\big(\Hh^{n+1}+e^{-\sigma\Delta t}\Hh^n\big)\Delta\overline{W}_1^n+\bm{\lambda}_2\Delta W^n_2,\label{div1}\\
		\Hh^{n+1}-e^{-\sigma\Delta t}\Hh^n&=-\frac{\Delta t}{2}\nabla\times\big(\Ee^{n+1}+e^{-\sigma\Delta t}\Ee^n\big)\notag\\
		&\qquad+\frac{\lambda_{1}}{2}\big(\Ee^{n+1}+e^{-\sigma\Delta t}\Ee^n\big)\Delta\overline{W}_1^n+\bm{\lambda}_2\Delta W_2^n.\label{div2}
	\end{align}
	After applying $\nabla\cdot$ on both sides of \eqref{div1} and \eqref{div2}, we have
	\begin{align}
		\nabla\cdot \Ee^{n+1}-e^{-\sigma\Delta t}\big(\nabla\cdot \Ee^n\big)&=\big(\nabla e^{-\sigma\Delta t}\big)\cdot\Ee^n-\frac{\lambda_{1}}{2}\nabla\cdot\big[\big(\Hh^{n+1}+e^{-\sigma\Delta t }\Hh^n\big)\Delta\overline{W}_1^n\big]\notag\\
		&\qquad+\bm{\lambda}_2\cdot\nabla(\Delta W_2^n),\label{div11}\\
		\nabla\cdot \Hh^{n+1}-e^{-\sigma\Delta t}\big(\nabla\cdot \Hh^n\big)&=\big(\nabla e^{-\sigma\Delta t}\big)\cdot\Hh^n+\frac{\lambda_{1}}{2}\nabla\cdot\big[\big(\Ee^{n+1}+e^{-\sigma\Delta t }\Ee^n\big)\Delta\overline{W}_1^n\big]\notag\\
		&\qquad+\bm{\lambda}_2\cdot\nabla(\Delta W_2^n).\label{div22}
	\end{align}
	Next we apply $\langle\cdot,\nabla\cdot \Ee^{n+1}+e^{-\sigma\Delta t}\nabla\cdot \Ee^n\rangle_{L^2(D)}$ and $\langle\cdot,\nabla\cdot \Hh^{n+1}+e^{-\sigma\Delta t}\nabla\cdot \Hh^n\rangle_{L^2(D)}$ on both sides of \eqref{div11} and \eqref{div22}, respectively, to get
	\begin{align}\label{uioo}
		&\|\nabla\cdot \Ee^{n+1}\|^2_{L^2(D)}+\|\nabla\cdot \Hh^{n+1}\|^2_{L^2(D)}\notag\\
		&-\|e^{-\sigma\Delta t}(\nabla\cdot \Ee^n)\|^2_{L^2(D)}-\|e^{-\sigma\Delta t}(\nabla\cdot \Hh^n)\|^2_{L^2(D)}\notag\\
		&=\langle (\nabla e^{-\sigma\Delta t})\cdot\Ee^n,\nabla\cdot \Ee^{n+1}+e^{-\sigma\Delta t}\nabla\cdot \Ee^n\rangle_{L^2(D)}\notag\\
		&\quad+\langle(\nabla e^{-\sigma\Delta t})\cdot\Hh^n,\nabla\cdot \Hh^{n+1}+e^{-\sigma\Delta t}\nabla\cdot \Hh^n\rangle_{L^2(D)}\notag\\
		&\quad+\langle\bm{\lambda}_2\cdot\nabla(\Delta W_2^n),\nabla\cdot \Ee^{n+1}+e^{-\sigma\Delta t}\nabla\cdot \Ee^n\rangle_{L^2(D)}\\
		&\quad+\langle\bm{\lambda}_2\cdot\nabla(\Delta W_2^n),\nabla\cdot \Hh^{n+1}+e^{-\sigma\Delta t}\nabla\cdot \Hh^n\rangle_{L^2(D)}\notag\\
		&\quad-\frac{\lambda_{1}}{2}\langle\nabla\cdot[(\Hh^{n+1}+e^{-\sigma\Delta t }\Hh^n)\Delta\overline{W}_1^n],\nabla\cdot \Ee^{n+1}+e^{-\sigma\Delta t}\nabla\cdot \Ee^n\rangle_{L^2(D)}\notag\\
		&\quad+\frac{\lambda_{1}}{2}\langle\nabla\cdot[(\Ee^{n+1}+e^{-\sigma\Delta t }\Ee^n)\Delta\overline{W}_1^n],\nabla\cdot \Hh^{n+1}+e^{-\sigma\Delta t}\nabla\cdot \Hh^n\rangle_{L^2(D)}\notag\\
		&=:\sum_{k=1}^{6}A_k.\notag
	\end{align}
	For terms $A_1$ and $A_2$, by the Young inequality, Proposition \ref{uniformH} and $\nabla e^{-\sigma\Delta t}=-\Delta te^{-\sigma\Delta t}\nabla\sigma$, we have
	\begin{align*}
		&\mathbb{E}\big[A_{1}+A_2\big]\\
		&\leq \mathbb{E}\big[\|(\nabla e^{-\sigma\Delta t})\cdot \Ee^n\|_{L^2(D)}(\|\nabla\cdot \Ee^{n+1}\|_{L^2(D)}+\|e^{-\sigma\Delta t}(\nabla\cdot \Ee^n)\|_{L^2(D)})\big]\\
		&\quad+\mathbb{E}\big[\|(\nabla e^{-\sigma\Delta t})\cdot \Hh^n\|_{L^2(D)}(\|\nabla\cdot \Hh^{n+1}\|_{L^2(D)}+\|e^{-\sigma\Delta t}(\nabla\cdot \Hh^n)\|_{L^2(D)})\big]\\
		&\leq\frac{\sigma_0}{16}\Delta t\mathbb{E}\big[\|\nabla\cdot \Ee^{n+1}\|_{L^2(D)}^2+\|\nabla\cdot \Hh^{n+1}\|_{L^2(D)}^2\big]+C\Delta t\\
		&\quad+\frac{\sigma_0}{16} e^{-2\sigma_0\Delta t}\Delta t\mathbb{E}\big[\|\nabla\cdot \Ee^n\|^2_{L^2(D)}+\|\nabla\cdot \Hh^n\|^2_{L^2(D)}\big].
	\end{align*}
	By using the fact that $\Delta\overline{W}_1^n$ and $\Delta W_2^n$ are independent of $\mathcal{F}_{t_n}$, and substituting \eqref{div11} and \eqref{div22} into $A_{3}$ and $A_{4}$, respectively, it yields 
	\begin{align*}
		&\mathbb{E}\big[A_{3}+A_{4}\big]\\
		&=2\mathbb{E}\big[\|\bm{\lambda}_2\cdot\nabla(\Delta W_2^n)\|^2_{L^2(D)}\big]\\
		&\quad-\frac{\lambda_{1}}{2}\mathbb{E}\big[\langle\bm{\lambda}_2\cdot\nabla(\Delta W_2^n),\Delta\overline{W}_1^n(\nabla\cdot \Hh^{n+1})+\Hh^{n+1}\cdot\nabla(\Delta\overline{W}_1^n)\rangle_{L^2(D)}\big]\\
		&\quad+\frac{\lambda_{1}}{2}\mathbb{E}\big[\langle\bm{\lambda}_2\cdot\nabla(\Delta W_2^n),\Delta\overline{W}_1^n(\nabla\cdot \Ee^{n+1})+\Ee^{n+1}\cdot\nabla(\Delta\overline{W}_1^n)\rangle_{L^2(D)}\big]\\
		&\leq\frac{\sigma_0}{16}\Delta t\mathbb{E}\big[\|\nabla\cdot \Ee^{n+1}\|_{L^2(D)}^2+\|\nabla\cdot \Hh^{n+1}\|_{L^2(D)}^2\big]+C\Delta t
	\end{align*}
	in view of the Sobolev embedding $H^{\gamma}(D)\hookrightarrow L^{\infty}(D)$ for $\gamma>\frac{3}{2}$, the Young inequality and Proposition \ref{uniformH}.
	
	For terms $A_5$ and $A_6$, we note that
	\begin{align*}
		&\nabla\cdot\big[(\Hh^{n+1}+e^{-\sigma\Delta t}\Hh^n)\Delta\overline{W}_1^n\big]\\
		&=\Delta\overline{W}_1^n\big(\nabla\cdot \Hh^{n+1}+e^{-\sigma\Delta t}(\nabla\cdot \Hh^n)\big)+
		\Hh^{n+1}\cdot\nabla(\Delta\overline{W}_1^n)\\
		&\quad+ e^{-\sigma\Delta t}\Hh^n\cdot\nabla(\Delta\overline{W}_1^n)+
		\Delta\overline{W}_1^n(\nabla e^{-\sigma\Delta t})\cdot \Hh^n
	\end{align*}
	and
	\begin{align*}
		&\nabla\cdot\big[(\Ee^{n+1}+e^{-\sigma\Delta t}\Ee^n)\Delta\overline{W}_1^n\big]\\
		&=
		\Delta\overline{W}_1^n\big(\nabla\cdot \Ee^{n+1}+ e^{-\sigma\Delta t}(\nabla\cdot \Ee^n)\big)+
		\Ee^{n+1}\cdot\nabla(\Delta\overline{W}_1^n)\\
		&\quad+ e^{-\sigma\Delta t}\Ee^n\cdot\nabla(\Delta\overline{W}_1^n)+
		\Delta\overline{W}_1^n (\nabla e^{-\sigma\Delta t})\cdot\Ee^n.
	\end{align*}
	Therefore,
	\begin{align*}
		\mathbb{E}[A_5+A_6]=&-\frac{\lambda_1}{2}\mathbb{E}\Big[\langle \Hh^{n+1}\cdot\nabla(\Delta\overline{W}_1^n),\nabla\cdot \Ee^{n+1}+e^{-\sigma\Delta t}(\nabla\cdot \Ee^n)\rangle_{L^2(D)}\Big]\\
		&-\frac{\lambda_1}{2}\mathbb{E}\Big[\langle e^{-\sigma\Delta t}\Hh^n\cdot\nabla(\Delta\overline{W}_1^n)+\Delta\overline{W}_1^n \Hh^n\cdot(\nabla e^{-\sigma\Delta t}),\nabla\cdot \Ee^{n+1}\rangle_{L^2(D)}\Big]\\
		&+\frac{\lambda_1}{2}\mathbb{E}\Big[\langle \Ee^{n+1}\cdot\nabla(\Delta\overline{W}_1^n),\nabla\cdot \Hh^{n+1}+e^{-\sigma\Delta t}(\nabla\cdot \Hh^n)\rangle_{L^2(D)}\Big]\\
		&+\frac{\lambda_1}{2}\mathbb{E}\Big[\langle e^{-\sigma\Delta t}\Ee^n\cdot\nabla(\Delta\overline{W}_1^n)+\Delta\overline{W}_1^n \Ee^n\cdot(\nabla e^{-\sigma\Delta t}),\nabla\cdot \Hh^{n+1}\rangle_{L^2(D)}\Big]\\
		=:&A_{5,1}+A_{5,2}+A_{6,1}+A_{6,2}
	\end{align*}
	due to the fact that $\Delta\overline{W}_1^n$ and $\Delta W_2^n$ are independent of $\mathcal{F}_{t_n}$.
	First, we consider the term $A_{5,1}$,
	\begin{align*}
		A_{5,1}=&-\frac{\lambda_1}{2}\mathbb{E}\Big[\langle \Hh^{n+1}\cdot\nabla(\Delta\overline{W}_1^n),\nabla\cdot \Ee^{n+1}-e^{-\sigma\Delta t}(\nabla\cdot \Ee^n)\rangle_{L^2(D)}\Big]\\
		&\hphantom{=}-\lambda_{1}\mathbb{E}\Big[\langle (\Hh^{n+1}-e^{-\sigma\Delta t}\Hh^n)\cdot\nabla(\Delta\overline{W}_1^n),e^{-\sigma\Delta t}(\nabla\cdot \Ee^n)\rangle_{L^2(D)}\Big]\\
		=&:A_{5,1,1}+A_{5,1,2}.
	\end{align*}
	Substituting \eqref{div11} into the term $A_{5,1,1}$ and using the Sobolev embedding $H^{\gamma}(D)\hookrightarrow L^{\infty}(D)$ for $\gamma>\frac{3}{2}$ lead to
	\begin{align*}
		A&_{5,1,1}
		=-\frac{\lambda_{1}}{2}\mathbb{E}\Big[\langle \Hh^{n+1}\cdot\nabla(\Delta\overline{W}_1^n),(\nabla e^{-\sigma\Delta t})\cdot \Ee^n+\bm{\lambda}_2\cdot\nabla(\Delta W_2^n)\rangle_{L^2(D)}\Big]\\
		&+\frac{\lambda_{1}^2}{4}\mathbb{E}\Big[\langle \Hh^{n+1}\cdot(\nabla\Delta\overline{W}_1^n),\Delta\overline{W}_1^n(\nabla\cdot \Hh^{n+1})\rangle_{L^2(D)}+\| \Hh^{n+1}\cdot\nabla(\Delta\overline{W}_1^n)\|^2_{L^2(D)}\Big]\\
		&+\frac{\lambda_{1}^2}{4}\mathbb{E}\Big[\Big\langle \Hh^{n+1}\cdot\nabla(\Delta\overline{W}_1^n),\Delta\overline{W}_1^ne^{-\sigma\Delta t}(\nabla\cdot \Hh^n)\\
		&\qquad\qquad\qquad+\Big(e^{-\sigma\Delta t}\nabla(\Delta \overline{W}_1^n)+\Delta\overline{W}_1^n(\nabla e^{-\sigma\Delta t})\Big)\cdot \Hh^n\Big\rangle_{L^2(D)}\Big]\\
		\leq&\,C\mathbb{E}\Big[\|\Delta\overline{W}_1^n\|_{H^{1+\gamma}(D)}\|\Hh^{n+1}\|_{L^2(D)^3}\Big(\Delta t\|\Ee^n\|_{L^2(D)^3}+\|\Delta W_2^n\|_{H^{1}(D)}\Big)\Big]\\
		&+C\mathbb{E}\Big[\|\Delta\overline{W}_1^n\|^2_{H^{1+\gamma}(D)}\|\Hh^{n+1}\|_{L^2(D)^3}\Big(\|\Hh^{n+1}\|_{L^2(D)^3}+\|\nabla\cdot \Hh^{n+1}\|_{L^2(D)}\Big)\Big]\\
		&+C\mathbb{E}\Big[\|\Delta\overline{W}_1^n\|^2_{H^{1+\gamma}(D)}\|\Hh^{n+1}\|_{L^2(D)^3}\Big((1+\Delta t)\|\Hh^n\|_{L^2(D)^3}+e^{-\sigma_0\Delta t}\|\nabla\cdot \Hh^n\|_{L^2(D)}\Big)\Big]\\
		\leq&\,\frac{\sigma_0}{16}\Delta t\mathbb{E}\Big[\|\nabla\cdot \Hh^{n+1}\|_{L^2(D)}^2\Big]+\frac{\sigma_0}{16}\Delta te^{-2\sigma_0\Delta t}\mathbb{E}\Big[\|\nabla\cdot \Hh^n\|^2_{L^2(D)}\Big]+C\Delta t,
	\end{align*}
	where in the last step we use the Young inequality and Proposition \ref{uniformH}.
	We substitute \eqref{div2} into $A_{5,1,2}$ and use a similar argument as $A_{5,1,1}$ to get
	\begin{align*}
		A_{5,1,2}
		&\leq\frac{\sigma_0}{20}\Delta t\mathbb{E}\Big[\|\nabla\times \Ee^{n+1}\|^2_{L^2(D)^3}\Big]\\
		&\quad+(C\Delta t+\frac{\sigma_0}{8})\Delta te^{-2\sigma_0\Delta t}\mathbb{E}\Big[\|\nabla\cdot \Ee^n\|^2_{L^2(D)}\Big]+C\Delta t.
	\end{align*}
	Combining the estimates of $A_{5,1,1}$ and $A_{5,1,2}$, it holds that
	\begin{align*}
		A_{5,1}\leq&\,\frac{\sigma_0}{20}\Delta t\mathbb{E}\Big[\|\nabla\times \Ee^{n+1}\|^2_{L^2(D)^3}\Big]+(C\Delta t+\frac{\sigma_0}{8})\Delta te^{-2\sigma_0\Delta t}\mathbb{E}\Big[\|\nabla\cdot \Ee^n\|^2_{L^2(D)}\Big]\\
		&+\frac{\sigma_0}{16}\Delta t\mathbb{E}\Big[\|\nabla\cdot \Hh^{n+1}\|_{L^2(D)}^2\Big]+\frac{\sigma_0}{16}\Delta te^{-2\sigma_0\Delta t}\mathbb{E}\Big[\|\nabla\cdot \Hh^n\|^2_{L^2(D)}\Big]+C\Delta t.
	\end{align*}
	Similar to the estimate of $A_{5,1,1}$, for the term $A_{5,2}$ we obtain
	\begin{align*}
		A_{5,2}&=-\frac{\lambda_{1}}{2}\mathbb{E}\Big[\langle e^{-\sigma\Delta t}\Hh^n\cdot\nabla(\Delta \overline{W}_1^n),\nabla\cdot \Ee^{n+1}-e^{-\sigma\Delta t}\nabla\cdot \Ee^n\rangle_{L^2(D)}\Big]\\
		&\quad-\frac{\lambda_{1}}{2}\mathbb{E}\Big[\langle\Delta\overline{W}_1^n(\nabla e^{-\sigma\Delta t})\cdot \Hh^n,\nabla\cdot \Ee^{n+1}\rangle_{L^2(D)}\Big]\\
		&\leq\frac{\sigma_0}{16}\Delta t\mathbb{E}\Big[\|\nabla\cdot \Hh^{n+1}||_{L^2(D)}^2\Big]+\frac{\sigma_0}{16}\Delta te^{-2\sigma_0\Delta t}\mathbb{E}\Big[\|\nabla\cdot \Hh^n\|_{L^2(D)}^2\Big]+C\Delta t.
	\end{align*}
	Since terms $A_{6,1}$ and $A_{6,2}$ can be similarly estimated as terms $A_{5,1}$ and $A_{5,2}$ respectively, 
	one gets
	\begin{align*}
		\mathbb{E}[A_5+A_6]&\leq\frac{\sigma_0}{20}\Delta t\mathbb{E}\Big[\|Mu^{n+1}\|^2_{\H}\Big]\\
		&\quad+(C\Delta t+\frac{\sigma_0}{8})\Delta te^{-2\sigma_0\Delta t}\mathbb{E}\Big[\|\nabla\cdot \Ee^n\|^2_{L^2(D)}+\|\nabla\cdot \Hh^n\|^2_{L^2(D)}\Big]\\
		&\quad+\frac{\sigma_0}{8}\Delta t\mathbb{E}\Big[\|\nabla\cdot \Ee^{n+1}\|_{L^2(D)}^2+\|\nabla\cdot \Hh^{n+1}\|_{L^2(D)}^2\Big].
	\end{align*}
	
	Combining the estimates of $A_1$--$A_6$, we arrive at
	\begin{equation*}\label{divp=1}
		\begin{split}
			&\mathbb{E}\Big[\|\nabla\cdot \Hh^{n+1}\|^2_{L^2(D)}+\|\nabla\cdot \Ee^{n+1}\|^2_{L^2(D)}\Big]\\
			&\leq e^{-2\sigma_0\Delta t}\mathbb{E}\Big[\|\nabla\cdot \Ee^n\|^2_{L^2(D)}+\|\nabla\cdot \Hh^n\|^2_{L^2(D)}\Big]\\
			&\quad+\frac{\sigma_0}{4}\Delta t\mathbb{E}\Big[\|\nabla\cdot \Ee^{n+1}\|^2_{L^2(D)}+\|\nabla\cdot \Hh^{n+1}\|^2_{L^2(D)}\Big]\\
			&\quad+(C\Delta t+\frac{3\sigma_0}{16})\Delta te^{-2\sigma_0\Delta t}\mathbb{E}\Big[\|\nabla\cdot \Ee^n\|^2_{L^2(D)}+\|\nabla\cdot \Hh^n\|^2_{L^2(D)}\Big]\\
			&\quad+\frac{\sigma_0}{20}\Delta t\mathbb{E}\Big[\|Mu^{n+1}\|_{\H}^2\Big]+C\Delta t.
		\end{split}
	\end{equation*}
	
	{\em Step 2. Estimate of} $\mathbb{E}\Big[\|Mu^{n}\|^2_{\H}\Big]$.
	
	We apply $\langle\cdot,M(u^{n+1}-e^{-\sigma\Delta t}u^n)\rangle_{\H}$ on both sides of \eqref{ns} and get
	\begin{equation}\label{ytu}
		\begin{split}
			&\| M(u^{n+1}\|_{\H}^2-\|M(e^{-\sigma\Delta t}u^n)\|^2_{\H}\\
			&=-\frac{\lambda_{1}}{\Delta t}\langle J(u^{n+1}+e^{-\sigma\Delta t}u^n)\Delta\overline{W}_1^n,M(u^{n+1}-e^{-\sigma\Delta t}u^n)\rangle_{\H}
			\\
			&\quad-\frac{2}{\Delta t}\langle\widetilde{{\bm{\lambda}}}_2\Delta W_2^n,M(u^{n+1}-e^{-\sigma\Delta t}u^n)\rangle_{\H}.
		\end{split}
	\end{equation}
	By the skew-adjointness of the Maxwell operator $M$, we substitute \eqref{ns} into the two terms on the right side of \eqref{ytu} and take the expectation to obtain
	\begin{align*}
		\mathbb{E}\Big[\|Mu^{n+1}\|^2_{\H}\Big]=\,&\mathbb{E}\Big[\|M(e^{-\sigma\Delta t}u^n)\|_{\H}^2\Big]+\mathbb{E}\Big[\langle M(\widetilde{{\bm{\lambda}}}_2\Delta W_2^n),M(u^{n+1}+e^{-\sigma\Delta t}u^n)\rangle_{\H}\Big]\notag\\
		&+\frac{\lambda_{1}}{2}\mathbb{E}\Big[\langle M(J(u^{n+1}+e^{-\sigma\Delta t}u^n)\Delta\overline{W}_1^n),M(u^{n+1}+e^{-\sigma\Delta t}u^n)\rangle_{\H}\Big]\notag\\
		=&:B_1+B_2+B_3.
	\end{align*}
	For the term $B_1$, we note that
	$M(e^{-\sigma\Delta t}u^n)=e^{-\sigma\Delta t}Mu^n+R_{\sigma}^{\Delta t}u^n,$
	where
	$$R^{\Delta t}_\sigma:=\left(          
	\begin{array}{ccc}  
		0&(\nabla e^{-\sigma\Delta t})\times\\
		-(\nabla e^{-\sigma\Delta t})\times&0\\
	\end{array}
	\right).$$ 
	Since
	\begin{align*}
		\|R_\sigma^{\Delta t} u^n\|_{\H}\leq 2\Delta te^{-\sigma_0\Delta t}\|\sigma\|_{W^{1,\infty}(D)}\|u^n\|_{\H}\leq C\Delta te^{-\sigma_0\Delta t}\|u^n\|_{\H},
	\end{align*}
	we derive that
	\begin{align*}
		B_1
		&=\mathbb{E}\Big[\|e^{-\sigma\Delta t}Mu^n\|^2_{\H}\Big]+\mathbb{E}\Big[\|R_\sigma^{\Delta t} u^n\|_{\H}^2\Big]+2\mathbb{E}\Big[\langle e^{-\sigma\Delta t}Mu^n,R_\sigma u^n\rangle_{\H}\Big]\\
		&\leq  e^{-2\sigma_0\Delta t}\mathbb{E}\Big[\|Mu^n\|_{\H}^2\Big]+\frac{\sigma_0}{16}\Delta te^{-2\sigma_0\Delta t}\mathbb{E}\Big[\|Mu^n\|^2_{\H}\Big]+C\Delta t,
	\end{align*}
	where we use the Young inequality and Proposition \ref{uniformH}.
	
	Using the fact that $\Delta W_2^n$ is independent of $\mathcal{F}_{t_n}$, the skew-adjointness of $M$ and substituting \eqref{ns} into the term $B_2$, we get
	\begin{align*}
		B_2&=\mathbb{E}\Big[\langle M(\widetilde{{\bm{\lambda}}}_2\Delta W_2^n),M(u^{n+1}-e^{-\sigma\Delta t}u^n)\rangle_{\H}\Big]\\
		&=-\frac{\Delta t}{2}\mathbb{E}\Big[\langle M^2(\widetilde{{\bm{\lambda}}}_2\Delta W_2^n),Mu^{n+1}\rangle_{\H}\Big]+\mathbb{E}\Big[\|M(\widetilde{{\bm{\lambda}}}_2\Delta W_2^n)\|^2_{\H}\Big]\\
		&\quad-\frac{\lambda_{1}}{2}\mathbb{E}\Big[\langle M^2(\widetilde{{\bm{\lambda}}}_2\Delta W_2^n),Ju^{n+1}\Delta\overline{W}_1^n\rangle_{\H}\Big]\\
		&\leq\frac{\sigma_0}{20}\Delta t\mathbb{E}\Big[\|Mu^{n+1}\|^2_{\H}\Big]+C\Delta t
	\end{align*}
	due the Young inequality, Sobolev embedding $H^{\gamma}(D)\hookrightarrow L^{\infty}(D)$ for $\gamma>\frac{3}{2}$ and Proposition \ref{uniformH}.
	
	Let $H_W^1:=\left(          
	\begin{array}{ccc}  
		(\nabla(\Delta \overline{W}_1^n))\times&0\\
		0&(\nabla(\Delta \overline{W}_1^n))\times\\
	\end{array}
	\right)$. For the term $B_3$, we have
	\begin{align*}
		B_3
		&=\frac{\lambda_{1}}{2}\mathbb{E}\Big[\langle J\Delta\overline{W}_1^nM(u^{n+1}+e^{-\sigma\Delta t}u^n),M(u^{n+1}+e^{-\sigma\Delta t}u^n)\rangle_{\H}\Big]\\
		&\hphantom{=}+\frac{\lambda_{1}}{2}\mathbb{E}\Big[\langle H_W^1(u^{n+1}+e^{-\sigma\Delta t}u^n),M(u^{n+1}+e^{-\sigma\Delta t}u^n)\rangle_{\H}\Big]\\
		&=\frac{\lambda_{1}}{2}\mathbb{E}\Big[\langle H_W^1(u^{n+1}-e^{-\sigma\Delta t}u^n), M(u^{n+1}+e^{-\sigma\Delta t}u^n)\rangle_{\H}\Big]\\
		&\hphantom{==}+\lambda_{1}\mathbb{E}\Big[\langle H_W^1e^{-\sigma\Delta t}u^n,M(u^{n+1}-e^{-\sigma\Delta t}u^n)\rangle_{\H}\Big]\\
		&=:B_{3,1}+B_{3,2}.
	\end{align*}
	For the term $B_{3,1}$, we substitute \eqref{ns} into it and obtain
	\begin{align*}
		B_{3,1}=\,&\frac{\lambda_{1}\Delta t}{4}\mathbb{E}\Big[\langle H_W^1M(u^{n+1}+e^{-\sigma\Delta t}u^n),M(u^{n+1}+e^{-\sigma\Delta t}u^n)\rangle_{\H}\Big]\\
		&+\frac{\lambda_{1}^2}{4}\mathbb{E}\Big[\langle H_W^1J(u^{n+1}+e^{-\sigma\Delta t}u^n)\Delta\overline{W}_1^n,M(u^{n+1}+e^{-\sigma\Delta t}u^n)\rangle_{\H}\Big]\\
		&+\frac{\lambda_{1}}{2}\mathbb{E}\Big[\langle H_W^1\widetilde{{\bm{\lambda}}}_2\Delta W_2^n,M(u^{n+1}+e^{-\sigma\Delta t}u^n)\rangle_{\H}\Big]\\
		=&:B_{3,1,1}+B_{3,1,2}+B_{3,1,3}.
	\end{align*}
	It follows from the Sobolev embedding $H^{\gamma}(D)\hookrightarrow L^{\infty}(D)$ for $\gamma>\frac{3}{2}$, the Young inequality and Proposition \ref{uniformH} that 
	\begin{align*}
		B_{3,1,1}
		&\leq C\Delta t\mathbb{E}\Big[\|\Delta\overline{W}_1^n\|_{H^{1+\gamma}(D)}\|Mu^{n+1}\|_{\H}^2\Big]\\
		&\quad+C\Delta t\mathbb{E}\Big[\|\Delta\overline{W}_1^n\|_{H^{1+\gamma}(D)}\|Mu^{n+1}\|_{\H}\Big(e^{-\sigma_0\Delta t}\|Mu^n\|_{\H}+C\Delta te^{-\sigma_0\Delta t}\|u^n\|_{\H}\Big)\Big].
	\end{align*}
	The condition $\sum_{i=1}^{\infty}\|Q_1^{\frac{1}{2}}e_i\|_{H^{1+\gamma}(D)}<\infty$ implies $\|\Delta \overline{W}_1^n\|\leq C\Delta t^{\frac{1}{2}}A_{\Delta t}$, which yields
	\begin{align*}
		B_{3,1,1}&\leq C\Delta t^{\frac{3}{2}}A_{\Delta t}\mathbb{E}\Big[\|Mu^{n+1}||^2_{\H}\Big]+\frac{\sigma_0}{20}\Delta t\mathbb{E}\Big[\|Mu^{n+1}||^2_{\H}\Big]\\
		&\quad+C\Delta te^{-2\sigma_0\Delta t}\mathbb{E}\Big[\|\Delta \overline{W}_1^n\|^2_{H^{1+\gamma}(D)}\|Mu^n\|^2_{\H}\Big]\\
		&\quad+C\Delta t^3\mathbb{E}\Big[\|\Delta \overline{W}^n_1\|^2_{H^{1+\gamma}(D)}\|u^n\|^2_{\H}\Big]\\
		&\leq (C\Delta t^{\frac{1}{2}}A_{\Delta t}+\frac{\sigma_0}{20})\Delta t\mathbb{E}\Big[\|Mu^{n+1}\|^2_{\H}\Big]+C\Delta t^2e^{-2\sigma_0\Delta t}\mathbb{E}\Big[\|Mu^{n}\|_{\H}^2\Big]+C\Delta t^4.
	\end{align*}
	For terms $B_{3,1,2}$ and $B_{3,1,3}$, we have
	$$B_{3,1,2}+B_{3,1,3}\leq \frac{\sigma_0}{20}\Delta t\mathbb{E}[\|Mu^{n+1}\|_{\H}^2]+\frac{\sigma_0}{16}\Delta te^{-2\sigma_0\Delta t}\mathbb{E}[\|Mu^n\|^2_{\H}]+C\Delta t.$$
	Hence, 
	\begin{align*}
		B_{3,1}&\leq(C\Delta t^{\frac{1}{2}}A_{\Delta t}+\frac{\sigma_0}{10})\Delta t\mathbb{E}\Big[\|Mu^{n+1}\|^2_{\H}\Big]\\
		&\quad+(C\Delta t+\frac{\sigma_0}{16})\Delta te^{-2\sigma_0\Delta t}\mathbb{E}\Big[\|Mu^n\|_{\H}^2\Big]+C\Delta t.
	\end{align*}
	
	For the term $B_{3,2}$, we use the skew-adjointness of $M$ and then substitute \eqref{ns} into it to obtain
	\begin{align*}
		B_{3,2}&=-\lambda_{1}\mathbb{E}\Big[\langle e^{-\sigma\Delta t} M(H_W^1u^n),u^{n+1}-e^{-\sigma\Delta t}u^n\rangle_{\H}\Big]-\lambda_{1}\mathbb{E}\Big[\langle  R_\sigma^{\Delta t}(H_W^1u^n),u^{n+1}\rangle_{\H}\Big].\\
		&=-\frac{\lambda_{1}\Delta t}{2}\mathbb{E}\Big[\langle e^{-\sigma\Delta t} M(H_W^1u^n),Mu^{n+1}\rangle_{\H}\Big]\\
		&\quad-\frac{\lambda_{1}^2}{2}\mathbb{E}\Big[\langle e^{-\sigma\Delta t} M(H_W^1u^n),J(u^{n+1}+e^{-\sigma\Delta t}u^n)\Delta\overline{W}_1^n\rangle_{\H}\Big]\\
		&\quad-\lambda_{1}\mathbb{E}\Big[\langle  R_\sigma^{\Delta t}(H_W^1u^n),u^{n+1}\rangle_{\H}\Big].
	\end{align*}
	Notice that
	\begin{align*}
		M(H_W^1u^n)=&\left(          
		\begin{array}{ccc}  
			(\nabla\cdot \Hh^n)\nabla(\Delta\overline{W}_1^n)\\
			-(\nabla\cdot \Ee^n)\nabla(\Delta\overline{W}_1^n)\\
		\end{array}
		\right)
		+\left(          
		\begin{array}{ccc}  
			(\Hh^n\cdot\nabla)\nabla(\Delta\overline{W}_1^n)\\
			-(\Ee^n\cdot\nabla)\nabla(\Delta\overline{W}_1^n)\\
		\end{array}
		\right)\\
		&+(\nabla\cdot(\nabla(\Delta\overline{W}_1^n)))Ju^n
		+ (\nabla(\Delta\overline{W}_1^n)\cdot\nabla)Ju^n,
	\end{align*}
	which leads to
	\begin{equation*}
		\|M(H_W^1u^n)\|^2_{\H}
		\leq C\|\Delta\overline{W}_1^n\|^2_{H^{2+\gamma}(D)}\Big(\|\nabla\cdot \Hh^n\|^2_{L^2(D)}+\|\nabla\cdot \Ee^n\|^2_{L^2(D)}+\|u^n\|^2_{\H}+\|Mu^n\|_{\H}^2\Big)
	\end{equation*}
	due to the Sobolev embedding $H^{\gamma}(D)\hookrightarrow L^{\infty}(D)$ for $\gamma>\frac{3}{2}$ and the Young inequality.
	Thus, one gets
	\begin{align*}
		B_{3,2}
		&\leq\big(C\Delta t^{\frac{1}{2}}+\frac{\sigma_0}{16}\big)\Delta te^{-2\sigma_0\Delta t}\mathbb{E}\Big[\|\nabla\cdot \Ee^n\|^2_{L^2(D)}+\|\nabla\cdot \Hh^n\|^2_{L^2(D)}\Big]\\
		&\quad+\frac{\sigma_0}{20}\Delta t\mathbb{E}\Big[\|Mu^{n+1}\|^2_{\H}\Big]
		+\big(C\Delta t+\frac{\sigma_0}{16}\big)\Delta te^{-2\sigma_0\Delta t}\mathbb{E}\Big[\|Mu^n\|^2_{\H}\Big]+C\Delta t.
	\end{align*}
	Combining the estimates of $B_{3,1}$ and $B_{3,2}$, it yields
	\begin{align*}
		B_3\leq \,&\big(C\Delta t^{\frac{1}{2}}A_{\Delta t}+\frac{3\sigma_0}{20}\big)\Delta t\mathbb{E}\Big[\|Mu^{n+1}\|^2_{\H}\Big]+\big(C\Delta t+\frac{\sigma_0}{8}\big)\Delta te^{-2\sigma_0\Delta t}\mathbb{E}\Big[\|Mu^n\|^2_{\H}\Big]\\
		&+\big(C\Delta t^{\frac{1}{2}}+\frac{\sigma_0}{16}\big)\Delta te^{-2\sigma_0\Delta t}\mathbb{E}\Big[\|\nabla\cdot \Ee^n\|^2_{L^2(D)}+\|\nabla\cdot \Hh^n\|^2_{L^2(D)}\Big]+C\Delta t.
	\end{align*}
	Putting all estimates of $B_1$--$B_3$ together, we arrive at
	\begin{equation*}
		\begin{split}
			\mathbb{E}\big[\|Mu^{n+1}\|_{\H}^2\big]\leq\,&e^{-2\sigma_0\Delta t}\mathbb{E}\big[\|Mu^{n}\|^2_{\H}\big]+\big(\frac{3\sigma_0}{16}+C\Delta t\big)\Delta te^{-2\sigma_0\Delta t}\mathbb{E}\big[\|Mu^n\|^2_{\H}\big]\\
			&+\big(C\Delta t^{\frac{1}{2}}A_{\Delta t}+\frac{\sigma_0}{5}\big)\Delta t\mathbb{E}\big[\|Mu^{n+1}\|^2_{\H}\big]+C\Delta t\\
			&+\big(C\Delta t^{\frac{1}{2}}+\frac{\sigma_0}{16}\big)\Delta te^{-2\sigma_0\Delta t}\mathbb{E}\big[\|\nabla\cdot \Ee^n\|^2_{L^2(D)}+\|\nabla\cdot \Hh^n\|^2_{L^2(D)}\big].
		\end{split}
	\end{equation*}
	
	{\em Step 3. Proof of \eqref{div}.}
	
	Combining {\em Step 1} and {\em Step 2}, it holds that
	\begin{align*}
		&\mathbb{E}\Big[\|Mu^{n+1}\|^2_{\H}+\|\nabla\cdot \Hh^{n+1}\|^2_{L^2(D)}+\|\nabla\cdot \Ee^{n+1}\|^2_{L^2(D)}\Big]\\
		&\leq e^{-2\sigma_0\Delta t}\mathbb{E}\Big[\|Mu^n\|^2_{\H}+\|\nabla\cdot \Hh^n\|^2_{L^2(D)}+\|\nabla\cdot \Ee^n\|^2_{L^2(D)}\Big]\\
		&\hphantom{=}+\big(C\Delta t^{\frac{1}{2}}+\frac{\sigma_0}{4}\big)\Delta te^{-2\sigma_0\Delta t}\mathbb{E}\Big[\|Mu^n\|^2_{\H}+\|\nabla\cdot \Hh^n\|^2_{L^2(D)}+\|\nabla\cdot \Ee^n\|^2_{L^2(D)}\Big]\\
		&\hphantom{=}+\big(C\Delta t^{\frac{1}{2}}A_{\Delta t}+\frac{\sigma_0}{4}\big)\Delta t\mathbb{E}\Big[\|Mu^{n+1}\|^2_{\H}+\|\nabla\cdot \Hh^{n+1}\|^2_{L^2(D)}+\|\nabla\cdot \Ee^{n+1}\|^2_{L^2(D)}\Big]\\
		&\hphantom{=}+C\Delta t.
	\end{align*}
	Since $\lim_{\Delta t\rightarrow0}\Delta t^{\frac{1}{2}}A_{\Delta t}=0$, there exists a $\Delta t^*>0$ such that for all $\Delta t\in(0,\Delta t^*]$, we have
	$$C\big(\Delta t^{\frac{1}{2}}+\Delta t^{\frac{1}{2}}A_{\Delta t}(1+\sigma_0\Delta t)\big)+\frac{\sigma_0^2}{4}\Delta t\leq \frac{\sigma_0}{2},$$
	which implies that
	$$\frac{1+(C\Delta t^{\frac{1}{2}}+\frac{\sigma_0}{4})\Delta t}{1-(C\Delta t^{\frac{1}{2}}A_{\Delta t}+\frac{\sigma_0}{4})\Delta t}\leq 1+\sigma_0\Delta t\leq e^{\sigma_0\Delta t}.$$ 
	By the Gronwall inequality, we finish the proof.
\end{proof}

For the fourth moment estimates of the curl and divergence of the numerical solution $u^n$, by multiplying $\|\nabla\cdot \Ee^{n+1}\|^2_{L^2(D)}+\|\nabla\cdot \Hh^{n+1}\|^2_{L^2(D)}$ and $\|Mu^{n+1}\|^2_{\H}$ on both sides of \eqref{uioo} and \eqref{ytu}, respectively, we can derive the following result. The proof is similar to that of Lemma \ref{uniformdivM}.
\begin{coro}\label{uniformdivMp2}
	Set $\gamma>\frac{3}{2}$. Let Assumption \ref{ass} hold with $\gamma_1:=2+\gamma$ and $\gamma_2\geq 2$ and let $\sum_{i=1}^{\infty}\|Q_1^{\frac{1}{2}}e_i\|_{H^{1+\gamma}(D)}$
	$<\infty$, and $u_0\in L^4(\Omega,H^1(D)^6)$.   
	There exist positive constants $C_7$ and $C_8$ such that for sufficiently small $\Delta t>0$,
	\begin{equation*}
		\mathbb{E}\big[\|u^n\|^4_{H^1(D)^6}\big]\leq C_7e^{-3\sigma_0t}\mathbb{E}\big[\|u_0\|^4_{H^1(D)^6}\big]+C_8,
	\end{equation*}	
	where $C_7$ depends on $|D|$, and $C_8$ depends on $\sigma_0,C_{4,1},C_{4,2},\lambda_{1},\widetilde{{\bm{\lambda}}}_2,\mathbb{E}\big[\|u_0\|^4_{\H}\big],\|Q_1^{\frac{1}{2}}\|_{\mathcal{L}_2^{\gamma_1}}$ and $\|Q_2^{\frac{1}{2}}\|_{\mathcal{L}_2^{\gamma_2}}$.
\end{coro}

Let $\{P_n\}_{n\in\mathbb{N}}$ be the Markov transition semigroup associated to the numerical solution $\{u^n\}_{n\in\mathbb{N}}$.	
Based on Lemma \ref{uniformdivM}, the ergodicity of $\{u^n\}_{n\in\mathbb{N}}$ and the convergence of $P^*_n\pi$ to the numerical invariant measure in the $L^2$-Wasserstein distance are similar to that of Proposition \ref{SMEErgo}, which are stated below.

\begin{theo}\label{hbf}
	Under the conditions in Lemma \ref{uniformdivM}, the following statements hold.\\
	{\rm(i)}
	The numerical solution $\{u^n\}_{n\in\mathbb{N}}$ of \eqref{ns} has a unique invariant measure $\pi^{\Delta t}\in\mathcal{P}_2(\H)$ for sufficiently small $\Delta t>0$. Thus $\{u^n\}_{n\in\mathbb{N}}$ is ergodic. Moreover, $\{u^n\}_{n\in\mathbb{N}}$ is exponentially mixing.\\
	{\rm (ii)} For any distribution $\pi\in\mathcal{P}_2(\H)$, 
	\begin{equation*}
		\mathcal{W}_2(P^*_n\pi,\pi^{\Delta t})\leq e^{-\sigma_0t_n}\mathcal{W}_2(\pi,\pi^{\Delta t}).
	\end{equation*}
\end{theo}
\begin{proof}
	Notice that $u^{n}-\widetilde{u}^{n}$ solves
	\begin{align*}
		(u^{n}-\widetilde{u}^{n})-e^{-\sigma\Delta t}(u^{n-1}-\widetilde{u}^{n-1})=&\Delta tM\frac{(u^{n}-\widetilde{u}^{n})+e^{-\sigma\Delta t}(u^{n-1}-\widetilde{u}^{n-1})}{2}\notag\\
		&+\lambda_{1}J\frac{(u^{n}-\widetilde{u}^{n})+e^{-\sigma\Delta t}(u^{n-1}-\widetilde{u}^{n-1})}{2}\Delta\overline{W}^{n-1}_1.
	\end{align*}	
	We apply $\langle\cdot,(u^{n}-\widetilde{u}^{n})+e^{-\sigma\Delta t}(u^{n-1}-\widetilde{u}^{n-1})\rangle_{\H}$ to both sides of the above equation and take expectation to get
	\begin{equation*}\label{ditui1}
		\mathbb{E}\big[\|u^{n}-\widetilde{u}^{n}\|_{\H}^2\big]\leq e^{-2\sigma_0\Delta t}\mathbb{E}\big[\|u^{n-1}-\widetilde{u}^{n-1}\|_{\H}^2\big]\leq \cdots\leq e^{-2\sigma_0 t_{n}}\mathbb{E}\big[\|u_0-\widetilde{u}_0\|_{\H}^2\big],
	\end{equation*}
	by which assertions (i) and (ii) can be obtained similarly to Proposition \ref{SMEErgo}. Thus we finish the proof.
\end{proof}

\subsubsection{Stochastic conformal multi-symplecticity}
Now we turn to the stochastic conformal multi-symplecticity of the temporal semi-discretization $\eqref{ns}$. Let
\begin{equation*}
	\delta_t^{\sigma}u^n=\frac{u^{n+1}-e^{-\sigma\Delta t}u^n}{\Delta t},\quad A_t^{\sigma}u^n=\frac{u^{n+1}+e^{-\sigma\Delta t}u^n}{2}.
\end{equation*}
Then \eqref{ns} can be transformed into the following compact form
\begin{align*}
	&F\delta_t^\sigma u^n+K_1\partial_x A_t^{\sigma}u^n+K_2\partial_y A_t^{\sigma}u^n+K_3\partial_z A_t^{\sigma}u^n\\
	&=\nabla_uS_1(A_t^{\sigma}u^n)\frac{\Delta\overline{W}_1^n}{\Delta t}+\nabla_uS_2(A_t^\sigma u^n)\frac{\Delta W_2^n}{\Delta t}.
\end{align*}

Similar to \cite[Theorem 2.2]{jiang2013stochastic}, we can obtain the following result.
\begin{prop}
	The temporal semi-discretization \eqref{ns} possesses the stochastic conformal multi-symplectic conservation law
	\begin{equation*}
		\delta_t^{2\sigma}\omega^n+\partial_x\kappa_1^n+\partial_y\kappa_2^n+\partial_z\kappa_3^n=0\qquad\mathbb{P}\text{-}a.s.,
	\end{equation*}
	that is,
	\begin{align*}
		\int_{z_0}^{z_1}&\int_{y_0}^{y_1}\int_{x_0}^{x_1}\omega^{n+1}(x,y,z)\dd x\dd y\dd z-\int_{z_0}^{z_1}\int_{y_0}^{y_1}\int_{x_0}^{x_1}e^{-2\sigma\Delta t}\omega^{n}(x,y,z)\dd x\dd y\dd z\\
		&+\Delta t\int_{z_0}^{z_1}\int_{y_0}^{y_1}\kappa_1^{n}(x_1,y,z)\dd y\dd z-\Delta t\int_{z_0}^{z_1}\int_{y_0}^{y_1}\kappa_1^n(x_0,y,z)\dd y\dd z\\
		&+\Delta t\int_{z_0}^{z_1}\int_{x_0}^{x_1}\kappa_2^{n}(x,y_1,z)\dd x\dd z-\Delta t\int_{z_0}^{z_1}\int_{x_0}^{x_1}\kappa_2^n(x,y_0,z)\dd x\dd z\\
		&+\Delta t\int_{y_0}^{y_1}\int_{x_0}^{x_1}\kappa_3^{n}(x,y,z_1)\dd x\dd y-\Delta t\int_{y_0}^{y_1}\int_{x_0}^{x_1}\kappa_3^n(x,y,z_0)\dd x\dd y=0,
	\end{align*}
	where $\omega^n=\frac{1}{2}du^n\wedge Fdu^n,\kappa_s^n=\frac{1}{2}(dA^\sigma_tu^n)\wedge K_sd(A^\sigma_tu^n),$ $s=1,2,3$ are differential 2-forms associated with the skew-symmetric matrices $F$ and $K_s$, respectively,  and
	$(x_{0},x_{1})\times(y_{0},y_{1})\times(z_{0},z_{1})$
	is the local domain of $u^n(x,y,z)$.
\end{prop}

\subsection{Error analysis of the temporal semi-discretization}
In this section, we study the mean-square convergence order of the temporal semi-discretization \eqref{ns}. To this end,
we rewrite \eqref{ns} into the following form
\begin{equation}\label{semi}
	u^{n+1}=\hat{S}_{\Delta t}u^n+\lambda_{1}T_{\Delta t}JA_t^{\sigma}u^n\Delta\overline{W}_1^n+T_{\Delta t}\widetilde{{\bm{\lambda}}}_2\Delta W_2^n,\quad \hphantom{=}n\geq0,
\end{equation}
where $$\hat{S}_{\Delta t}=\Big(I-\frac{\Delta t}{2}M\Big)^{-1}\Big(I+\frac{\Delta t}{2}M\Big)e^{-\sigma\Delta t},\quad T_{\Delta t}=\Big(I-\frac{\Delta t}{2}M\Big)^{-1}.$$

In order to derive the error estimate of \eqref{ns}, we need to introduce the following lemma, whose proof is given in Appendix A.

\begin{lemm}\label{opee}
	Let $\hat{S}(t):=e^{t(M-\sigma I)}$, $t\geq0$. There exist positive constants $C$ such that\\
	{\rm (1)} $\|\hat{S}(t)\|_{\mathcal{L}(\H,\H)}\leq e^{-\sigma_0t}$ and $\|(\hat{S}_{\Delta t})^n\|_{\mathcal{L}(\H,\H)}\leq e^{-\sigma_0n\Delta t}$.\\
	{\rm (2)} $\|\hat{S}(t_n)-(\hat{S}_{\Delta t})^n\|_{\mathcal{L}(D(M),\mathbb{H})}\leq Ce^{-\sigma_0n\Delta t/2}\Delta t^{1/2}$.\\
	{\rm (3)} $\|\hat{S}(t)-I\|_{\mathcal{L}(D(M),\H)}\leq Ct\quad\forall t\geq0$ and $\|\hat{S}(t)-e^{-\sigma\Delta t}\|_{\mathcal{L}(D(M),\H)}\leq C\Delta t\quad\forall\, t\in[0,\Delta t].$\\
	{\rm (4)} $\|T_{\Delta t}\|_{\mathcal{L}(\H,\H)}\leq 1$ and $\|I-T_{\Delta t}\|_{\mathcal{L}(\mathcal{D}(M),\H)}\leq C\Delta t$.\\
	{\rm (5)} $\|\hat{S}(t_n-r)-(\hat{S}_{\Delta t})^{n-k-1}T_{\Delta t}\|_{\mathcal{L}(D(M),\H)}\leq Ce^{-\sigma_0(n-k-1)\Delta t/2}\Delta t^{1/2}\quad\forall\,r\in[t_{k},t_{k+1}]$, $k=0,\ldots,n-1.$
\end{lemm}

The error of \eqref{ns} is estimated in the following theorem.
\begin{theo}\label{order}
	Under the conditions in Corollary \ref{uniformdivMp2}, if $F_{Q_1}\in W^{1,\infty}(D)$, there exists a positive constant $C_9$ such that for sufficiently small $\Delta t>0$,
	\begin{equation*}
		\max_{n\geq 0}\big(\mathbb{E}\big[\|u(t_n)-u^n\|_{\H}^2\big]\big)^{1/2}\leq C_9\Delta t^{1/2},
	\end{equation*}
	where $C_9$ depends on $\sigma_0,\lambda_{1},\widetilde{{\bm{\lambda}}}_2,C_1$--$C_3,C_{4,1},C_{4,2},C_{5}$--$C_{8},\mathbb{E}\big[\|u_0\|^4_{H^1(D)^6}\big]$,
	$\|Q_1^{\frac{1}{2}}\|_{\mathcal{L}_2^{\gamma_1}}$, $\|Q_2^{\frac{1}{2}}\|_{\mathcal{L}_2^{\gamma_2}}$,
	$\|\sigma\|_{W^{1,\infty}(D)}$ and $\|F_{Q_1}\|_{W^{1,\infty}(D)}$.
\end{theo}
\begin{proof}
	For $n\in\mathbb{N}$, we note that the mild solution of \eqref{SMEs with PML} is
	\begin{equation}\label{ewe}
		\begin{split}
			u(t_n)=&\hat{S}(t_n)u_0-\frac{1}{2}\lambda_{1}^2\int_{0}^{t_n}\hat{S}(t_n-s)F_{Q_1}u(s)\dd s\\
			&\quad+\lambda_{1}\int_{0}^{t_n}\hat{S}(t_n-s)Ju(r)\dd W_1(s)+\int_{0}^{t_n}\hat{S}(t_n-s)\widetilde{{\bm{\lambda}}}_2\dd W_2(s).
		\end{split}
	\end{equation}
	From \eqref{semi}, we have
	\begin{equation}\label{erw}
		u^n=\big(\hat{S}_{\Delta t}\big)^nu_0+\lambda_{1}\sum_{k=0}^{n-1}\big(\hat{S}_{\Delta t}\big)^{n-k-1}T_{\Delta t}JA_t^{\sigma}u^k\Delta\overline{W}_1^k+\sum_{k=0}^{n-1}\big(\hat{S}_{\Delta t}\big)^{n-1-k}T_{\Delta t}\widetilde{{\bm{\lambda}}}_2\Delta W_2^k.
	\end{equation}
	Let $e_n:=u(t_n)-u^n$. Subtracting \eqref{erw} from \eqref{ewe}, it yields
	\begin{align*}
		e_n=:I+II+III
	\end{align*}
	with
	\begin{align*}
		I=\,&\big[\hat{S}(t_n)-(\hat{S}_{\Delta t})^n\big]u_0,\\
		II=&\int_{0}^{t_n}\hat{S}(t_n-r)\widetilde{{\bm{\lambda}}}_2\dd W_2(r)-\sum_{k=0}^{n-1}\big(\hat{S}_{\Delta t}\big)^{n-1-k}T_{\Delta t}\widetilde{{\bm{\lambda}}}_2\Delta W_2^k,\\
		III=\,&\lambda_{1}\int_{0}^{t_n}\hat{S}(t_n-r)Ju(r)\dd W_1(r)-\frac{1}{2}\lambda_{1}^2\int_{0}^{t_n}\hat{S}(t_n-r)F_{Q_1}u(r)\dd r\\
		&\quad-\lambda_{1}\sum_{k=0}^{n-1}\big(\hat{S}_{\Delta t}\big)^{n-1-k}T_{\Delta t}JA_t^{\sigma}u^{k}\Delta \overline{W}_1^k.
	\end{align*}
	
	By Lemma \ref{opee}(2), we have $$\sup_{n\geq 0}\mathbb{E}\big[\|I\|^2_{\H}\big]\leq C\Delta t.$$
	
	Using the It\^o isometry and Lemma \ref{opee}(5), it holds that
	\begin{align*}
		\mathbb{E}\big[\|II\|_{\H}^2\big]&=\mathbb{E}\bigg[\Big\|\sum_{k=0}^{n-1}\int_{t_k}^{t_{k+1}}(\hat{S}(t_n-r)-(\hat{S}_{\Delta t})^{n-k-1}T_{\Delta t})\widetilde{{\bm{\lambda}}}_2\dd W_2(r)\Big\|_{\H}^2\bigg]\\
		&=\sum_{k=0}^{n-1}\int_{t_k}^{t_{k+1}}\big\|(\hat{S}(t_n-r)-(\hat{S}_{\Delta t})^{n-k-1}T_{\Delta t})\widetilde{{\bm{\lambda}}}_2 Q_2^{\frac{1}{2}}\big\|_{HS(L^2(D),\H)}^2\dd r
		\\
		&\leq \sum_{k=0}^{n-1}Ce^{-\sigma_0(n-k-1)\Delta t}\Delta t^2=C\Delta t,
	\end{align*}	
	where in the last step we use the fact $\frac{1-e^{-\sigma_0n\Delta t}}{1-e^{-\sigma_0\Delta t}}\leq \frac{1}{\sigma_0\Delta t}$.
	
	For the third term $III$, we substitute $\eqref{ns}$ into it and obtain
	\begin{align*}
		III=&\sum_{k=0}^{n-1}\Big[\lambda_{1}\int_{t_k}^{t_{k+1}}\hat{S}(t_n-r)Ju(r)\dd W_1(r)-\frac{1}{2}\lambda_{1}^2\int_{t_k}^{t_{k+1}}\hat{S}(t_n-r)F_{Q_1}u(r)\dd r\\
		&\qquad-\lambda_{1}(\hat{S}_{\Delta t})^{n-1-k}T_{\Delta t}JA_t^{\sigma}u^k\Delta \overline{W}_1^k\Big]=:\sum_{i=1}^{4}\sum_{k=0}^{n-1}A_{i,k},
	\end{align*}	
	where
	\begin{align*}
		A_{1,k}&=\frac{\lambda_{1}^2}{2}\Big[\big(\hat{S}_{\Delta t}\big)^{n-1-k}T_{\Delta t}e^{-\sigma\Delta t}u^k(\Delta W_1^k)^2-\int_{t_k}^{t_{k+1}}\hat{S}(t_n-r)F_{Q_1}u(r)\dd r\Big],\\
		A_{2,k}&=\lambda_{1}\big(\hat{S}_{\Delta t}\big)^{n-1-k}T_{\Delta t}Je^{-\sigma\Delta t}u^k\big(\Delta W_1^k-\Delta\overline{W}_1^k\big)\\
		&\quad-\frac{\lambda_{1}\Delta t}{4}\big(\hat{S}_{\Delta t}\big)^{n-1-k}T_{\Delta t}JM(u^{k+1}+e^{-\sigma\Delta t}u^k)\Delta\overline{W}_1^k\\
		&\quad+\frac{\lambda_{1}^2}{4}\big(\hat{S}_{\Delta t}\big)^{n-1-k}T_{\Delta t}\big(u^{k+1}-e^{-\sigma\Delta t}u^k\big)(\Delta\overline{W}_1^k)^2\\
		&\quad+\frac{\lambda_{1}^2}{2}\big(\hat{S}_{\Delta t}\big)^{n-1-k}T_{\Delta t}e^{-\sigma\Delta t}u^k\big[(\Delta\overline{W}_1^k)^2-(\Delta W_1^k)^2\big],\\
		A_{3,k}&=\lambda_{1}\int_{t_k}^{t_{k+1}}\Big(\hat{S}(t_n-r)Ju(r)-\big(\hat{S}_{\Delta t}\big)^{n-1-k}T_{\Delta t}Je^{-\sigma\Delta t}u^k\Big)\dd W_1(r),\\
		A_{4,k}&=-\frac{\lambda_{1}}{2}\big(\hat{S}_{\Delta t}\big)^{n-1-k}T_{\Delta t}J\widetilde{{\bm{\lambda}}}_2\Delta\overline{W}_1^k\Delta W_2^k.
	\end{align*}
	
	(i) {\em Estimate of the term $\sum_{k=0}^{n-1}A_{1,k}$.}
	
	Notice that
	\begin{align*}
		A_{1,k}&=\lambda_{1}^2\int_{t_k}^{t_{k+1}}\int_{t_k}^{r}\big(\hat{S}_{\Delta t}\big)^{n-1-k}T_{\Delta t}e^{-\sigma\Delta t}u^k\dd W_1(\rho)\dd W_1(r)\\
		&\qquad\quad+\frac{\lambda_{1}^2}{2}\Big[\big(\hat{S}_{\Delta t}\big)^{n-1-k}T_{\Delta t}F_{Q_1}e^{-\sigma\Delta t}u^k\Delta t-\int_{t_k}^{t_{k+1}}\hat{S}(t_n-r)F_{Q_1}u(r)\dd r\Big]\\
		&=\lambda_{1}^2\int_{t_k}^{t_{k+1}}\int_{t_k}^{r}\big(\hat{S}_{\Delta t}\big)^{n-1-k}T_{\Delta t}e^{-\sigma\Delta t}u^k\dd W_1(\rho)\dd W_1(r)\\
		&\quad+\frac{\lambda_{1}^2}{2}\int_{t_k}^{t_{k+1}}\Big[\big(\hat{S}_{\Delta t}\big)^{n-1-k}T_{\Delta t}-\hat{S}(t_n-r)\Big](F_{Q_1}e^{-\sigma\Delta t}u^k)\dd r\\
		&\quad-\frac{\lambda_{1}^2}{2}\int_{t_k}^{t_{k+1}}\hat{S}(t_n-r)\big(F_{Q_1}e^{-\sigma\Delta t}e_k\big)\dd r\\
		&\quad-\frac{\lambda_1^2}{2}\int_{t_k}^{t_{k+1}}\hat{S}(t_n-r)F_{Q_1}\big(u(r)-e^{-\sigma\Delta t}u(t_k)\big)\dd r\\
		&=:A_{1,k,1}+A_{1,k,2}+A_{1,k,3}+A_{1,k,4}.
	\end{align*}
	
	For the term $A_{1,k,1}$, it follows from the It\^o isometry, Sobolev embedding $H^{\gamma}(D)\hookrightarrow L^{\infty}(D)$ for $\gamma>\frac{3}{2}$, Lemma \ref{opee}(1)(4) and Proposition \ref{uniformH} that 
	\begin{align*}
		\mathbb{E}\big[\|A_{1,k,1}\|_{\H}^2\big]
		&\leq C\int_{t_k}^{t_{k+1}}\int_{t_k}^{r}\mathbb{E}\big[\|(\hat{S}_{\Delta t})^{n-k-1}T_{\Delta t}e^{-\sigma\Delta t}u^k Q_1^{\frac{1}{2}}\|^2_{HS(L^2(D),\H)}\big]\dd\rho\dd r\\
		&\leq C\int_{t_k}^{t_{k+1}}\int_{t_k}^{r}\mathbb{E}\big[\|(\hat{S}_{\Delta t})^{n-k-1}T_{\Delta t}e^{-\sigma\Delta t}u^k\|^2_{\H}\big]\dd\rho\dd r\\
		&\leq Ce^{-2\sigma_0(n-k)\Delta t}\Delta t^2,
	\end{align*}
	which implies 
	$$\mathbb{E}\Big[\Big\|\sum_{k=0}^{n-1}A_{1,k,1}\Big\|_{\H}^2\Big]=\sum_{k=0}^{n-1}\mathbb{E}\Big[\Big\|A_{1,k,1}\Big\|_{\H}^2\Big]\leq C\Delta t.$$
	
	For the term $A_{1,k,2}$, we use Proposition \ref{uniformH} and Lemma \ref{opee}(5) to obtain
	\begin{align*}
		&\mathbb{E}\big[\|A_{1,k,2}\|^2_{\H}\big]\\
		&\leq C\Delta t\int_{t_k}^{t_{k+1}}\|(\hat{S}_{\Delta t})^{n-k-1}T_{\Delta t}-\hat{S}(t_n-r)\|^2_{\mathcal{L}(\mathcal{D}(M),\H)}\mathbb{E}\big[\|F_{Q_1}e^{-\sigma\Delta t}u^k||^2_{\mathcal{D}(M)}\big]\dd r\\
		&\leq Ce^{-\sigma_0(n-k)}\Delta t^3,
	\end{align*}
	which leads to
	\begin{align*}
		&\mathbb{E}\Big[\Big\|\sum_{k=0}^{n-1}A_{1,k,2}\Big\|^2_{\H}\Big]=\sum_{i,j=0}^{n-1}\mathbb{E}\big[\langle A_{1,i,2},A_{1,j,2}\rangle_{\H}\big]\\
		&\leq \sum_{i,j=0}^{n-1}\Big(\mathbb{E}\big[\|A_{1,i,2}\|_{\H}^2\big]\Big)^{\frac{1}{2}}\Big(\mathbb{E}\big[\|A_{1,j,2}\|_{\H}^2\big]\Big)^{\frac{1}{2}}
		=\Bigg(\sum_{k=0}^{n-1}\Big(\mathbb{E}\big[\|A_{1,k,2}\|_{\H}^2\big]\Big)^{\frac{1}{2}}\Bigg)^2\\
		&=C\bigg(\Delta t\sum_{k=0}^{n-1}e^{-\sigma_0(n-k)\Delta t/2}\bigg)^2\Delta t\leq C\Delta t.
	\end{align*}

	For the term $A_{1,k,3}$, the H\"older inequality and Lemma \ref{opee}(1) imply
	\begin{align*}
		\mathbb{E}\big[\|A_{1,k,3}\|^2_{\H}\big]&\leq C\Delta t\int_{t_k}^{t_{k+1}}\mathbb{E}\big[\big\|\hat{S}(t_n-r)\big(F_{Q_1}e^{-\sigma\Delta t}e_k\big)\big\|^2_{\H}\big]\dd r\\
		&\leq Ce^{-2\sigma_0(n-k)\Delta t}\Delta t^2\mathbb{E}\big[\|e_k\|^2_{\H}\big],
	\end{align*}
	by which we get
	\begin{align*}
		&\mathbb{E}\Big[\Big\|\sum_{k=0}^{n-1}A_{1,k,3}\Big\|_{\H}^2\Big]\leq \Bigg(\sum_{k=0}^{n-1}\Big(\mathbb{E}\big[\|A_{1,k,3}\|_{\H}^2\big]\Big)^{\frac{1}{2}}\Bigg)^2\\
		&\leq \Bigg(\sum_{k=0}^{n-1}Ce^{-\sigma_0(n-k)\Delta t}\Delta t\Big(\mathbb{E}\big[\|e_k\|_{\H}^2\big]\Big)^{\frac{1}{2}}\Bigg)^2\\
		&\leq C\Bigg(\sum_{k=0}^{n-1}\Big(\Delta t^{\frac{1}{2}}e^{-\sigma_0(n-k)\Delta t/2}\Big)\Big(\Delta t^{\frac{1}{2}}e^{-\sigma_0(n-k)\Delta t/2}\big(\mathbb{E}\big[\|e_k\|_{\H}^2\big]\big)^{\frac{1}{2}}\Big)\Bigg)^2\\
		&\leq C\Big(\sum_{k=0}^{n-1}\Delta te^{-\sigma_0(n-k)\Delta t}\Big)\Big(\sum_{k=0}^{n-1}\Delta te^{-\sigma_0(n-k)\Delta t}\mathbb{E}\big[\|e_k\|^2_{\H}\big]\Big)\\
		&\leq C\Delta t\sum_{k=0}^{n-1}e^{-\sigma_0(n-k)\Delta t}\mathbb{E}\big[\|e_k\|^2_{\H}\big].
	\end{align*}
	
	For $r\in[t_k,t_{k+1}]$, by \eqref{ewe} we have
	\begin{equation}\label{bnb}
		\begin{split}
			&\mathbb{E}\big[\|u(r)-e^{-\sigma\Delta t}u(t_k)\|^2_{\H}\big]\\
			&\leq C\big\|\hat{S}(r-t_k)-e^{-\sigma\Delta t}\big\|^2_{\mathcal{L}(D(M),\H)}\mathbb{E}\big[\|u(t_k)\|^2_{D(M)}\big]\\
			&\quad+C\Delta t\int_{t_k}^{r}\mathbb{E}\big[\|\hat{S}(r-\rho)F_{Q_1}u(\rho)\|^2_{\H}\big]\dd\rho\\
			&\quad+C\int_{t_k}^{r}\mathbb{E}\big[\|\hat{S}(r-\rho)Ju(\rho) Q_1^{\frac{1}{2}}\|^2_{HS(L^2(D),\H)}\big]\dd\rho\\
			&\quad+C\int_{t_k}^{r}\mathbb{E}\big[\|\hat{S}(r-\rho)\widetilde{{\bm{\lambda}}}_2 Q_2^{\frac{1}{2}}\|^2_{HS(L^2(D),\H)}\big]\dd\rho\\
			&\leq C\Delta t,
		\end{split}
	\end{equation}
	where we use the It\^o isometry, Lemma \ref{opee}(1)(3) and Proposition \ref{H-boundedness}.
	Therefore, for the term $A_{1,k,4}$, one obtains
	\begin{align*}
		\mathbb{E}\big[\|A_{1,k,4}\|^2_{\H}\big]&\leq C\Delta te^{-2\sigma_0(n-k)\Delta t}\int_{t_k}^{t_{k+1}}\mathbb{E}\big[\|u(r)-e^{-\sigma\Delta t}u(t_k)\|^2_{\H}\big]\dd r\\
		&\leq Ce^{-2\sigma_0(n-k)\Delta t}\Delta t^3,
	\end{align*}	
	from which we have
	\begin{align*}
		&\mathbb{E}\Big[\Big\|\sum_{k=1}^{n-1}A_{1,k,4}\Big\|^2_{\H}\Big]\leq C\Delta t.
	\end{align*}
	Combining the above estimates of $A_{1,k,1}$--$A_{1,k,4}$, it holds that
	\begin{align*}
		\mathbb{E}\Big[\Big\|\sum_{k=0}^{n-1}A_{1,k}\Big\|^2_{\H}\Big]
		\leq C\Delta t+C\Delta t\sum_{k=0}^{n-1}e^{-\sigma_0(n-k)\Delta t}\mathbb{E}\big[\|e_k\|_{\H}^2\big].
	\end{align*}
	
	(ii) {\em Estimate of the term $\sum_{k=0}^{n-1}A_{2,k}$.}
	
	Substituting \eqref{ns} into the third term of $A_{2,k}$ and Lemma \ref{opee}(1)(4), we have
	\begin{align*}
		\mathbb{E}\big[\|A_{2,k}\|^2_{\H}\big]
		&\leq Ce^{-2\sigma_0(n-k)\Delta t}\mathbb{E}\big[\|u^k\|^2_{\H}\|\Delta W_1^k-\Delta\overline{W}_1^k\|^2_{H^{\gamma}(D)}\big]\\
		&\hphantom{=}+C\Delta t^2e^{-2\sigma_0(n-k)\Delta t}\mathbb{E}\big[\|M(u^{k+1}+e^{-\sigma\Delta t}u^k)\|^2_{\H}\|\Delta\overline{W}_1^k\|^2_{H^{\gamma}(D)}\big]\\
		&\hphantom{=}+Ce^{-2\sigma_0(n-k)\Delta t}\mathbb{E}\big[\|(u^{k+1}+e^{-\sigma\Delta t}u^k)(\Delta\overline{W}_1^k)^3\|_{\H}^2\big]\\
		&\hphantom{=}+Ce^{-2\sigma_0(n-k)\Delta t}\mathbb{E}\big[\|\widetilde{{\bm{\lambda}}}_2\Delta W_2^k(\Delta\overline{W}_1^k)^2\|^2_{\H}\big]\\
		&\hphantom{=}+Ce^{-2\sigma(n-k)\Delta t}\mathbb{E}\big[\|u^k\|_{\H}^2\|(\Delta\overline{W}_1^k)^2-(\Delta W_1^k)^2\|_{H^{\gamma}(D)}^2\big],
	\end{align*}
	which along with Proposition \ref{uniformH}, Lemma \ref{uniformdivM}, Corollary \ref{uniformdivMp2} and \eqref{wfzq}, leads to
	$$\mathbb{E}\Big[\Big\|\sum_{k=0}^{n-1}A_{2,k}\Big\|^2_{\H}\Big]\leq C\Delta t.$$

	(iii) {\em Estimate of the term $\sum_{k=0}^{n-1}A_{3,k}$.}
	
	We use the  It\^o isometry, $\eqref{bnb}$, Lemma \ref{regularity in H1} and Lemma \ref{opee}(1)(4)(5) to obtain
	\begin{align*}
		&\mathbb{E}\big[\|A_{3,k}\|^2_{\H}\big]\\
		&=\lambda_{1}^2\int_{t_{k}}^{t_{k+1}}\mathbb{E}\big[\big\|\big(\hat{S}(t_n-r)Ju(r)-\big(\hat{S}_{\Delta t}\big)^{n-1-k}T_{\Delta t}Je^{-\sigma\Delta t}u^k\big) Q_1^{\frac{1}{2}}\big\|^2_{HS(L^2(D),\H)}\big]\dd r\\
		&\leq C\int_{t_k}^{t_{k+1}}\mathbb{E}\big[\big\|\hat{S}(t_n-r)Ju(r)-\big(\hat{S}_{\Delta t}\big)^{n-k-1}T_{\Delta t}Ju(r)\big\|^2_{\H}\big]\dd r\\
		&\quad+C\int_{t_{k}}^{t_{k+1}}\mathbb{E}\big[\big\|\big(\hat{S}_{\Delta t}\big)^{n-k-1}T_{\Delta t}Ju(r)-\big(\hat{S}_{\Delta t}\big)^{n-k-1}T_{\Delta t}Je^{-\sigma\Delta t}u^k\big\|_{\H}^2\big]\dd r\\
		&\leq C\int_{t_k}^{t_{k+1}}\big\|\hat{S}(t_n-r)-\big(\hat{S}_{\Delta t}\big)^{n-k-1}T_{\Delta t}\big\|^2_{\mathcal{L}(\mathcal{D}(M),\H)}\mathbb{E}\big[\|u(r)\|_{\mathcal{D}(M)}^2\big]\dd r\\
		&\quad+C\int_{t_k}^{t_{k+1}}\big\|(\hat{S}_{\Delta t})^{n-1-k}T_{\Delta t}\big\|_{\mathcal{L}(\H,\H)}^2\mathbb{E}\big[\big\|u(r)-e^{-\sigma\Delta t}u(t_k)\big\|^2+\big\|e^{-\sigma\Delta t}e_k\big\|^2_{\H}\big]\dd r\\
		&\leq Ce^{-\sigma_0(n-k)\Delta t}\Delta t\Big(\Delta t+\mathbb{E}\big[\|e_k\|^2\big]\Big),
	\end{align*}
	by which 
	\begin{align*}
		\mathbb{E}\bigg[\Big\|\sum_{k=0}^{n-1}A_{3,k}\Big\|_{\H}^2\bigg]=\sum_{k=0}^{n-1}\mathbb{E}\big[\|A_{3,k}\|^2_{\H}\big]\leq C\Delta t+C\Delta t\sum_{k=0}^{n-1}e^{-\sigma_0(n-k)\Delta t}\mathbb{E}\big[\|e_k\|^2_{\H}\big].
	\end{align*}
	
	(iv) {\em Estimate of the term $\sum_{k=0}^{n-1}A_{4,k}$.}
	
	By the independence of increments of $W_1$ and $W_2$, we have 
	\begin{align*}
		&\mathbb{E}\Big[\Big\|\sum_{k=0}^{n-1}A_{4,k}\Big\|_{\H}^2\Big]=\sum_{k=0}^{n-1}\mathbb{E}\big[\|A_{4,k}\|_{\H}^2\big]\\
		&\leq C\sum_{k=0}^{n-1}e^{-\sigma_0(n-k-1)}\Delta t\mathbb{E}\big[\|\widetilde{{\bm{\lambda}}}_2\Delta W_2^k\Delta\overline{W}_1^k\|^2_{\H}\big]\leq C\Delta t.
	\end{align*}
	
	Combining (i)--(iv), it yields
	$$\mathbb{E}\big[\|III\|_{\H}^2\big]\leq C\Delta t+C\Delta t\sum_{k=0}^{n-1}e^{-\sigma_0(n-k)\Delta t}\mathbb{E}\big[\|e_k\|_{\H}^2\big].$$
	Altogether, we conclude that
	$$\mathbb{E}\big[\|e_n\|_{\H}^2\big]\leq C\Delta t\exp\Big(\sum_{k=0}^{n-1}\Delta te^{-\sigma_0(n-k)\Delta t}\Big)\leq C\Delta t$$
	due to the Gronwall inequality.
	Thus we finish the proof.
\end{proof}	
Similar to Proposition \ref{SMEErgo}(ii), the error of the invariant measure between the exact solution and the numerical solution in $L^2$-Wasserstein distance can be  estimated via Theorem \ref{order}.
\begin{coro}
	Under the conditions in Lemma \ref{regularity in H1} and Theorem \ref{order}, there exists a postive constant $C_{10}$ such that
	\begin{equation*}
		\mathcal{W}_2(\pi^{*},\pi^{\Delta t})\leq C_{10}\Delta t^{\frac{1}{2}}.
	\end{equation*}
\end{coro}
\begin{proof}
	By Proposition \ref{SMEErgo} and Theorem \ref{order}, we have
	\begin{align*}
		\mathcal{W}_2(\pi^*,\pi^{\Delta t})&\leq\mathcal{W}_2(P^*_{n}\pi^{\Delta t},P^*_{t_n}\pi^{\Delta t})+ \mathcal{W}_2(P^*_{t_n}\pi^{\Delta t},P^*_{t_n}\pi^{*})\\
		&\leq C_9\Delta t^{\frac{1}{2}}+e^{-\sigma_0t_n}\mathcal{W}_2(\pi^*,\pi^{\Delta t}),
	\end{align*}
	which leads to the assertion by letting $n\rightarrow\infty$.
\end{proof}

\section{Ergodic full discretizations}
This section focuses on the study of ergodic full discretizations for \eqref{SMEs with PML} which are based on the further discretizations on the temporal semi-discretization by the dG method and the finite difference method in the spatial direction, respectively. 

\subsection{Ergodic dG full discretization}
We first apply the dG method to discretize \eqref{ns} in space and obtain a dG full discretization for \eqref{SMEs with PML}. For the dG full discretization, we prove the ergodicity based on the analysis of the uniform boundedness of the numerical solution in $L^2(\Omega,\H)$. Moreover, the mean-square convergence order of the dG full discretization in both temporal and spatial directions is shown. As a result, the $L^2$-Wasserstein distance between the numerical invariant measure and the exact one is estimated.

To this end, we introduce some basic notations and properties of the dG method. Let $\mathcal{T}_h=\{K\}$ be a simplicial, shape- and contact-regular mesh of the domain $D$ consisting of elements $K$, i.e., $D=\bigcup K$. The index $h$ refers to the maximum diameter of all elements of $\mathcal{T}_h$. We denote the restriction of a function $v$ to an element $K$ by $v_K:=v|_K$. The dG space with respect to the mesh $\mathcal{T}_h$ is taken to be the set of piecewise linear functions, i.e., $\H_h:=\{v_h\in L^2(D):v_h|_K\in\mathbb{P}_1(K)\}^6$, where $\mathbb{P}_1(K)$ denotes the set of continuous piecewise polynomials of degree up to $1$. The set of faces is denoted by $\mathcal{G}_h=\mathcal{G}_h^{\text{int}}\cup\mathcal{G}_h^{\text{ext}}$, where $\mathcal{G}_h^{\text{int}}$ and $\mathcal{G}_h^{\text{ext}}$ consist of all interior and all exterior faces, repsectively. We denote the unit normal of a face $F\in\mathcal{G}_h^{\text{int}}$ by $\mathbf{n}_F$, where the orientation of $\mathbf{n}_F$ is fixed once and forever for each interior face. For a face $F\in \mathcal{G}_h^{\text{ext}}$, $\mathbf{n}_F$ is the outward normal vector. Jumps of $v_h$ on an interior face $F$ with normal vector $\mathbf{n}_F$ pointing from $K$ to $K_F$ are defined as 
$[[v_h]]_F:=(v_{K_F})|_F-(v_K)|_F.$
Note that the sign of the jump on face $F$ is fixed by the direction of the normal vector $\mathbf{n}_F$. Define the broken Sobolev spaces by $$H^k(\mathcal{T}_h):=\{v\in L^2(D):v_K\in H^k(K)\text{ for all }K\in\mathcal{T}_h\},\quad k\in\mathbb{N},$$ 
with seminorm and norm being $|v|^2_{H^k(\mathcal{T}_h)}:=\sum_{K\in\mathcal{T}_h}|v|^2_{H^k(K)}$ and $\|v\|^2_{H^k(\mathcal{T}_h)}:=\sum_{j=0}^{k}|v|^2_{H^j(\mathcal{T}_h)}$, respectively.  Let $\pi_h:\H\rightarrow\H_h$ be the $L^2$-orthogonal projection operator on $\H_h$, where the projection acts componentwise for vector fields. For the projection operator $\pi_h$, we have
\begin{align}\label{touying}
	\|\pi_hv\|_{\H}\leq \|v\|_{\H}\qquad \forall v\in\H
\end{align}
and 
\begin{align}\label{gza}
	\langle v-\pi_hv,u_h\rangle_{\H}=0\qquad \forall u_h\in \H_h.
\end{align}
Moreover, for all $v\in H^1(\mathcal{T}_h)^6$, it holds that
\begin{align}\label{touyingwucha}
	\|v-\pi_hv\|_{\H}\leq Ch|v|_{H^1(\mathcal{T}_h)^6}
\end{align}
and 
\begin{align}\label{wdx}
	\sum_{F\in\mathcal{G}_h}\|v-\pi_hv\|^2_{L^2(F)^6}\leq Ch|v|^2_{H^1(\mathcal{T}_h)^6}.
\end{align}
For the parameter $\sigma$, we give the following assumption in this subsection.
\begin{assumption}
	Suppose that $\sigma\equiv\sigma_0$ is constant on $D$.
\end{assumption}


Now we are in the position to propose the following dG full discretization:
\begin{align}\label{dgfull}
	u^{n+1}_h-e^{-\sigma_0\Delta t}u^n_h&=\frac{\Delta t}{2}\big(M_hu^{n+1}_h+e^{-\sigma_0\Delta t}M_hu^n_h\big)\notag\\
	&\quad+\frac{\lambda_1}{2}\pi_h\Big[J\big(u^{n+1}_h+e^{-\sigma_0\Delta t}u^n_h\big)\Delta\overline{W}_1^n\Big]+\pi_h(\widetilde{{\bm{\lambda}}}_2\Delta W_2^n),
\end{align}
where the discrete Maxwell operator $M_h:\H_h\rightarrow\H_h$ is given as
\begin{align*}
	&\langle M_hu_h,v_h\rangle_{\H}:=\sum_{K\in\mathcal{T}_h}\Big(\langle\nabla\times\mathbf{H}_h,\psi_h\rangle_{L^2(K)^3}-\langle\nabla\times\mathbf{E}_h,\phi_h\rangle_{L^2(K)^3}\Big)\\
	&+\frac{1}{2}\sum_{F\in\mathcal{G}_h^{\text{int}}}\Big(\langle\mathbf{n}_F\times[[\mathbf{H}_h]]_F,\psi_K+\psi_{K_F}\rangle_{L^2(F)^3}-\langle\mathbf{n}_F\times[[\mathbf{E}_h]]_F,\phi_K+\phi_{K_F}\rangle_{L^2(F)^3}\\
	&\qquad\qquad\qquad-\langle\mathbf{n}_F\times[[\mathbf{E}_h]]_F,\mathbf{n}_F\times[[\psi_h]]_F\rangle_{L^2(F)^3}\\
	&\qquad\qquad\qquad-\langle\mathbf{n}_F\times[[\mathbf{H}_h]]_F,\mathbf{n}_F\times[[\phi_h]]_F\rangle_{L^2(F)^3}\Big)\\
	&+\sum_{F\in\mathcal{G}_h^{\text{ext}}}\Big(\langle\mathbf{n}_F\times\mathbf{E}_h,\phi_h\rangle_{L^2(F)^3}-\langle\mathbf{n}_F\times\mathbf{E}_h,\mathbf{n}_F\times\psi_h\rangle_{L^2(F)^3}\Big)
\end{align*}
with $u_h=(\mathbf{E}_h^\top,\mathbf{H}_h^{\top})^{\top},v_h=(\psi_h^{\top},\phi_h^{\top})^{\top}\in \H_h.$
\begin{prop}\label{efs}
	Let Assumption \ref{ass} hold with $\gamma_1\geq 1$ and $\gamma_2\geq 1$, 
	and let $Q_2^{\frac{1}{2}}\in HS(L^2(D),H_0^{\gamma_2}(D))$
	and $u_0\in L^2(\Omega,\H)$. There exists a constant $C_{11}$ independent of time and $h$ such that for sufficiently small $\Delta t>0$,
	$$\sup_{n\geq0}\mathbb{E}\big[\|u^n_h\|_{\H}^2\big]\leq C_{11},$$
	where the positive constant $C_{11}$ depends on $\sigma_0,\|Q_1^{\frac{1}{2}}\|_{\mathcal{L}_2^{\gamma_1}}, \|Q_2^{\frac{1}{2}}\|_{\mathcal{L}_2^{\gamma_2}},\mathbb{E}\big[\|u_0\|_{\H}^2\big],\lambda_1$ and $\widetilde{{\bm{\lambda}}}_2$.
\end{prop}
\begin{proof}
	We apply $\langle\cdot,u^{n+1}_h+e^{-\sigma_0\Delta t}u^n_h\rangle_{\H}$ on both sides of \eqref{dgfull} and take the expectation to get
	\begin{equation}\label{ygb}
		\begin{split}
			&\mathbb{E}\big[\|u^{n+1}_h\|^2_{\H}\big]-e^{-2\sigma_0\Delta t}\mathbb{E}\big[\|u^n_h\|^2_{\H}\big]\\
			&=\frac{\Delta t}{2}\mathbb{E}\big[\langle M_h(u^{n+1}_h+e^{-\sigma_0\Delta t}u^n_h),u^{n+1}_h+e^{-\sigma_0\Delta t}u^n_h\rangle_{\H}\big]\\
			&\quad+\frac{\lambda_{1}}{2}\mathbb{E}\big[\langle \pi_h(J(u^{n+1}_h+e^{-\sigma_0\Delta t}u^n_h)\Delta\overline{W}_1^n),u^{n+1}_h+e^{-\sigma_0\Delta t}u^n_h\rangle_{\H}\big]\\
			&\quad+\mathbb{E}\big[\langle \pi_h(\widetilde{{\bm{\lambda}}}_2\Delta W_2^n),u^{n+1}_h+e^{-\sigma_0\Delta t}u^n_h\rangle_{\H}\big]\\
			&\leq \mathbb{E}\big[\langle \pi_h(\widetilde{{\bm{\lambda}}}_2\Delta W_2^n),u^{n+1}_h+e^{-\sigma_0\Delta t}u^n_h\rangle_{\H}\big],
		\end{split}
	\end{equation}
	where in the last step we use the dissipative property of $M_h$ (see e.g., \cite[Proposition 4.4(ii)]{CCC2021}) and \eqref{gza}. By using the fact that $\Delta W_2^n$ is independent of $\mathcal{F}_{t_n}$, we substitute \eqref{dgfull} into \eqref{ygb} to obtain
	\begin{align*}
		&\mathbb{E}\big[\|u^{n+1}_h\|^2_{\H}\big]-e^{-2\sigma_0\Delta t}\mathbb{E}\big[\|u^n_h\|^2_{\H}\big]\leq\mathbb{E}\big[\langle \widetilde{{\bm{\lambda}}}_2\Delta W_2^n,u^{n+1}_h-e^{-\sigma_0\Delta t}u^n_h\rangle_{\H}\big]\\
		&=\frac{\Delta t}{2}\mathbb{E}\big[\langle \widetilde{{\bm{\lambda}}}_2\Delta W_2^n,M_hu^{n+1}_h\rangle_{\H}\big]+\mathbb{E}\big[\langle \widetilde{{\bm{\lambda}}}_2\Delta W_2^n,\frac{\lambda_{1}}{2}\pi_h(Ju^{n+1}_h\Delta\overline{W}_1^n)+\pi_h(\widetilde{{\bm{\lambda}}}_2\Delta W_2^n)\rangle_{\H}\big]\\
		&=:A_{1}+A_{2}.
	\end{align*}
	For the term $A_1$,we use \cite[Lemma A.4]{IM2015} to obtain
	\begin{align*}
		A_1&=\frac{\Delta t}{2}\mathbb{E}\big[\langle M_hu_h^{n+1},\pi_h(\widetilde{\bm{\lambda}}_2\Delta W_2^n)\rangle_{\H}\big]
		\leq \frac{\sigma_0}{8}\Delta t\mathbb{E}\big[\|u^{n+1}_h\|^2_{\H}\big]+C\Delta t^2.
	\end{align*}
	For the term $A_2$, we have
	\begin{align*}
		A_2
		&\leq C\mathbb{E}\Big[\|\Delta W_2^n\|_{L^4(D)}\|\Delta\overline{W}_1^n\|_{L^4(D)}\|u_h^{n+1}\|_{\H}\Big]+\mathbb{E}\big[\|\widetilde{{\bm{\lambda}}}_2\Delta W_2^n\|^2_{\H}\big]\\
		&\leq \frac{\sigma_0}{8}\Delta t\mathbb{E}\big[\|u^{n+1}\|_{\H}^2\big]+C\Delta t
	\end{align*}
	in view of \eqref{touying} and the Sobolev embedding $L^{4}(D)\hookrightarrow H^{1}(D)$.
	
	Combining $A_1$ and $A_2$, we have
	\begin{align*}
		\mathbb{E}\big[\|u^{n+1}_h\|^2_{\H}\big]\leq e^{-2\sigma_0\Delta t}\mathbb{E}\big[\|u^n_h\|_{\H}^2\big]+\frac{\sigma_0}{4}\Delta t\mathbb{E}\big[\|u^{n+1}_h\|_{\H}^2\big]+C\Delta t.
	\end{align*}
	Then the Gronwall inequality implies the assertion for any $\Delta t\in(0,\frac{3}{\sigma_0}]$.
\end{proof}

We denote by $\{P_{h}^n\}_{n\in\mathbb{N}}$ the Markov transition semigroup associated to the numerical solution $\{u^{n}_h\}_{n\in\mathbb{N}}$.	
The following proposition gives the ergodicity of $\{u^n_h\}_{n\in\mathbb{N}}$ and the convergence of $(P_h^n)^*\pi$ towards the numerical invariant measure in the $L^2$-Wasserstein distance.

\begin{theo}
	Under the conditions in Proposition \ref{efs}, the following statements hold.\\
	{\rm (i)}
	The numerical solution $\{u^n_h\}_{n\in\mathbb{N}}$ of \eqref{dgfull} has a unique invariant measure $\pi^{\Delta t,h}\in\mathcal{P}_2(\H)$ for sufficiently small $\Delta t>0$. Thus $\{u^n_h\}_{n\in\mathbb{N}}$ is ergodic. Moreover, $\{u^n_h\}_{n\in\mathbb{N}}$ is exponentially mixing.\\
	{\rm (ii)} For any distribution $\pi\in\mathcal{P}_2(\H)$, 
	\begin{equation*}
		\mathcal{W}_2((P_h^n)^*\pi,\pi^{\Delta t,h})\leq e^{-\sigma_0t_n}\mathcal{W}_2(\pi,\pi^{\Delta t,h}).
	\end{equation*}
\end{theo}
\begin{proof}
	The proof is similar to that of Theorem \ref{hbf}. The main difference lies in the proof of  the continuous dependence of the solution on the initial data.
	Since $u^{n}-\widetilde{u}^{n}$ solves
	\begin{align*}
		&(u^{n}_h-\widetilde{u}_h^{n})-e^{-\sigma_0\Delta t}(u_h^{n-1}-\widetilde{u}_h^{n-1})\notag\\
		&=\Delta tM_h\frac{(u_h^{n}-\widetilde{u}_h^{n})+e^{-\sigma_0\Delta t}(u_h^{n-1}-\widetilde{u}_h^{n-1})}{2}\notag\\
		&\quad+\lambda_{1}\pi_h\Big[J\frac{(u_h^{n}-\widetilde{u}_h^{n})+e^{-\sigma_0\Delta t}(u_h^{n-1}-\widetilde{u}_h^{n-1})}{2}\Delta\overline{W}^{n-1}_1\Big],
	\end{align*}	
	we apply $\langle\cdot,(u_h^{n}-\widetilde{u}_h^{n})+e^{-\sigma_0\Delta t}(u_h^{n-1}-\widetilde{u}_h^{n-1})\rangle_{\H}$ and take the expectation to get
	\begin{equation*}
		\mathbb{E}\big[\|u_h^{n}-\widetilde{u}_h^{n}\|_{\H}^2\big]\leq e^{-2\sigma_0\Delta t}\mathbb{E}\big[\|u_h^{n-1}-\widetilde{u}_h^{n-1}\|_{\H}^2\big]\leq \cdots\leq e^{-2\sigma_0 t_{n}}\mathbb{E}\big[\|u^0_h-\widetilde{u}^0_h\|_{\H}^2\big]
	\end{equation*}
	due to the dissipative property of $M_h$ given in \cite[Proposition 4.4(ii)]{CCC2021}.
\end{proof}

\subsubsection{Error analysis of the dG full discretization}
In this subsection, we aim to give the mean-square convergence analysis of the dG full discretization \eqref{dgfull}. Let $\widehat{S}_{h,\Delta t}:=(I-\frac{\Delta t}{2}M_h)^{-1}(I+\frac{\Delta t}{2}M_h)e^{-\sigma_0\Delta t}$, $T_{h,\Delta t}:=(I-\frac{\Delta t}{2}M_h)^{-1}\pi_h$, then the dG full discretization can be rewritten as
\begin{align*}
	u^{n}_h&=(\widehat{S}_{h,\Delta t})^nu^0_h+\lambda_{1}\sum_{k=0}^{n-1}(\widehat{S}_{h,\Delta t})^{n-k-1}T_{h,\Delta t}(JA_t^{\sigma_0}u^k_h\Delta\overline{W}_1^k)\\
	&\quad+\sum_{k=0}^{n-1}(\widehat{S}_{h,\Delta t})^{n-k-1}T_{h,\Delta t}\widetilde{{\bm{\lambda}}}_2\Delta W_2^k
\end{align*}
for $n\in\mathbb{N}_+$.

In order to estimate the mean-square error of \eqref{dgfull}, we need to introduce the following lemma, whose proof is given in Appendix B.
\begin{lemm}\label{ooq}
	Let $v=(v_1^{\top},v_2^{\top})^{\top}\in H^1(D)^6$ with $\mathbf{n}\times v_1|_{\partial D}=0$ and $\mathbf{n}\cdot v_2|_{\partial D}=0$. There exist positive constants $C$ such that\\
	{\rm (1)} $\|\widehat{S}_{h,\Delta t}\|_{\mathcal{L}(\H_h,\H_h)}\leq e^{-\sigma_0\Delta t}$ \text{and} $\|T_{h,\Delta t}\|_{\mathcal{L}(\H,\H_h)}\leq 1.$\\
	{\rm (2)}
	$\|(\widehat{S}_{\Delta t})^nv\|_{H^1(D)^6}\leq Ce^{-\sigma_0n\Delta t/2}\|v\|_{H^1(D)^6}.$\\
	{\rm (3)} $\|(\pi_h(\widehat{S}_{\Delta t})^n-(\widehat{S}_{h,\Delta t})^n\pi_h)v\|_{\H}\leq Ch^{\frac{1}{2}}e^{-\sigma_0n\Delta t/2}\|v\|_{H^1(D)^6}$.\\
	{\rm (4)} $\|(\widehat{S}(t_n)-(\widehat{S}_{h,\Delta t})^n\pi_h)v\|_{\H}\leq Ce^{-\sigma_0n\Delta t/2}(\Delta t^{\frac{1}{2}}+h^{\frac{1}{2}})\|v\|_{H^1(D)^6}$.\\
	{\rm (5)} $\|(\widehat{S}(t_n-r)-(\widehat{S}_{h,\Delta t})^{n-k-1}T_{h,\Delta t})v\|_{\H}\leq Ce^{-\sigma_0(n-k-1)\Delta t/2}(\Delta t^{\frac{1}{2}}+h^{\frac{1}{2}})\|v\|_{H^1(D)^6}$, for $r\in[t_{k},t_{k+1}]$ and $k=0,1,\cdots,n-1$.	
\end{lemm}

\begin{theo}\label{dgorder}
	Let $Q_1^{\frac{1}{2}}\in HS(L^2(D),H_0^{\gamma_1}(D))$, $Q_2^{\frac{1}{2}}\in HS(L^2(D),H_0^{\gamma_2}(D))$ hold with $\gamma_1:=1+\gamma>\frac52$ and $\gamma_2\geq1$, and  let $u_0\in L^4(\Omega,H^1(D)^6)$ and $F_{Q_1}\in W^{1,\infty}(D)$. There exists a positive constant $C_{12}$ independent of $\Delta t$ and $h$ such that
	$$\sup_{n\geq 0}\big(\mathbb{E}[\|u^n_h-u(t_n)\|^2_{\H}]\big)^{\frac{1}{2}}\leq C_{12}(\Delta t^{\frac{1}{2}}+h^{\frac{1}{2}}),$$
	where the positive constant $C_{12}$ depends on $|D|$, $\|F_{Q_1}\|_{W^{1,\infty}(D)}$,
	$C_1,$
	$\mathbb{E}\big[\|u_0\|^4_{H^1(D)^6}\big],$
	$ \lambda_{1},$$\widetilde{{\bm{\lambda}}}_2,$
	$\|Q_1^{\frac{1}{2}}\|_{\mathcal{L}_2^{\gamma_1}}$ and $\|Q_2^{\frac{1}{2}}\|_{\mathcal{L}_2^{\gamma_2}}$. 		
\end{theo}
\begin{proof}
	For $n\in\mathbb{N}_+$, we introduce an auxiliary process
	\begin{align}\label{jhg}
		\widetilde{u}^{n}_h&=\big(\widehat{S}_{h,\Delta t}\big)^nu^0_h+\lambda_{1}\sum_{k=0}^{n-1}\big(\widehat{S}_{h,\Delta t}\big)^{n-k-1}T_{h,\Delta t}\big(JA_t^{\sigma_0}u(t_k)\Delta\overline{W}_1^k\big)\notag\\
		&\quad+\sum_{k=0}^{n-1}\big(\widehat{S}_{h,\Delta t}\big)^{n-k-1}T_{h,\Delta t}\big(\widetilde{{\bm{\lambda}}}_2\Delta W_2^k\big),
	\end{align}
	that is,
	\begin{align}\label{oneauxi}
		\widetilde{u}^{n+1}_h-e^{-\sigma_0\Delta t}\widetilde{u}^n_h&=\frac{\Delta t}{2}\big(M_h\widetilde{u}_h^{n+1}+e^{-\sigma_0\Delta t}M_h\widetilde{u}_h^n\big)\notag\\
		&\quad+\frac{\lambda_{1}}{2}\pi_h\Big[J\big(u(t_{n+1})+e^{-\sigma_0\Delta t}u(t_n)\big)\Delta\overline{W}_1^n\Big]+\pi_h\big(\widetilde{\bm{\lambda}}_2\Delta W_2^n\big).
	\end{align}
	Let $e^{full}_n:=u^n_h-u(t_n)=\widetilde{e}_n+\widehat{e}_n$, where $\widetilde{e}_n:=u^n_h-\widetilde{u}^n_h$ and $\widehat{e}_n:=\widetilde{u}^n_h-u(t_n)$.
	
	{\em Step 1. Estimate of $\,\mathbb{E}\big[\|\widehat{e}_{n+1}\|^2_{\H}\|\Delta \overline{W}_1^n\|^{2p}_{H^{\gamma}(D)}\big],\,p=0,1,2$.}
	
	We use $\widehat{e}_{n+1}=\widetilde{u}^{n+1}_h-u(t_{n+1})$, \eqref{ewe} and \eqref{jhg} to obtain
	\begin{align*}
		&\mathbb{E}\big[\|\widehat{e}_{n+1}\|^2_{\H}\|\Delta \overline{W}_1^n\|^{2p}_{H^{\gamma}(D)}\big]\\
		&\leq 3\mathbb{E}\big[\big\|\big((\widehat{S}_{h,\Delta t})^{n+1}\pi_h-\widehat{S}(t_{n+1})\big)u_0\big\|^2_{\H}\big]\mathbb{E}\big[\|\Delta\overline{W}_1^n\|^{2p}_{H^{\gamma}(D)}\big]\\
		&\quad+3\mathbb{E}\Big[\Big\|\sum_{k=0}^{n}(\widehat{S}_{h,\Delta t})^{n-k}T_{h,\Delta t}(\widetilde{{\bm{\lambda}}}_2\Delta W_2^k)\\
		&\qquad\qquad\qquad-\int_{0}^{t_{n+1}}\widehat{S}(t_{n+1}-r)\widetilde{{\bm{\lambda}}}_2\dd W_2(r)\Big\|^2_{\H}\Big]\mathbb{E}\big[\|\Delta\overline{W}_1^n\|^{2p}_{H^{\gamma}(D)}\big]\\
		&\quad+3\mathbb{E}\Big[\Big\|\lambda_1\sum_{k=0}^{n}\big(\widehat{S}_{h,\Delta t}\big)^{n-k}T_{h,\Delta t}\big(JA^{\sigma_0}_tu(t_k)\Delta\overline{W}_1^k\big)\\
		&\qquad\qquad\qquad-\lambda_{1}\int_{0}^{t_{n+1}}\widehat{S}(t_{n+1}-r)Ju(r)\dd W_1(r)\\
		&\quad\quad\quad\qquad\quad+\frac{1}{2}\lambda_{1}^2\int_{0}^{t_{n+1}}\widehat{S}(t_{n+1}-r)\big(F_{Q_1}u(r)\big)\dd r\Big\|_{\H}^2\|\Delta\overline{W}_1^n\|^{2p}_{H^{\gamma}(D)}\Big]\\
		&=:I_{1}+I_{2}+I_{3}.
	\end{align*}
	For the term $I_{1}$, Lemma \ref{ooq}(4) yields
	\begin{align*}
		I_{1}\leq C\big(\Delta t+h\big)\Delta t^p.
	\end{align*}
	For the term $I_{2}$, it follows from Lemma \ref{opee}(5) and $Q_2^{\frac{1}{2}}\in HS(L^2(D),H_0^{\gamma_2}(D))$ that
	\begin{align*}
		I_{2}&\leq C\Delta t^p\sum_{k=0}^{n}\int_{t_k}^{t_{k+1}}\big\|(\widehat{S}(t_{n+1}-r)-(\widehat{S}_{h,\Delta t})^{n-k}T_{h,\Delta t})\widetilde{{\bm{\lambda}}}_2Q_2^{\frac{1}{2}}\big\|^2_{HS(L^2(D),\H)}\dd r\\
		&\leq C\Delta t^p\sum_{k=0}^{n}e^{-\sigma_0(n+1-k)\Delta t}\big(\Delta t+h\big)\Delta t\leq C\Delta t^p\big(\Delta t+h\big).
	\end{align*}
	For the term $I_{3}$, notice that 
	\begin{align*}
		&\lambda_{1}\sum_{k=0}^{n}\big(\widehat{S}_{h,\Delta t}\big)^{n-k}T_{h,\Delta t}\big(JA^{\sigma_0}_tu(t_k)\Delta\overline{W}_1^k\big)-\lambda_{1}\int_{0}^{t_{n+1}}\widehat{S}(t_{n+1}-r)Ju(r)\dd W_1(r)\\
		&\qquad\qquad+\frac{1}{2}\lambda_{1}^2\int_{0}^{t_{n+1}}\widehat{S}(t_{n+1}-r)\big(F_{Q_1}u(r)\big)\dd r\\
		&=\sum_{k=0}^{n}\Big[\lambda_{1}\big(\widehat{S}_{h,\Delta t}\big)^{n-k}T_{h,\Delta t}\big(Je^{-\sigma_0\Delta t}u(t_k)\big(\Delta \overline{W}_1^k-\Delta W_1^k\big)\big)\Big]\\
		&\quad+\sum_{k=0}^{n}\Big[\lambda_{1}\int_{t_k}^{t_{k+1}}\Big[\big(\widehat{S}_{h,\Delta t}\big)^{n-k}T_{h,\Delta t}J\big(e^{-\sigma_0\Delta t}u(t_k)\big)-\widehat{S}(t_{n+1}-r)Ju(r)\Big]\dd W_1(r)\Big]\\
		&\quad+\sum_{k=0}^{n}\Big[\lambda_{1}(\widehat{S}_{h,\Delta t})^{n-k}T_{h,\Delta t}\Big(J\frac{u(t_{k+1})-e^{-\sigma_0\Delta t}u(t_k)}{2}\Delta\overline{W}_1^k\Big)\\
		&\qquad\qquad\qquad+\frac{1}{2}\lambda_{1}^2\int_{t_k}^{t_{k+1}}\widehat{S}(t_{n+1}-r)\big(F_{Q_1}u(r)\big)\dd r\Big]\\
		&=:\sum_{k=0}^{n}A_{1,k}+\sum_{k=0}^{n}A_{2,k}+\sum_{k=0}^{n}A_{3,k}.
	\end{align*}
	Therefore,
	\begin{align*}
		I_{3}&\leq C\mathbb{E}\Big[\Big\|\sum_{k=0}^{n}A_{1,k}\|^2_{\H}\Big\|\Delta\overline{W}_1^n\|^{2p}_{H^{\gamma}(D)}\Big]+C\mathbb{E}\Big[\Big\|\sum_{k=0}^{n}A_{2,k}\Big\|^2_{\H}\|\Delta\overline{W}_1^n\|^{2p}_{H^{\gamma}(D)}\Big]\\
		&\qquad+C\mathbb{E}\Big[\Big\|\sum_{k=0}^{n}A_{3,k}\Big\|_{\H}^2\|\Delta\overline{W}_1^n\|^{2p}_{H^{\gamma}(D)}\Big]\\
		&=:I_{3,1}+I_{3,2}+I_{3,3}.
	\end{align*}
	
	For the term $I_{3,1}$, Lemma \ref{ooq}(1), Proposition \ref{obv}, Sobolev embedding $H^{\gamma}\hookrightarrow L^{\infty}(D)$ for $\gamma>\frac{3}{2}$ and \eqref{2ptruncate} yield 
	\begin{align*}
		\mathbb{E}\big[\|A_{1,k}\|^4_{\H}\big]&\leq Ce^{-4\sigma_0(n-k)\Delta t}\mathbb{E}\big[\|u(t_k)\|^4_{\H}\big]\mathbb{E}\big[\|\Delta\overline{W}_1^k-\Delta W_1^k\|^4_{H^{\gamma}(D)}\big]\\
		&\leq Ce^{-4\sigma_0(n-k)\Delta t}\Delta t^6,
	\end{align*}
	which implies 
	\begin{align*}
		&\mathbb{E}\Big[\Big\|\sum_{k=0}^{n}A_{1,k}\Big\|_{\H}^4\Big]
		\leq\Big(\sum_{k=0}^{n}\big(\mathbb{E}\big[\|A_{1,k}\|_{\H}^4\big]\big)^{\frac{1}{4}}\Big)^4\leq \Big(\sum_{k=0}^{n}Ce^{-\sigma_0(n-k)\Delta t}\Delta t^{\frac{3}{2}}\Big)^4\leq C\Delta t^2.
	\end{align*}	
	Hence 
	$$I_{3,1}\leq C\Big(\mathbb{E}\Big[\Big\|\sum_{k=0}^{n}A_{1,k}\Big\|_{\H}^4\Big]\Big)^{\frac{1}{2}}\Big(\mathbb{E}\big[\|\Delta\overline{W}_1^n\|_{H^{\gamma}(D)}^{4p}\big]\Big)^{\frac{1}{2}}\leq C\Delta t^{1+p}.$$
	
	For the term $I_{3,2}$, notice that
	\begin{align*}
		&\mathbb{E}\Big[\Big\|\sum_{k=0}^{n-1}A_{2,k}\Big\|^2_{\H}\Big]\\
		&\leq C\sum_{k=0}^{n-1}\mathbb{E}\int_{t_k}^{t_{k+1}}\sum_{i=1}^{\infty}\big\|\big[\big(\widehat{S}_{h,\Delta t}\big)^{n-k}T_{h,\Delta t}J\big(e^{-\sigma_0\Delta t}u(t_k)\big)-\widehat{S}(t_{n+1}-r)Ju(r)\big]Q_1^{\frac{1}{2}}q_i\big\|_{\H}^2\dd r\\
		&\leq C\bigg(\sum_{k=0}^{n-1}\mathbb{E}\int_{t_k}^{t_{k+1}}\Big[ \sum_{i=1}^{\infty}\big\|\big(\widehat{S}_{h,\Delta t}\big)^{n-k}T_{h,\Delta t}J\big(\big(e^{-\sigma_0\Delta t}u(t_k)-u(r)\big)Q_1^{\frac{1}{2}}q_i\big)\big\|_{\H}^2\\
		&\qquad\qquad\qquad\qquad\quad+\sum_{i=1}^{\infty}\big\|\big(\big(\widehat{S}_{h,\Delta t}\big)^{n-k}T_{h,\Delta t}-\widehat{S}(t_{n+1}-r)\big)J\big(u(r)Q_1^{\frac{1}{2}}q_i\big)\big\|^2_{\H}\Big]\dd r\bigg)\\
		&\leq C\sum_{k=0}^{n-1}e^{-(n-k)\sigma_0\Delta t}(\Delta t+h)\Delta t\leq C(\Delta t+h),
	\end{align*}
	where we use \eqref{bnb}, Lemma \ref{regularity in H1}, the It\^o isometry, Sobolev embedding $H^{\gamma}(D)\hookrightarrow L^{\infty}(D)$ for $\gamma>\frac{3}{2}$ and Lemma \ref{ooq}(1)(5). Similar to the estimate of \eqref{bnb}, for $r\in[t_n,t_{n+1}]$, we have
	\begin{align}\label{wfz}
		\mathbb{E}\big[\|u(r)-e^{-\sigma_0\Delta t}u(t_n)\|_{\H}^{2p}\big]\leq C\Delta t^{p}.
	\end{align}
	By utilizing the Burkholder--Davis--Gundy-type inequality for stochastic integrals, \eqref{wfz}, Proposition \ref{obv}, Lemma \ref{opee}(1) and Lemma \ref{ooq}(1), it yields that
	\begin{align*}
		\mathbb{E}\big[\|A_{2,n}\|_{\H}^4\big]\leq C\Delta t^2.
	\end{align*}
	Consequently,
	\begin{align*}
		I_{3,2}&\leq C\mathbb{E}\Big[\|\sum_{k=0}^{n-1}A_{2,k}\|^2_{\H}\Big]\mathbb{E}\big[\|\Delta\overline{W}_1^n\|^{2p}_{H^{\gamma}(D)}\big]+C\Big(\mathbb{E}\big[\|A_{2,n}\|^4_{\H}\big]\Big)^{\frac{1}{2}}\Big(\mathbb{E}\big[\|\Delta\overline{W}_1^n\|^{4p}_{H^{\gamma}(D)}\big]\Big)^{\frac{1}{2}}\\
		&\leq C(\Delta t+h)\Delta t^p.
	\end{align*}
	
	For the term $I_{3,3}$, we substitute \eqref{ewe} into $A_{3,k}$ and further split it to obtain 
	\begin{align*}
		A_{3,k}:=A_{3,k,1}+A_{3,k,2}+A_{3,k,3}+A_{3,k,4}+A_{3,k,5},
	\end{align*}
	where
	{\small	\begin{align*}
			A_{3,k,1}&=\frac{\lambda_{1}}{2}(\widehat{S}_{h,\Delta t})^{n-k}T_{h,\Delta t}\Big[J\big(\big(\widehat{S}(\Delta t)-e^{\sigma_0\Delta t}\big)u(t_k)\big)\Delta\overline{W}_1^k\Big]\\
			&\quad-\frac{\lambda_1^3}{4}(\widehat{S}_{h,\Delta t})^{n-k}T_{h,\Delta t}\Big[J\Big(\int_{t_k}^{t_{k+1}}\widehat{S}(t_{k+1}-r)\big(F_{Q_1}u(r)\big)\dd r\Big)\Delta\overline{W}_1^k\Big]\\
			&\quad+\frac{\lambda_{1}^2}{2}(\widehat{S}_{h,\Delta t})^{n-k}T_{h,\Delta t}\Big[J\Big(\int_{t_k}^{t_{k+1}}\widehat{S}(t_{k+1}-r)J\big(u(r)-e^{-\sigma_0\Delta t}u(t_k)\big)\dd W_1(r)\Big)\Delta\overline{W}_1^k\Big]\\
			&\quad+\frac{\lambda_{1}^2}{2}(\widehat{S}_{h,\Delta t})^{n-k}T_{h,\Delta t}\Big[J\Big(\int_{t_k}^{t_{k+1}}\widehat{S}(t_{k+1}-r)J\big(e^{-\sigma_0\Delta t}u(t_k)\big)\dd W_1(r)\Big)\big(\Delta\overline{W}_1^k-\Delta W_1^k\big)\Big],\\
			A_{3,k,2}&=\frac{\lambda_{1}^2}{2}\big(\widehat{S}_{h,\Delta t}\big)^{n-k}T_{h,\Delta t}\Big[J\Big(\int_{t_k}^{t_{k+1}}\big(\widehat{S}(t_{k+1}-r)-I\big)J\big(e^{-\sigma_0\Delta t}u(t_k)\big)\dd W_1(r)\Big)\Delta W_1^k\Big]\\
			&\quad-\frac{\lambda_1^2}{2}\big(\widehat{S}_{h,\Delta t}\big)^{n-k}T_{h,\Delta t}\int_{t_k}^{t_{k+1}}F_{Q_1}\big(e^{-\sigma_0\Delta t}u(t_k)-u(r)\big)\dd r\\
			&\quad+\frac{\lambda_{1}^2}{2}\int_{t_k}^{t_{k+1}}\big(\widehat{S}(t_{n+1}-r)-(\widehat{S}_{h,\Delta t})^{n-k}T_{h,\Delta t}\big)\big(F_{Q_1}\big(u(r)-e^{-\sigma_0\Delta t}u(t_k)\big)\big)\dd r,\\
			A_{3,k,3}&=\frac{\lambda_{1}}{2}(\widehat{S}_{h,\Delta t})^{n-k}T_{h,\Delta t}\Big[J\Big(\int_{t_k}^{t_{k+1}}\widehat{S}(t_{k+1}-r)\widetilde{{\bm{\lambda}}}_2\dd W_2(r)\Big)\Delta\overline{W}_1^k\Big],\\
			A_{3,k,4}&=-\lambda_{1}^2(\widehat{S}_{h,\Delta t})^{n-k}T_{h,\Delta t}\Big(\int_{t_k}^{t_{k+1}}\int_{t_k}^{r}e^{-\sigma_0\Delta t}u(t_k)\dd W_1(\rho)\dd W_1(r)\Big),\\
			A_{3,k,5}&=\frac{\lambda_{1}^2}{2}\int_{t_k}^{t_{k+1}}\Big(\widehat{S}(t_{n+1}-r)-(\widehat{S}_{h,\Delta t})^{n-k}T_{h,\Delta t}\Big)\big(F_{Q_1}e^{-\sigma_0\Delta t}u(t_k)\big)\dd r.
	\end{align*}}

	For the term $A_{3,k,1}$, it follows from Proposition \ref{obv}, Sobolev embedding $H^{\gamma}\hookrightarrow L^{\infty}(D)$ for $\gamma>\frac{3}{2}$, Lemma \ref{opee}(1)(3), Lemma \ref{ooq}(1), \eqref{wfz}, the Burkholder--Davis--Gundy-type inequality and \eqref{2ptruncate} that
	{\small\begin{align*}
			&\mathbb{E}\big[\|A_{3,k,1}\|^2_{\H}\big]\\
			&\leq Ce^{-2(n-k)\sigma_0\Delta t}\Delta t^2\mathbb{E}\big[\|u(t_k)\|^2_{D(M)}\big]\mathbb{E}\big[\|\Delta\overline{W}_1^k\|_{H^{\gamma}(D)}^2\big]\\
			&\quad+Ce^{-2(n-k)\sigma_0\Delta t}\mathbb{E}\Big[\Big\|\int_{t_k}^{t_{k+1}}\widehat{S}(t_{k+1}-r)(F_{Q_1}u(r))\dd r\Big\|_{\H}^2\|\Delta\overline{W}_1^k\|^2_{H^{\gamma}(D)}\Big]\\
			&\quad+Ce^{-2(n-k)\sigma_0\Delta t}\mathbb{E}\Big[\Big\|\int_{t_k}^{t_{k+1}}\widehat{S}(t_{k+1}-r)J(u(r)-e^{-\sigma_0\Delta t}u(t_k))\dd W_1(r)\Big\|^2_{\H}\|\Delta\overline{W}_1^k\|^2_{H^{\gamma}(D)}\Big]\\
			&\quad+Ce^{-2(n-k)\sigma_0\Delta t}\mathbb{E}\Big[\Big\|\int_{t_k}^{t_{k+1}}\widehat{S}(t_{k+1}-r)J(e^{-\sigma_0\Delta t}u(t_k))\dd W_1(r)\Big\|^2_{\H}\|\Delta\overline{W}_1^k-\Delta W_1^k\|^2_{H^{\gamma}(D)}\Big]\\
			&\leq Ce^{-2(n-k)\sigma_0\Delta t}\Delta t^3.
	\end{align*}}
	Similarly, we have
	\begin{align*}
		\mathbb{E}\big[\|A_{3,n,1}\|_{\H}^2\|\Delta\overline{W}_1^n\|^{2p}_{H^{\gamma}(D)}\big]\leq C\Delta t^{3+p}.
	\end{align*}
	Therefore,
	\begin{align*}
		\mathbb{E}\Big[\Big\|\sum_{k=0}^{n}A_{3,k,1}\Big\|^2_{\H}\|\Delta\overline{W}_1^n\|^{2p}_{H^{\gamma}(D)}\Big]\leq C\Delta t^{1+p}.
	\end{align*}
	
	For the term $A_{3,k,2}$, Lemma \ref{opee}(3), \eqref{wfz} and the Burkholder--Davis--Gundy-type inequality yield
	\begin{align*}
		&\mathbb{E}[\|A_{3,k,2}\|^4_{\H}]\\
		&\leq Ce^{-4\sigma_0(n-k)\Delta t}\mathbb{E}\Big[\Big\|\int_{t_k}^{t_{k+1}}\Big(\widehat{S}(t_{k+1}-r)-I\Big)J(e^{-\sigma_0\Delta t}u(t_k))\dd W_1(r)\Big\|^4_{\H}\|\Delta W_1\|^4_{H^{\gamma}(D)}\Big]\\
		&\quad+Ce^{-4(n-k)\sigma_0\Delta t}\mathbb{E}\Big[\Big\|\int_{t_k}^{t_{k+1}}F_{Q_1}(u(r)-e^{-\sigma_0\Delta t}u(t_k))\dd r\Big\|_{\H}^4\Big]\\
		&\quad+C\Delta t^3\int_{t_k}^{t_{k+1}}e^{-4\sigma_0(n-k)\Delta t}\mathbb{E}[\|F_{Q_1}(u(r)-e^{-\sigma_0\Delta t}u(t_k))\|_{\H}^4]\dd r\\
		&\leq Ce^{-4(n-k)\sigma_0\Delta t}\Delta t^6,
	\end{align*}
	which implies that
	\begin{align*}
		&\mathbb{E}\Big[\Big\|\sum_{k=0}^{n}A_{3,k,2}\Big\|^2_{\H}\|\Delta\overline{W}_1^n\|^{2p}_{H^{\gamma}(D)}\Big]\\
		&\leq \Big(\mathbb{E}\Big[\Big\|\sum_{k=0}^{n}A_{3,k,2}\Big\|_{\H}^{4}\Big]\Big)^{\frac{1}{2}}\Big(\mathbb{E}\big[\|\Delta\overline{W}_1^n\|^{4p}_{H^{\gamma}(D)}\big]\Big)^{\frac{1}{2}}\leq C\Delta t^{1+p}.
	\end{align*}
	
	For the term $A_{3,k,3}$, notice that
	\begin{align*}
		\mathbb{E}\big[\|A_{3,k,3}\|_{\H}^4\big]&\leq Ce^{-4(n-k)\sigma_0\Delta t}\mathbb{E}\Big[\Big\|\int_{t_k}^{t_{k+1}}\widehat{S}(t_{k+1}-r)\widetilde{{\bm{\lambda}}}_2\dd W_2(r)\Big\|^4_{\H}\Big]\mathbb{E}\big[\|\Delta\overline{W}_1^k\|^4_{H^{\gamma}(D)}\big]\\
		&\leq Ce^{-4(n-k)\sigma_0\Delta t}\Delta t^4,
	\end{align*}
	which leads to
	\begin{align*}
		&\mathbb{E}\Big[\Big\|\sum_{k=0}^{n}A_{3,k,3}\Big\|^4_{\H}\Big]
		\leq C\sum_{i=0}^{n}\mathbb{E}\big[\|A_{3,i,3}\|_{\H}^4\big]+C\sum_{i\neq j}^{n}\mathbb{E}\big[\|A_{3,i,3}\|_{\H}^2\|A_{3,j,3}\|_{\H}^2\big]\\
		&\leq C\Big( \sum_{i=0}^{n}\big(\mathbb{E}\big[\|A_{3,i,3}\|_{\H}^4\big]\big)^{\frac{1}{2}}\Big)^2
		\leq C\Big(\sum_{k=0}^{n}e^{-2(n-k)\sigma_0\Delta t}\Delta t^2\Big)^2\leq C\Delta t^2.
	\end{align*}
	Hence 
	\begin{align*}
		&\mathbb{E}\Big[\Big\|\sum_{k=0}^{n}A_{3,k,3}\Big\|^2_{\H}\|\Delta\overline{W}_1^n\|^{2p}_{H^{\gamma}(D)}\Big]\\
		&\leq \Big(\mathbb{E}\Big[\Big\|\sum_{k=0}^{n}A_{3,k,3}\Big\|_{\H}^4\Big]\Big)^{\frac{1}{2}}\Big(\mathbb{E}\big[\|\Delta\overline{W}_1^n\|^{4p}_{H^{\gamma}(D)}\big]\Big)^{\frac{1}{2}}\leq C\Delta t^{1+p}.
	\end{align*}
	
	For the term $A_{3,k,4}$, it holds that
	\begin{align*}
		\mathbb{E}\Big[\Big\|\sum_{k=0}^{n-1}A_{3,k,4}\Big\|^2_{\H}\Big]&\leq C\Big(\sum_{k=0}^{n-1}\int_{t_k}^{t_{k+1}}e^{-2\sigma_0(n-k)\Delta t}\mathbb{E}\Big[\Big\|\int_{t_k}^{r}e^{-\sigma_0\Delta t}u(t_k)\dd W_1(\rho)\Big\|^2_{\H}\Big]\dd r\Big)\\
		&\leq C\sum_{k=0}^{n-1}e^{-2\sigma_0(n-k)\Delta t}\Delta t^2\leq C\Delta t
	\end{align*}
	and
	\begin{align*}
		\mathbb{E}\big[\|A_{3,n,4}\|_{\H}^4\big]&\leq C\mathbb{E}\Big[\Big(\int_{t_k}^{t_{k+1}}\Big\|\int_{t_k}^{r}e^{-\sigma_0\Delta t}u(t_k)\dd W_1(\rho)\Big\|_{\H}^2\dd r\Big)^2\Big]\\
		&\leq C\Delta t\int_{t_k}^{t_{k+1}}\mathbb{E}\Big[\Big\|\int_{t_k}^{r}u(t_k)\dd W_1(\rho)\Big\|^4_{\H}\Big]\dd r\leq C\Delta t^4,
	\end{align*}
	which imply that
	\begin{align*}
		\mathbb{E}\Big[\Big\|\sum_{k=0}^{n}A_{3,k,4}\Big\|^2_{\H}\|\Delta\overline{W}_1^n\|^{2p}_{H^{\gamma}(D)}\Big]\leq C\Delta t^{1+p}.
	\end{align*}
	
	For the term $A_{3,k,5}$, notice that
	\begin{align*}
		&\mathbb{E}\big[\|A_{3,k,5}\|_{\H}^2\big]\\
		&\leq C\Delta t\int_{t_k}^{t_{k+1}}\mathbb{E}\big[\|(\widehat{S}(t_{n+1}-r)-(S_{h,\Delta t})^{n-k}T_{h,\Delta t})(F_{Q_1}e^{-\sigma_0\Delta t}u(t_k))\|_{\H}^2\big]\dd r\\
		&\leq C\Delta t^2(\Delta t+h)e^{-\sigma_0(n-k)\Delta t},
	\end{align*}
	which yields
	\begin{align*}
		\mathbb{E}\Big[\Big\|\sum_{k=0}^{n}A_{3,k,5}\Big\|^2_{\H}\|\Delta\overline{W}_1^n\|^{2p}_{H^{\gamma}(D)}\Big]\leq C(\Delta t+h)\Delta t^p.
	\end{align*}
	Hence 
	\begin{align*}
		I_{3,3}\leq C\sum_{i=1}^{5}\mathbb{E}\Big[\Big\|\sum_{k=0}^{n}A_{3,k,i}\Big\|^2_{\H}\|\Delta\overline{W}_1^n\|^{2p}_{H^{\gamma}(D)}\Big]\leq C(\Delta t+h)\Delta t^p.
	\end{align*}
	Combining $I_{3,1}$--$I_{3,3}$, we have $I_{3}\leq C(\Delta t+h)\Delta t^p$. Further, we have 
	\begin{align}\label{ee1}
		\mathbb{E}\big[\|\widehat{e}_{n+1}\|^2_{\H}\|\Delta \overline{W}_1^n\|^{2p}_{H^{\gamma}(D)}\big]\leq C(\Delta t+h)\Delta t^{p}.
	\end{align}
	
	{\em Step 2. Estimate of $\sup_{n\in\mathbb{N}}\mathbb{E}\big[\|\widetilde{e}_n\|^2_{\H}\big]$.}
	
	Subtracting \eqref{oneauxi} from \eqref{dgfull} leads to
	\begin{align}\label{gf}
		\widetilde{e}_{n+1}-e^{-\sigma_0\Delta t}\widetilde{e}_n=\frac{\Delta t}{2}M_h(\widetilde{e}_{n+1}+e^{-\sigma_0\Delta t}\widetilde{e}_n)+\frac{\lambda_{1}}{2}\pi_h\big[J(e^{full}_{n+1}+e^{-\sigma_0\Delta t}e^{full}_n)\Delta \overline{W}_1^n\big],
	\end{align}
	which is equivalent to 
	\begin{align}\label{ufnz}
		\widetilde{e}_{n+1}=\widehat{S}_{h,\Delta t}\widetilde{e}_n+\frac{\lambda_{1}}{2}T_{h,\Delta t}\big[J(e_{n+1}^{full}+e^{-\sigma_0\Delta t}e^{full}_n)\Delta\overline{W}_1^n\big].
	\end{align}
	We apply $\langle\cdot,\widetilde{e}_{n+1}+e^{-\sigma_0\Delta t}\widetilde{e}_n\rangle_{\H}$ to both sides of \eqref{gf} and take the expectation to get
	\begin{align*}
		&\mathbb{E}\big[\|\widetilde{e}_{n+1}\|^2_{\H}\big]\\
		&\leq e^{-2\sigma_0\Delta t}\mathbb{E}\big[\|\widetilde{e}_n\|^2_{\H}\big]+\frac{\lambda_{1}}{2}\mathbb{E}\big[\langle J(\widehat{e}_{n+1}+e^{-\sigma_0\Delta t}\widehat{e}_n)\Delta\overline{W}_1^n,\widetilde{e}_{n+1}+e^{-\sigma_0\Delta t}\widetilde{e}_n\rangle_{\H}\big]
	\end{align*}
	due to the dissipative property of $M_h$ in \cite[Proposition 4.4(ii)]{CCC2021}.
	
	By \eqref{ufnz} and using the fact that $\Delta\overline{W}_1^n$ is independent of $\mathcal{F}_{t_n}$, we obtain
	\begin{align*}
		&\frac{\lambda_{1}}{2}\mathbb{E}\big[\langle J(\widehat{e}_{n+1}+e^{-\sigma_0\Delta t}\widehat{e}_n)\Delta\overline{W}_1^n,\widetilde{e}_{n+1}+e^{-\sigma_0\Delta t}\widetilde{e}_n\rangle_{\H}\big]\\
		&=\frac{\lambda_{1}}{2}\mathbb{E}\big[\langle J(\widehat{e}_{n+1}+e^{-\sigma_0\Delta t}\widehat{e}_n)\Delta\overline{W}_1^n,\widetilde{e}_{n+1}-e^{-\sigma_0\Delta t}\widetilde{e}_n\rangle_{\H}\big]\\
		&\quad+\lambda_{1}\mathbb{E}\big[\langle J(\widehat{e}_{n+1}-e^{-\sigma_0\Delta t}\widehat{e}_n)\Delta\overline{W}_1^n,e^{-\sigma_0\Delta t}\widetilde{e}_n\rangle_{\H}\big]\\
		&\quad+2\lambda_{1}\mathbb{E}\big[\langle Je^{-\sigma_0\Delta t}\widehat{e}_n\Delta\overline{W}_1^n,e^{-\sigma_0\Delta t}\widetilde{e}_n\rangle_{\H}\big]\\
		&=\frac{\lambda_{1}}{2}\mathbb{E}\big[\langle J(\widehat{e}_{n+1}-e^{-\sigma_0\Delta t}\widehat{e}_n)\Delta\overline{W}_1^n,(\widehat{S}_{h,\Delta t}+e^{-\sigma_0\Delta t})\widetilde{e}_n\rangle_{\H}\big]\\
		&\quad+\frac{\lambda_1^2}{4}\mathbb{E}\big[\langle J(\widehat{e}_{n+1}+e^{-\sigma_0\Delta t}\widehat{e}_n)\Delta\overline{W}_1^n,T_{h,\Delta t}(J(e_{n+1}^{full}+e^{-\sigma_0\Delta t}e_n^{full})\Delta\overline{W}_1^n)\rangle_{\H}\big]\\
		&=:I_4+I_5.
	\end{align*}
	Noticing that
	{\small
		\begin{align*}
			&\widehat{e}_{n+1}-e^{-\sigma_0\Delta t}\widehat{e}_n\\
			&=\big(\widehat{S}_{h,\Delta t}-e^{-\sigma_0\Delta t}\big)\big(\widehat{S}_{h,\Delta t}\big)^n\pi_hu_0-\big(\widehat{S}(\Delta t)-e^{-\sigma_0\Delta t}\big)\widehat{S}(t_n)u_0\\
			&\quad+\int_{t_n}^{t_{n+1}}\big(T_{h,\Delta t}-\widehat{S}(t_{n+1}-r)\big)\widetilde{{\bm{\lambda}}}_2\dd W_2(r)\\
			&\quad+\sum_{k=0}^{n-1}\int_{t_k}^{t_{k+1}}\Big[\big(\widehat{S}_{h,\Delta t}-e^{-\sigma_0\Delta t}\big)\big(\widehat{S}_{h,\Delta t}\big)^{n-1-k}T_{h,\Delta t}-\big(\widehat{S}(\Delta t)-e^{-\sigma_0\Delta t}\big)\widehat{S}(t_n-r)\Big]\widetilde{{\bm{\lambda}}}_2\dd W_2(r)\\
			&\quad+\lambda_{1}T_{h,\Delta t}\big(JA^{\sigma_0}_tu(t_n)\Delta\overline{W}_1^n\big)-\lambda_{1}\int_{t_n}^{t_{n+1}}\widehat{S}(t_{n+1}-r)Ju(r)\dd W_1(r)\\
			&\quad+\frac{1}{2}\lambda_{1}^2\int_{t_n}^{t_{n+1}}\widehat{S}(t_{n+1}-r)\big(F_{Q_1}u(r)\big)\dd r\\
			&\quad+\lambda_{1}\big(\widehat{S}_{h,\Delta t}-e^{-\sigma_0\Delta t}\big)\sum_{k=0}^{n-1}\big(\widehat{S}_{h,\Delta t}\big)^{n-k-1}T_{h,\Delta t}\big(JA^{\sigma_0}_tu(t_k)\Delta\overline{W}_1^k\big)\\
			&\quad-\lambda_{1}\big(\widehat{S}(\Delta t)-e^{{-\sigma_0\Delta t}}\big)\int_{0}^{t_{n}}\widehat{S}(t_n-r)Ju(r)\dd W_1(r)\\
			&\quad+\frac{1}{2}\lambda_{1}^2\big(\widehat{S}(\Delta t)-e^{-\sigma_0\Delta t}\big)\int_{0}^{t_{n}}\widehat{S}(t_n-r)\big(F_{Q_1}u(r)\big)\dd r,
	\end{align*}}
	and using the fact that $\Delta\overline{W}_1^n$ is independent of $\mathcal{F}_{t_n}$ and $W_2(t)$, we get 
	\begin{align*}
		I_4&=\frac{\lambda_{1}}{2}\mathbb{E}\Big[\langle J\Big(\lambda_{1}T_{h,\Delta t}(JA^{\sigma_0}_tu(t_n)\Delta\overline{W}_1^n)-\lambda_{1}\int_{t_n}^{t_{n+1}}\widehat{S}(t_{n+1}-r)Ju(r)\dd W_1(r)\\
		&\qquad\qquad\quad+\frac{\lambda_{1}^2}{2}\int_{t_n}^{t_{n+1}}\widehat{S}(t_{n+1}-r)(F_{Q_1}u(r))\dd r\Big)\Delta\overline{W}_1^n,(\widehat{S}_{h,\Delta t}+e^{-\sigma_0\Delta t})\widetilde{e}_n\rangle_{\H}\Big]\\
		&=\frac{\lambda_{1}^2}{2}\mathbb{E}\Big[\langle J\big(T_{h,\Delta t}\big[JA^{\sigma_0}_tu(t_n)(\Delta\overline{W}_1^n-\Delta W_1^n)\big]\big)\Delta\overline{W}_1^n,\big(\widehat{S}_{h,\Delta t}+e^{-\sigma_0\Delta t}\big)\widetilde{e}_n\rangle_{\H}\Big]\\
		&\quad+\frac{\lambda_{1}^2}{2}\mathbb{E}\Big[\langle J\int_{t_n}^{t_{n+1}}\big[T_{h,\Delta t}(Je^{-\sigma_0\Delta t}u(t_n))\\
		&\qquad\qquad\qquad\qquad\qquad-\widehat{S}(t_{n+1}-r)Ju(r)\big]\dd W_1(r)\Delta\overline{W}_1^n,\big(\widehat{S}_{h,\Delta t}+e^{-\sigma_0\Delta t}\big)\widetilde{e}_n\rangle_{\H}\Big]\\
		&\quad+\frac{\lambda_{1}^2}{2}\mathbb{E}\Big[\langle J\big[T_{h,\Delta t}\big(J\frac{u(t_{n+1})-e^{-\sigma_0\Delta t}u(t_n)}{2}\Delta W_1^n\big)\big]\Delta\overline{W}_1^n,\big(\widehat{S}_{h,\Delta t}+e^{-\sigma_0\Delta t}\big)\widetilde{e}_n\rangle_{\H}\Big]\\
		&\quad+\frac{\lambda_{1}^3}{4}\mathbb{E}\Big[\langle J\int_{t_n}^{t_{n+1}}\widehat{S}(t_{n+1}-r)(F_{Q_1}u(r))\dd r\Delta\overline{W}_1^n,\big(\widehat{S}_{h,\Delta t}+e^{-\sigma_0\Delta t}\big)\widetilde{e}_n\rangle_{\H}\Big]\\
		&=:I_{4,1}+I_{4,2}+I_{4,3}+I_{4,4}.
	\end{align*}
	
	For the term $I_{4,1}$, it holds that 
	\begin{align*}
		I_{4,1}&\leq C\mathbb{E}\big[\|A^{\sigma_0}_tu(t_n)\|_{\H}\|\Delta\overline{W}_1^n-\Delta W_1^n\|_{H^{\gamma}(D)}\|\Delta\overline{W}_1^n\|_{H^{\gamma}(D)}e^{-\sigma_0\Delta t}\|\widetilde{e}_n\|_{\H}\big]\\
		&\leq \frac{\sigma_0}{16}e^{-2\sigma_0\Delta t}\Delta t\mathbb{E}\big[\|\widetilde{e}_n\|_{\H}^2\big]+C\Delta t^2
	\end{align*}
	due to Proposition \ref{obv}, the Sobolev embedding $H^{\gamma}(D)\hookrightarrow L^{\infty}(D)$ for $\gamma>\frac{3}{2}$, Lemma \ref{ooq}(1) and \eqref{2ptruncate}.
	
	For the term $I_{4,2}$, we use \eqref{bnb}, Lemma \ref{ooq}(1)(5), $Q_1^{\frac{1}{2}}\in HS(L^2(D),H^{\gamma_1}_0(D))$, Lemma \ref{regularity in H1} and Sobolev embedding $H^{\gamma}(D)\hookrightarrow L^\infty(D)$ for $\gamma>\frac{3}{2}$ to obtain
	\begin{align*}
		I_{4,2}&\leq C\mathbb{E}\Big[\Big\|\int_{t_n}^{t_{n+1}}\big[T_{h,\Delta t}\big(Je^{-\sigma_0\Delta t}u(t_n)\big)\\
		&\qquad\qquad\qquad\qquad-\widehat{S}(t_{n+1}-r)\big(Ju(r)\big)\big]\dd W_1(r)\Big\|_{\H}\|\Delta\overline{W}_1^n\|_{H^{\gamma}(D)}e^{-\sigma_0\Delta t}\|\widetilde{e}_n\|_{\H}\Big]\\
		&\leq C\mathbb{E}\Big[\Big\|\int_{t_n}^{t_{n+1}}\big[T_{h,\Delta t}J\big(e^{-\sigma_0\Delta t}u(t_n)-u(r)\big)\\
		&\qquad\qquad\qquad\qquad+\big(T_{h,\Delta t}-\widehat{S}(t_{n+1}-r)\big)Ju(r)\big]\dd W_1(r)\Big\|^2_{\H}\Big]\\
		&\quad+\frac{\sigma_0}{16}\Delta te^{-2\sigma_0\Delta t}\mathbb{E}\big[\|\widetilde{e}_n\|^2_{\H}\big]\\
		&\leq C\mathbb{E}\Big[\int_{t_n}^{t_{n+1}}\sum_{k=1}^{\infty}\big\|T_{h,\Delta t}J\big(u(r)-e^{-\sigma_0\Delta t}u(t_n)\big)Q_1^{\frac{1}{2}}q_{k}\big\|_{\H}^2\dd r\Big]\\
		&\quad+C\mathbb{E}\Big[\int_{t_n}^{t_{n+1}}\sum_{k=1}^{\infty}\big\|\big(\widehat{S}(t_{n+1}-r)-T_{h,\Delta t}\big)Ju(r)Q_1^{\frac{1}{2}}q_{k}\big\|^2_{\H}\dd r\Big]\\
		&\quad+\frac{\sigma_0}{16}\Delta te^{-2\sigma_0\Delta t}\mathbb{E}\big[\|\widetilde{e}_n\|^2_{\H}\big]\\
		&\leq C(\Delta t+h)\Delta t+\frac{\sigma_0}{16}\Delta te^{-2\sigma_0\Delta t}\mathbb{E}\big[\|\widetilde{e}_n\|_{\H}^2\big].
	\end{align*}
	
	For the term $I_{4,3}$, Lemma \ref{ooq}(1), \eqref{wfz} and Sobolev embedding $H^{\gamma}(D)\hookrightarrow L^{\infty}(D)$ for $\gamma>\frac{3}{2}$ yield
	\begin{align*}
		I_{4,3}&\leq C\mathbb{E}\big[\|u(t_{n+1})-e^{-\sigma_0\Delta t}u(t_n)\|_{\H}\|\Delta\overline{W}_1^n\|_{H^{\gamma}(D)}\|\Delta W_1^n\|_{H^{\gamma}(D)}e^{-\sigma_0\Delta t}\|\widetilde{e}_n\|_{\H}\big]\\
		&\leq \frac{\sigma_0}{16}\Delta t e^{-2\sigma_0\Delta t}\mathbb{E}\big[\|\widetilde{e}_n\|^2_{\H}\big]\\
		&\quad+\frac{C}{\Delta t}\mathbb{E}\big[\|u(t_{n+1})-e^{-\sigma_0\Delta t}u(t_n)\|^2_{\H}\|\Delta\overline{W}_1^n\|^2_{H^{\gamma}(D)}\|\Delta W_1^n\|^2_{H^{\gamma}(D)}\big]\\
		&\leq \frac{\sigma_0}{16}\Delta t e^{-2\sigma_0\Delta t}\mathbb{E}\big[\|\widetilde{e}_n\|^2_{\H}\big]+C\Delta t^2.
	\end{align*}
	
	For the term $I_{4,4}$, one has
	\begin{align*}
		I_{4,4}&\leq C\mathbb{E}\Big[\Big\|\int_{t_n}^{t_{n+1}}\widehat{S}(t_{n+1}-r)(F_{Q_1}u(r))\dd r\Big\|_{\H}\|\Delta\overline{W}_1^n\|_{H^{\gamma}(D)}e^{-\sigma_0\Delta t}\|\widetilde{e}_n\|_{\H}\Big]\\
		&\leq \frac{\sigma_0}{16}\Delta te^{-2\sigma_0\Delta t}\mathbb{E}\big[\|\widetilde{e}_n\|^2_{\H}\big]+C\Delta t^2,
	\end{align*}
	where we use Sobolev embedding $H^{\gamma}(D)\hookrightarrow L^{\infty}(D)$ for $\gamma>\frac{3}{2}$, Lemma \ref{opee}(1), Lemma \ref{ooq}(1) and Proposition \ref{H-boundedness}.
	Altogether,
	\begin{align*}
		I_4\leq \frac{\sigma_0}{4}\Delta te^{-2\sigma_0\Delta t}\mathbb{E}\big[\|\widetilde{e}_n\|^2_{\H}\big]+C\Delta t(\Delta t+h).
	\end{align*}

	For the term $I_5$, we use \eqref{ee1} to obtain
	\begin{align*}
		I_5&\leq C\mathbb{E}\big[\|\widehat{e}_{n+1}+e^{-\sigma_0\Delta t}\widehat{e}_n\|_{\H}\|\Delta\overline{W}_1^n\|^2_{H^{\gamma}(D)}\|\widetilde{e}_{n+1}+\widehat{e}_{n+1}+e^{-\sigma_0\Delta t}\widetilde{e}_n+e^{-\sigma_0\Delta t}\widehat{e}_n\|_{\H}\big]\\
		&\leq C\mathbb{E}\big[\|\widehat{e}_{n+1}+e^{-\sigma_0\Delta t}\widehat{e}_n\|_{\H}^2\|\Delta\overline{W}_1^n\|^2_{H^{\gamma}(D)}\big]+\frac{\sigma_0}{4}\Delta t\mathbb{E}\big[\|\widetilde{e}_{n+1}\|^2_{\H}\big]\\
		&\quad+\frac{\sigma_0}{4}\Delta te^{-2\sigma_0\Delta t}\mathbb{E}\big[\|\widetilde{e}_n\|^2_{\H}\big]
		+\frac{C}{\Delta t}\mathbb{E}\big[\|\widehat{e}_{n+1}+e^{-\sigma_0\Delta t}\widehat{e}_n\|^2_{\H}\|\Delta\overline{W}_1^n\|^4_{H^{\gamma}(D)}\big]\\
		&\leq \frac{\sigma_0}{4}\Delta te^{-2\sigma_0\Delta t}\mathbb{E}\big[\|\widetilde{e}_n\|^2_{\H}\big]+\frac{\sigma_0}{4}\Delta t\mathbb{E}\big[\|\widetilde{e}_{n+1}\|_{\H}^2\big]+C\Delta t\big(\Delta t+h\big).
	\end{align*}
	
	Combining $I_4$ and $I_5$, we have
	\begin{align*}
		\mathbb{E}\big[\|\widetilde{e}_{n+1}\|_{\H}^2\big]&\leq e^{-2\sigma_0\Delta t}\mathbb{E}\big[\|\widetilde{e}_n\|_{\H}^2\big]+\frac{\sigma_0}{2}\Delta te^{-2\sigma_0\Delta t}\mathbb{E}\big[\|\widetilde{e}_n\|_{\H}^2\big]\\
		&\quad+\frac{\sigma_0}{4} \Delta t\mathbb{E}\big[\|\widetilde{e}_{n+1}\|_{\H}^2\big]+C\Delta t\big(\Delta t+h\big).
	\end{align*}
	There exists a $\Delta t^*:=\frac{1}{\sigma_0}$ such that for any $\Delta t\in(0,\Delta t^*]$,
	$$\sup_{n\in\mathbb{N}}\mathbb{E}\big[\|\widetilde{e}_n\|^2_{\H}\big]\leq C\big(\Delta t+h\big).$$
	Therefore, it concludes that
	\begin{align*}
		\sup_{n\in\mathbb{N}}\mathbb{E}\big[\|e^{full}_n\|^2_{\H}\big]\leq 2\sup_{n\in\mathbb{N}}\mathbb{E}\big[\|\widetilde{e}_n\|_{\H}^2\big]+2\sup_{n\in\mathbb{N}}\mathbb{E}\big[\|\widehat{e}_n\|^2_{\H}\big]\leq C\big(\Delta t+h\big).
	\end{align*}
	The proof is thus finished.
\end{proof}

Similar to Proposition \ref{SMEErgo}(ii), we state the following corollary.
\begin{coro}
	Under the conditions in Proposition \ref{efs} and Theorem \ref{dgorder}, there exists a postive constant $C_{13}$ such that
	\begin{equation*}
		\mathcal{W}_2(\pi^{*},\pi^{\Delta t,h})\leq C_{13}(\Delta t^{\frac{1}{2}}+h^{\frac{1}{2}}).
	\end{equation*}
\end{coro}
\begin{proof}
	By Proposition \ref{SMEErgo} and Theorem \ref{dgorder}, we have
	\begin{align*}
		\mathcal{W}_2(\pi^*,\pi^{\Delta t,h})&\leq\mathcal{W}_2((P^{n}_h)^*\pi^{\Delta t,h},P^*_{t_n}\pi^{\Delta t,h})+ \mathcal{W}_2(P^*_{t_n}\pi^{\Delta t,h},P^*_{t_n}\pi^{*})\\
		&\leq C\big(\Delta t^{\frac{1}{2}}+h^{\frac{1}{2}}\big)+e^{-\sigma_0t_n}\mathcal{W}_2(\pi^*,\pi^{\Delta t,h}),
	\end{align*}
	which gives the desired result by letting $n\rightarrow\infty$. 
\end{proof}

\subsection{Ergodic FD full discretization}	
For the temporal semi-discretization \eqref{ns}, we can apply many kinds of numerical methods to discretize the spatial direction to obtain full discretizations. In this part, we use a finite difference (FD) method to discretize the temporal semi-discretization in space.

We introduce a uniform partition with $\Delta x$, $\Delta y$ and $\Delta z$ being the mesh sizes in $x$, $y$ and $z$ directions,
respectively. For $i=0,1,\ldots,I_x$, $j=0,1,\ldots,J_y$ and $k=0,1,\ldots,K_z$, we define $x_{i}=x_{L}+i\Delta x$, $y_{j}=y_{L}+j\Delta y$,
$z_{k}=z_{L}+k\Delta z$, and $Z_{i,j,k}^{n}$ is the approximation of $Z(t,x,y,z)$ at node $(t_n,x_i,y_j,z_k)$ and let $\sigma_{i,j,k}=\sigma(x_i,y_j,z_k)$.
Denote 
{\small\begin{align*}
		\delta_xZ_{i,j,k}^{n}=\frac{Z_{i+1,j,k}^{n}-Z_{i,j,k}^n}{\Delta x},\quad
		\delta_yZ_{i,j,k}^{n}=\frac{Z_{i,j+1,k}^{n}-Z_{i,j,k}^n}{\Delta y},\quad
		\delta_zZ_{i,j,k}^{n}=\frac{Z_{i,j,k+1}^{n}-Z_{i,j,k}^n}{\Delta z}.
\end{align*}}

We now propose the following full discretization for \eqref{SMEs with PML} by applying the midpoint method to \eqref{ns} in the spatial direction:
\begin{subequations}\label{s1}
	\begin{align}
		\delta^{\sigma_{i,j,k}}_t(E_1)_{\bar{i},\bar{j},\bar{k}}^{n}&=\delta_{y}A_t^{\sigma_{i,j,k}}(H_3)_{\bar{i},j,\bar{k}}^{n}-\delta_{z}A_t^{\sigma_{i,j,k}}(H_2)_{\bar{i},\bar{j},k}^{n}\notag\\
		&\quad-\lambda_1 A_t^{\sigma_{i,j,k}}(H_1)_{\bar{i},\bar{j},\bar{k}}^{n}(\dot{\overline{W}}_1)_{\bar{i},\bar{j},\bar{k}}^n
		+\lambda_2^{(1)}(\dot{W}_2)_{\bar{i},\bar{j},\bar{k}}^n,\label{sub1}\\
		\delta^{\sigma_{i,j,k}}_t(E_2)_{\bar{i},\bar{j},\bar{k}}^{n}&=\delta_{z}A_t^{\sigma_{i,j,k}}(H_1)_{\bar{i},\bar{j},k}^{n}-\delta_{x}A_t^{\sigma_{i,j,k}}(H_3)_{i,\bar{j},\bar{k}}^{n}\notag\\
		&\quad-\lambda_1 A_t^{\sigma_{i,j,k}}(H_2)_{\bar{i},\bar{j},\bar{k}}^{n}(\dot{\overline{W}}_1)_{\bar{i},\bar{j},\bar{k}}^n
		+\lambda_2^{(2)}(\dot{W}_2)_{\bar{i},\bar{j},\bar{k}}^n,\label{sub2}\\
		\delta^{\sigma_{i,j,k}}_t(E_3)_{\bar{i},\bar{j},\bar{k}}^{n}&=\delta_{x}A_t^{\sigma_{i,j,k}}(H_2)_{i,\bar{j},\bar{k}}^{n}-\delta_{y}A_t^{\sigma_{i,j,k}}(H_1)_{\bar{i},j ,\bar{k}}^{n}\notag\\
		&\quad-\lambda_1 A_t^{\sigma_{i,j,k}}(H_3)_{\bar{i},\bar{j},\bar{k}}^{n}(\dot{\overline{W}}_1)_{\bar{i},\bar{j},\bar{k}}^n
		+\lambda_2^{(3)}(\dot{W}_2)_{\bar{i},\bar{j},\bar{k}}^n,\label{sub3}\\
		\delta^{\sigma_{i,j,k}}_t(H_1)_{\bar{i},\bar{j},\bar{k}}^{n}&=\delta_{z}A_t^{\sigma_{i,j,k}}(E_2)_{\bar{i},\bar{j},k}^{n}-\delta_{y}A_t^{\sigma_{i,j,k}}(E_3)_{\bar{i},j ,\bar{k}}^{n}\notag\\
		&\quad+\lambda_1 A_t^{\sigma_{i,j,k}}(E_1)_{\bar{i},\bar{j},\bar{k}}^{n}(\dot{\overline{W}}_1)_{\bar{i},\bar{j},\bar{k}}^n
		+\lambda_2^{(1)}(\dot{W}_2)_{\bar{i},\bar{j},\bar{k}}^n,\label{sub4}\\
		\delta^{\sigma_{i,j,k}}_t(H_2)_{\bar{i},\bar{j},\bar{k}}^{n}&=\delta_{x}A_t^{\sigma_{i,j,k}}(E_3)_{i,\bar{j},\bar{k}}^{n}-\delta_{z}A_t^{\sigma_{i,j,k}}(E_1)_{\bar{i},\bar{j} ,k}^{n}\notag\\
		&\quad+\lambda_1 A_t^{\sigma_{i,j,k}}(E_2)_{\bar{i},\bar{j},\bar{k}}^{n}(\dot{\overline{W}}_1)_{\bar{i},\bar{j},\bar{k}}^n
		+\lambda_2^{(2)}(\dot{W}_2)_{\bar{i},\bar{j},\bar{k}}^n,\label{sub5}\\
		\delta^{\sigma_{i,j,k}}_t(H_3)_{\bar{i},\bar{j},\bar{k}}^{n}&=\delta_{y}A_t^{\sigma_{i,j,k}}(E_1)_{\bar{i},j,\bar{k}}^{n}-\delta_{x}A_t^{\sigma_{i,j,k}}(E_2)_{i,\bar{j} ,\bar{k}}^{n}\notag\\
		&\quad+\lambda_1 A_t^{\sigma_{i,j,k}}(E_3)_{\bar{i},\bar{j},\bar{k}}^{n}(\dot{\overline{W}}_1)_{\bar{i},\bar{j},\bar{k}}^n
		+\lambda_2^{(3)}(\dot{W}_2)_{\bar{i},\bar{j},\bar{k}}^n,\label{sub6}
	\end{align}
\end{subequations}
where $\bar{i}=i+\frac{1}{2}$, $\bar{j}=j+\frac{1}{2}$, $\bar{k}=k+\frac{1}{2}$ and
\begin{align*}
	(\dot{\overline{W}}_{1})_{\bar{i},\bar{j},\bar{k}}^{n}&:=\frac{(\Delta \overline{W}_{1})_{\bar{i},\bar{j},\bar{k}}^{n}}{\Delta t}=\frac{\overline{W}_{1}(t_{n+1},x_{\bar{i}},y_{\bar{j}},z_{\bar{k}})-\overline{W}_{1}(t_{n},x_{\bar{i}},y_{\bar{j}},z_{\bar{k}})}{\Delta t},\\
	(\dot{W}_{2})_{\bar{i},\bar{j},\bar{k}}^{n}&:=\frac{(\Delta W_{2})_{\bar{i},\bar{j},\bar{k}}^{n}}{\Delta t}=\frac{W_{2}(t_{n+1},x_{\bar{i}},y_{\bar{j}},z_{\bar{k}})-W_{2}(t_{n},x_{\bar{i}},y_{\bar{j}},z_{\bar{k}})}{\Delta t}.
\end{align*}

Denote the discrete energy by 
\begin{align}\label{disave}
	\Phi(t_{n}):=\Delta x\Delta y\Delta z\sum_{i,j,k}\Big(|\Ee_{\bar{i},\bar{j},\bar{k}}^{n}|^{2}+|\Hh_{\bar{i},\bar{j},\bar{k}}^{n}|^{2}\Big),\quad n\in\mathbb{N}.
\end{align}
Similarly to \cite[Theorem 3.2]{CHZ2016}, we obtain the following discrete energy evolution law for \eqref{sub1}--\eqref{sub6}:
\begin{align}\label{energy1}
	\Phi(t_{n+1})&=\Delta x\Delta y\Delta z\sum_{i,j,k}e^{-2\sigma_{i,j,k}\Delta t}\Big(|\Ee_{\bar{i},\bar{j},\bar{k}}^{n}|^{2}+|\Hh_{\bar{i},\bar{j},\bar{k}}^{n}|^{2}\Big)\notag\\
	&\quad+2\Delta x\Delta y\Delta z\sum_{i,j,k}\Big[\Upsilon^{n}_{\bar{i},\bar{j},\bar{k}}
	(\Delta W_{2})_{\bar{i},\bar{j},\bar{k}}^{n}\Big]
\end{align}
under the periodic boundary condition.
Here,
$\Upsilon^{n}_{\bar{i},\bar{j},\bar{k}}:=\widetilde{{\bm{\lambda}}}_2\cdot A_t^{\sigma_{i,j,k}}\big((\Ee_{\bar{i},\bar{j},\bar{k}}^{n})^\top,(\Hh_{\bar{i},\bar{j},\bar{k}}^{n})^\top\big)^\top.$

We now investigate the ergodicity and stochastic multi-symplecticity of \eqref{s1}.
\subsubsection{Ergodicity} To get the ergodicity, we first establish the uniform boundedness of the averaged discrete energy in the following proposition.
\begin{prop}\label{full_bound}
	Assume that ${\bf E}_0,{\bf H}_0 \in L^2(\Omega; L^2(D)^3)$, $q_m\in C^{1}(D), m\in\mathbb{N}$ and let Assumption \ref{ass} hold with $\gamma_1\geq\gamma$ and $\gamma_2\geq 1+\gamma$ for $\gamma>\frac{3}{2}$. Then there exists a positive constant $\Delta t^*$ such that when $\Delta t\in (0,\Delta t^* ]$, the averaged discrete energy is uniformly bounded under the periodic boundary condition, i.e.,
	\begin{equation*}
		{\mathbb E}\big[\Phi(t_{n})\big]\leq e^{-\sigma_0 n\Delta t}{\mathbb E}\big[\Phi(t_{0})\big]+C_{14},\quad n\in\mathbb{N},
	\end{equation*}
	where the positive constant $C_{14}$ depends on $\sigma_0,\lambda_{1},\widetilde{{\bm{\lambda}}}_2,|D|$, $\|Q_1^{\frac{1}{2}}\|_{\mathcal{L}_2^{\gamma_1}}$ and  $\|Q^{\frac{1}{2}}_2\|_{\mathcal{L}_2^{\gamma_2}}.$
\end{prop}

\begin{proof}
	For the first term on the right side of \eqref{energy1}, we have 
	\begin{align}\label{efz}
		\Delta x\Delta y\Delta z\sum_{i,j,k}e^{-2\sigma_{i,j,k}\Delta t}\big(|{\textbf{E}}_{\bar{i},\bar{j},\bar{k}}^{n}|^{2}+|{\textbf{H}}_{\bar{i},\bar{j},\bar{k}}^{n}|^{2}\big)\leq e^{-2\sigma_0\Delta t}\Phi(t_{n}).
	\end{align}
	For the second term on the right side of \eqref{energy1}, which contains six sub-terms, the estimate is more technique. We consider the first sub-term. It holds that
	{\small\begin{equation}\label{hnh}
			\begin{split}
				&{\mathbb E}\bigg[\sum_{i,j,k}
				A_t^{\sigma_{i,j,k}}(E_1)_{\bar{i},\bar{j},\bar{k}}^{n}(\Delta W_{2})_{\bar{i},\bar{j},\bar{k}}^{n}\bigg]
				=\frac12{\mathbb E}\Bigg[\sum_{i,j,k}
				\Big[(E_1)_{\bar{i},\bar{j},\bar{k}}^{n+1}-e^{-\sigma_{i,j,k}\Delta t}(E_1)_{\bar{i},\bar{j},\bar{k}}^{n}\Big](\Delta W_{2})_{\bar{i},\bar{j},\bar{k}}^{n}\Bigg].
			\end{split}
	\end{equation}}
	Substituting \eqref{sub1} into \eqref{hnh}, we obtain
	\begin{align}\label{eq2.9}
		&{\mathbb E}\Big[\sum_{i,j,k}
		A_t^{\sigma_{i,j,k}}(E_1)_{\bar{i},\bar{j},\bar{k}}^{n}(\Delta W_{2})_{\bar{i},\bar{j},\bar{k}}^{n}\Big]\notag\\
		&=\frac12{\mathbb E}\Big[\sum_{i,j,k}
		\Big[-\Delta t
		A_t^{\sigma_{i,j,k}}(H_3)_{\bar{i},\bar{j},\bar{k}}^{n}\delta_{y}(\Delta W_{2})_{\bar{i},j,\bar{k}}^{n}+\Delta tA_t^{\sigma_{i,j,k}}(H_2)_{\bar{i},\bar{j},\bar{k}}^{n}\delta_{z}(\Delta W_{2})_{\bar{i},\bar{j},k}^{n}\Big]\Big]\notag\\
		&\quad+\frac12{\mathbb E}\Big[\sum_{i,j,k}
		\Big[-\lambda_1 A_t^{\sigma_{i,j,k}}(H_1)_{\bar{i},\bar{j},\bar{k}}^{n}(\Delta W_{1})_{\bar{i},\bar{j},\bar{k}}^n(\Delta W_{2})_{\bar{i},\bar{j},\bar{k}}^{n}
		+\lambda_2^{(1)}\big[(\Delta W_{2})_{\bar{i},\bar{j},\bar{k}}^n\big]^2
		\Big]\Big],
	\end{align}
	where we use the fact  
	\begin{align*}
		&\hphantom{=}\sum_{i,j,k}
		\delta_{y}A_t^{\sigma_{i,j,k}}(H_3)_{\bar{i},j,\bar{k}}^{n}(\Delta W_{2})_{\bar{i},\bar{j},\bar{k}}^n
		=-\sum_{i,j,k}A_t^{\sigma_{i,j,k}}(H_3)_{\bar{i},\bar{j},\bar{k}}^{n}\delta_{y}(\Delta W_{2})_{\bar{i},j,\bar{k}}^n
	\end{align*}
	and
	\begin{align*}
		&\hphantom{=}\sum_{i,j,k}
		\delta_{z}A_t^{\sigma_{i,j,k}}(H_2)_{\bar{i},\bar{j},k}^{n}(\Delta W_{2})_{\bar{i},\bar{j},\bar{k}}^n
		=-\sum_{i,j,k}A_t^{\sigma_{i,j,k}}(H_2)_{\bar{i},\bar{j},\bar{k}}^{n}\delta_{z}(\Delta W_{2})_{\bar{i},\bar{j},k}^n
	\end{align*}
	due to the periodic boundary condition. 
	Notice that 
	\begin{align}\label{hg1}
		&\Delta x\Delta y\Delta z\sum_{i,j,k}\mathbb{E}\big[|(\Delta W_2)^n_{\bar{i},\bar{j},\bar{k}}|^2\big]
		\notag\\
		&=\Delta x\Delta y\Delta z\sum_{i,j,k}\mathbb{E}\Big[\Big|\sum_{m=1}^{\infty}\sqrt{\eta_{m}^{(2)}}q_m(x_{\bar{i}},y_{\bar{j}},z_{\bar{k}})(\beta_m^{(2)}(t_{n+1})-\beta_m^{(2)}(t_n))\Big|^2\Big]\notag\\
		&\leq \Delta x\Delta y\Delta z\Delta t\sum_{i,j,k}\sum_{m=1}^{\infty}\eta_{m}^{(2)}\|q_m\|_{L^{\infty}(D)}^2\\
		&\leq C|D|\Delta t\sum_{m=1}^{\infty}\eta_{m}^{(2)}\|q_m\|_{H^{\gamma}(D)}^2=C\Delta t\notag
	\end{align}
	and similarly
	\begin{align}\label{hg2}
		&\Delta x\Delta y\Delta z\sum_{i,j,k}\Big(\mathbb{E}\Big[\big|(\Delta W_1)^n_{\bar{i},\bar{j},\bar{k}}\big|^2\Big]\mathbb{E}\Big[\big|(\Delta W_2)^n_{\bar{i},\bar{j},\bar{k}}\big|^2\Big]\Big)
		\leq C\Delta t^2,
	\end{align}
	where we use the Sobolev embedding $H^{\gamma}(D)\hookrightarrow L^{\infty}(D)$ for $\gamma>\frac{3}{2}$.
	Thus, for \eqref{eq2.9}, the H\"older inequality, the Young inequality and \eqref{hg1}--\eqref{hg2} lead to
	{\small	\begin{align*}
			&\Delta x\Delta y\Delta z\lambda_2^{(1)}{\mathbb E}\Big[\sum_{i,j,k}
			A_t^{\sigma_{i,j,k}}(E_1)_{\bar{i},\bar{j},\bar{k}}^{n}(\Delta W_{2})_{\bar{i},\bar{j},\bar{k}}^{n}\Big]\\
			&\leq \frac{\sigma_0}{8}\Delta t\Delta x\Delta y\Delta z{\mathbb E}\Big[\sum_{i,j,k}\Big[\big|A_t^{\sigma_{i,j,k}}(H_3)_{\bar{i},\bar{j},\bar{k}}^{n}\big|^2+\big|A_t^{\sigma_{i,j,k}}(H_2)_{\bar{i},\bar{j},\bar{k}}^{n}\big|^2
			\Big]\Big]\\
			&\quad+\frac{\big|\lambda_2^{(1)}\big|^2}{2\sigma_0}\Delta t\Delta x\Delta y\Delta z{\mathbb E}\Big[\sum_{i,j,k}\Big[\big|\delta_{y}(\Delta W_{2})_{\bar{i},j,\bar{k}}^n\big|^2+\big|\delta_{z}(\Delta W_{2})_{\bar{i},\bar{j},k}^n\big|^2\Big]\Big]\\
			&\quad+\frac12\Delta x\Delta y\Delta z\big|\lambda_2^{(1)}\big|\sum_{i,j,k}{\mathbb E}\Big[\big|(\Delta W_{2})_{\bar{i},\bar{j},\bar{k}}^{n}\big|^2\Big]+\frac{\sigma_0}{4}\Delta t\Delta x\Delta y\Delta z{\mathbb E}\Big[\sum_{i,j,k}\big|A_t^{\sigma_{i,j,k}}(H_1)_{\bar{i},\bar{j},\bar{k}}^{n}\big|^2\Big]\\
			&\quad+\frac{\big|\lambda_1\lambda_2^{(1)}\big|^2}{4\sigma_0\Delta t}\Delta x\Delta y\Delta z
			\sum_{i,j,k}{\mathbb E}\Big[\big|(\Delta W_{1})_{\bar{i},\bar{j},\bar{k}}^n\big|^2\Big]{\mathbb E}\Big[\big|(\Delta W_{2})_{\bar{i},\bar{j},\bar{k}}^{n}\big|^2\Big]\\
			&\leq \frac{\sigma_0}{4} \Delta t\Delta x\Delta y\Delta z{\mathbb E}\Big[\sum_{i,j,k}
			\Big[ \frac14 \big|(H_3)_{\bar{i},\bar{j},\bar{k}}^{n+1}\big|^2+ \frac14 |(H_2)_{\bar{i},\bar{j},\bar{k}}^{n+1}\big|^2+\frac12 |(H_1)_{\bar{i},\bar{j},\bar{k}}^{n+1}\big|^2\Big]\Big]\\
			&\hphantom{=}+\frac{\sigma_0}{4}\Delta t\Delta x\Delta y\Delta z{\mathbb E}\Big[\sum_{i,j,k}e^{-2\sigma_{i,j,k}\Delta t}
			\Big[ \frac14 \big|(H_3)_{\bar{i},\bar{j},\bar{k}}^{n}\big|^2+ \frac14 |(H_2)_{\bar{i},\bar{j},\bar{k}}^{n}|^2+\frac12 \big|(H_1)_{\bar{i},\bar{j},\bar{k}}^{n}\big|^2\Big]\Big]\\
			&\hphantom{=}+\frac{\big|\lambda_2^{(1)}\big|^2}{2\sigma_0}\Delta t\Delta x\Delta y\Delta z{\mathbb E}\Big[\sum_{i,j,k}\Big[\big|\delta_{y}(\Delta W_{2})_{\bar{i},j,\bar{k}}^n\big|^2+\big|\delta_{z}(\Delta W_{2})_{\bar{i},\bar{j},k}^n\big|^2\Big]\Big]+C\Delta t.
	\end{align*}}
	Applying the similar approach to estimate the other five sub-terms on the second term on the right side of \eqref{energy1}, we get
	
	\begin{align*}
		&\Delta x\Delta y\Delta z\sum_{i,j,k}\mathbb{E}\Big[\Upsilon^{n}_{\bar{i},\bar{j},\bar{k}}
		(\Delta W_{2})_{\bar{i},\bar{j},\bar{k}}^{n}\Big]\\
		&\leq \frac{\sigma_0}{4}\Delta t {\mathbb E}[\Phi(t_{n+1})]+\frac{\sigma_0}{4} \Delta t e^{-2\sigma_0\Delta t}{\mathbb E}[\Phi(t_{n})]+C\Delta t\\
		&\quad+\frac{|{\bm\lambda_2}|^2}{\sigma_0}\Delta t\Delta x\Delta y\Delta z{\mathbb E}\Big[\sum_{i,j,k}\Big[|\delta_{x}(\Delta W_{2})_{i,\bar{j},\bar{k}}^n|^2+
		|\delta_{y}(\Delta W_{2})_{\bar{i},j,\bar{k}}^n|^2
		+|\delta_{z}(\Delta W_{2})_{\bar{i},\bar{j},k}^n|^2\Big]\Big].
	\end{align*}
	It follows from the Sobolev embedding $H^{\gamma}\hookrightarrow L^{\infty}(D)$ for $\gamma>\frac{3}{2}$ that
	\begin{align*}
		&{\mathbb E}\Big[\sum_{i,j,k}\big|\delta_{x}(\Delta W_{2})_{i,\bar{j},\bar{k}}^n\big|^2\bigg]\\
		&=\sum_{i,j,k}{\mathbb E}\Big[\Big|
		\sum_{m=1}^{\infty}\sqrt{\eta_m^{(2)}}\frac{q_{m}(x_{i+1},y_{\bar{j}},z_{\bar{k}})-q_{m}(x_i,y_{\bar{j}},z_{\bar{k}})}{\Delta x}\Big(\beta_m^{(2)}(t_{n+1})-\beta_m^{(2)}(t_{n})\Big)\Big|^2\Big]\\
		&=\Delta t\sum_{i,j,k}
		\sum_{m=1}^{\infty}\eta_m^{(2)}\big|\partial_xq_{m}(x_i+\theta_i\Delta x,y_{\bar{j}},z_{\bar{k}})\big|^2\leq \Delta t\sum_{i,j,k}\sum_{m=1}^{\infty}\eta_{m}^{(2)}\|q_m\|^2_{W^{1,\infty}(D)}\\
		&\leq C\Delta t\sum_{i,j,k}\sum_{m=1}^{\infty}\eta_{m}^{(2)}\|q_m\|^2_{H^{\gamma_2}(D)}\leq C\Delta t I_xJ_yK_z,
	\end{align*}
	where $\theta_i\in[0,1]$. Therefore, we get
	\begin{align*}
		&\frac{|{\bm\lambda_2}|^2}{4\alpha}\Delta t\Delta x\Delta y\Delta z{\mathbb E}\Big[\sum_{i,j,k}|\delta_{x}(\Delta W_{2})_{i,\bar{j},\bar{k}}^n|^2\Big]
		\leq C\Delta t^2.
	\end{align*}
	Estimating terms concerning $|\delta_{y}(\Delta W_2)^n_{\bar{i},j,\bar{k}}|^2$ and $|\delta_{z}(\Delta W_2)^n_{\bar{i},\bar{j},k}|^2$ similarly, we conclude that
	{\small\begin{align}\label{wmd}
			\Delta x\Delta y\Delta z\sum_{i,j,k}\mathbb{E}\Big[\Upsilon^{n}_{\bar{i},\bar{j},\bar{k}}
			(\Delta W_{2})_{\bar{i},\bar{j},\bar{k}}^{n}\Big]\leq \frac{\sigma_0}{4} \Delta t {\mathbb E}[\Phi(t_{n+1})]+\frac{\sigma_0}{4}\Delta t  e^{-2\sigma_0\Delta t}{\mathbb E}[\Phi(t_{n})]+C\Delta t.
	\end{align}}
	Consequently, combining \eqref{efz} and \eqref{wmd} yields 
	\begin{equation*}
		{\mathbb E}[\Phi(t_{n+1})]\leq e^{-2\sigma_0\Delta t}{\mathbb E}[\Phi(t_{n})]+\frac{\sigma_0}{4}\Delta t  {\mathbb E}[\Phi(t_{n+1})]+\frac{\sigma_0}{4}\Delta t e^{-2\sigma_0\Delta t}{\mathbb E}[\Phi(t_{n})]+C\Delta t.
	\end{equation*}
	By the Gronwall inequality, it can be shown that for any $\Delta t\in(0,\frac{2}{\sigma_0}]$, the conclusion of the proposition holds. 
\end{proof}

\begin{rema}
	Especially, if the orthonormal basis of $L^2(D)$ is chosen as
	$$q_m(x,y,z)=q_{m_1}(x)q_{m_2}(y)q_{m_3}(z),\quad m=(m_1,m_2,m_3)\in\mathbb{N}^3$$
	with
	$$q_{m_1}(x)=\sqrt{\frac{2}{x_R-x_L}}\sin\Big(\frac{2m_1\pi (x-x_L)}{x_R-x_L}\Big),$$ 
	and $q_{m_2}(y),q_{m_3}(z)$ being defined similarly, the assumption of Proposition \ref{full_bound} can be weaken as $\gamma_1\geq 0$ and $\gamma_2\geq 1$ since the condition $q_m\in C^1(D)$ is automatically satisfied in this case.
\end{rema}

Let 
\begin{equation*}
	\begin{split}
		U&^n=\Big((E_1)^n_{\bar{1},\bar{1},\bar{1}},(E_1)^n_{\bar{2},\bar{1},\bar{1}},\ldots,(E_1)^n_{\bar{I}_x,\bar{1},\bar{1}},(E_1)^n_{\bar{1},\bar{2},\bar{1}},\ldots,(E_1)^n_{\bar{I}_x,\bar{2},\bar{1}}\ldots,(E_1)^n_{\bar{I}_x,\bar{J}_y,\bar{K}_z},\\
		&~(E_2)^n_{\bar{1},\bar{1},\bar{1}},\ldots,(E_2)^n_{\bar{I}_x,\bar{J}_y,\bar{K}_z},(E_3)_{\bar{1},\bar{1},\bar{1}}^n,\ldots,(E_3)^n_{\bar{I}_x,\bar{J}_y,\bar{K}_z},(H_1)^n_{\bar{1},\bar{1},\bar{1}}\ldots,(H_3)^n_{\bar{I}_x,\bar{J}_y,\bar{K}_z}\Big)^\top,
	\end{split}
\end{equation*}
then the discrete energy $\Phi$ given by \eqref{disave} can be rewritten into
\begin{align}\label{disave2}
	\Phi(t_{n})=\Delta x\Delta y\Delta z |U^n|^2,\quad n\in\mathbb{N}.
\end{align}
By taking \eqref{disave2} as a Lyapunov function, the ergodicity of \eqref{s1} can be proved similar to that of Proposition \ref{SMEErgo}, which is stated below.
\begin{theo}
	Under the conditions in Proposition \ref{full_bound}, for sufficiently small $\Delta t>0$, the numerical solution $\{U^n\}_{n\in\mathbb{N}}$ of \eqref{s1} has a unique invariant measure $\pi$. Thus $\{U^n\}_{n\in\mathbb{N}}$ is ergodic. Moreover, $\{U^n\}_{n\in\mathbb{N}}$ is exponentially mixing.
\end{theo}	

\subsubsection{Stochastic conformal multi-symplecticity}
Now we turn to the stochastic conformal multi-symplecticity of \eqref{s1}. By using the definitions of skew-symmetric matrices $F,K_s$, $s=1,2,3$, the full discretization \eqref{s1} can be read as
\begin{equation*}
	\begin{split}
		F\left(\delta_{t}^{\sigma_{i,j,k}}u_{\bar{i},\bar{j},\bar{k}}^{n}\right)+K_1\left(\delta_{x}A_{t}^{\sigma_{i,j,k}}u_{i,\bar{j},\bar{k}}^{n}\right)
		+K_2\left(\delta_{y}A_{t}^{\sigma_{i,j,k}}u_{\bar{i},j,\bar{k}}^{n}\right)
		+K_3\left(\delta_{z}A_{t}^{\sigma_{i,j,k}}u_{\bar{i},\bar{j},k}^{n}\right)\\
		=\nabla_u S_1\left(A_{t}^{\sigma_{i,j,k}}u_{\bar{i},\bar{j},\bar{k}}^{n}\right)(\dot{\overline{W}}_1)_{\bar{i},\bar{j},\bar{k}}^{n}+\nabla_u S_2\left(A_{t}^{\sigma_{i,j,k}}u_{\bar{i},\bar{j},\bar{k}}^{n}\right)(\dot{W}_2)_{\bar{i},\bar{j},\bar{k}}^{n}.
	\end{split}
\end{equation*}
Performing the wedge product on both sides of the above equation yields the following result.
\begin{theo}
	The full discretization \eqref{s1} preserves the discrete stochastic conformal multi-symplectic conservation law
	\begin{equation*}
		\begin{split}
			\delta_t^{2\sigma_{i,j,k}}\omega_{\bar{i},\bar{j},\bar{k}}^{n}+\delta_x (\kappa_1)^{n,\sigma_{i,j,k}}_{i,\bar{j},\bar{k}}+\delta_y (\kappa_2)^{n,\sigma_{i,j,k}}_{\bar{i},j,\bar{k}}+
			\delta_z (\kappa_3)^{n,\sigma_{i,j,k}}_{\bar{i},\bar{j},k}=0,
			\quad\mathbb{P}\text{-}a.s.,
		\end{split}
	\end{equation*}
	where $\omega^n_{i,j,k}=\frac12 du^{n}_{i,j,k}\wedge Fdu^{n}_{i,j,k}$, and $(\kappa_s)^{n,\sigma_{i,j,k}}_{i,j,k}=\frac12d(A_{t}^{\sigma_{i,j,k}}u^{n}_{i,j,k})\wedge K_s d(A_{t}^{\sigma_{i,j,k}}u^{n}_{i,j,k})$, $s=1,2,3.$ 
\end{theo}

\appendix

\section{The Proof of Lemma \ref{opee}}
To derive the estimates of the operators, we introduce the following two deterministic systems:
\begin{equation}\label{system_1}
	\left\{
	\begin{split}
		&\dd \widetilde{u}(t)=(M-\sigma I)\widetilde{u}(t)\dd t,\quad t>0,\\
		&\widetilde{u}(0)=v,
	\end{split}		
	\right.
\end{equation}
and
\begin{equation}\label{system_2}
	\left\{
	\begin{split}
		&\dd \overline{u}(t)=-\sigma\overline{u}(t)\dd t,\quad t>0,\\
		&\overline{u}(0)=v.
	\end{split}		
	\right.
\end{equation}

(1)	
Notice that $\widetilde{u}(t)=e^{t(M-\sigma I)}v$, $t\geq0$. Then
\begin{align}\label{ghg}
	\|\widetilde{u}(t)\|_{\H}\leq \|e^{tM}\|_{\mathcal{L}(\H,\H)}\|e^{-\sigma I}v\|_{\H}\leq e^{-\sigma_0t}\|v\|_{\H},
\end{align}
which leads to the first assertion of (1).

Let $\widetilde{u}^n:=[(I-\frac{\Delta t}{2}M)^{-1}(I+\frac{\Delta t}{2}M)e^{-\sigma\Delta t}]^nv,n\in\mathbb{N}$. Then 
$$\widetilde{u}^n=(I-\frac{\Delta t}{2}M)^{-1}(I+\frac{\Delta t}{2}M)e^{-\sigma\Delta t}\widetilde{u}^{n-1}, \quad n\geq1$$ \\
implies 
\begin{equation}\label{re1}
	\widetilde{u}^n-e^{-\sigma\Delta t}\widetilde{u}^{n-1}=\frac{\Delta t}{2}M(\widetilde{u}^n+e^{-\sigma\Delta t}\widetilde{u}^{n-1}).
\end{equation}
We apply $\langle\cdot,\widetilde{u}^n+e^{-\sigma\Delta t}\widetilde{u}^{n-1}\rangle_{\H}$ on both sides of \eqref{re1} and get
$$\|\widetilde{u}^n\|^2_{\H}-\|e^{-\sigma\Delta t}\widetilde{u}^{n-1}\|_{\H}^2=0.$$
Then we have
$$\|\widetilde{u}^n\|_{\H}\leq e^{-\sigma_0\Delta t}\|\widetilde{u}^{n-1}\|_{\H}\leq\cdots\leq e^{-n\Delta t\sigma_0}\|\widetilde{u}^0\|_{\H},$$
which leads to the second assertion of $(1)$.

(2) To estimate $\|\hat{S}(t_n)-(\hat{S}_{\Delta t})^n\|_{\mathcal{L}(D(M),\mathbb{H})}$, we proceed now in two steps.

{\em Step 1. Estimate of $\,\|\widetilde{u}^n\|_{\mathcal{D}(M)}$} and $\|\widetilde{u}(t)\|_{\mathcal{D}(M)}$.

Applying $\langle\cdot,M^2(\widetilde{u}^n+e^{-\sigma\Delta t}\widetilde{u}^{n-1})\rangle_{\H}$ on both sides of \eqref{re1} obtains
$$\|M\widetilde{u}^n\|^2_{\H}-\|M(e^{-\sigma\Delta t}\widetilde{u}^n)\|^2_{\H}=0.$$
Hence
\begin{align*}
	&\|M\widetilde{u}^n\|_{\H}=\|M(e^{-\sigma\Delta t}\widetilde{u}^{n-1})\|_{\H}\\
	&\leq e^{-\sigma_0\Delta t}\|M\widetilde{u}^{n-1}\|_{\H}+2\|\nabla\sigma\|_{L^\infty(D)^3}\Delta te^{-\sigma_0\Delta t}\|\widetilde{u}^{n-1}\|_{\H}\\
	&\leq e^{-n\Delta t\sigma_0}\|M\widetilde{u}^0\|_{\H}+2\|\nabla\sigma\|_{L^\infty(D)^3}\Delta t\Big(\sum_{k=1}^{n}e^{-\sigma_0k\Delta t}e^{-(n-k)\Delta t\sigma_0}\|\widetilde{u}^0\|_{\H}\Big)\\
	&\leq Ce^{-\sigma_0n\Delta t/2}\|\widetilde{u}^0\|_{\mathcal{D}(M)},
\end{align*}
from which we get	
\begin{align}\label{esx}
	\|\widetilde{u}^n\|_{\mathcal{D}(M)}\leq  \Big(Ce^{-2\sigma_0n\Delta t}\|\widetilde{u}^0\|^2_{\H}+Ce^{-\sigma_0n\Delta t}\|\widetilde{u}^0\|^2_{\mathcal{D}(M)}\Big)^{\frac{1}{2}} \leq Ce^{-\sigma_0n\Delta t/2}\|\widetilde{u}^0\|_{\mathcal{D}(M)}.
\end{align}
It follows from
$$\frac{\partial}{\partial t}\|M\widetilde{u}(t)\|^2_{\H}=2\langle M\frac{\partial\widetilde{u}(t)}{\partial t},M\widetilde{u}(t)\rangle_{\mathbb{H}}=-2\langle M(\sigma \widetilde{u}(t)),M\widetilde{u}(t)\rangle_{\H}$$
that
\begin{align*}
	\dd \|M\widetilde{u}(t)\|_{\H}^2&=-2\langle \sigma M\widetilde{u}(t),M\widetilde{u}(t)\rangle_{\H}\dd t-2\Big\langle \begin{pmatrix} 0&\nabla\sigma\times \\
		-\nabla\sigma\times&0 \end{pmatrix}\widetilde{u}(t),M\widetilde{u}(t) \Big\rangle_{\H}\dd t\\
	&\leq -2\sigma_0\|M\widetilde{u}(t)\|_{\H}^2\dd t+4\|M\widetilde{u}(t)\|_{\H}\|\widetilde{u}(t)\|_{\H}\|\nabla\sigma\|_{L^{\infty}(D)^3}\dd t\\
	&\leq -\sigma_0\|M\widetilde{u}(t)\|^2_{\H}\dd t+Ce^{-2\sigma_0t}\|\widetilde{u}(0)\|^2_{\H}\dd t.
\end{align*}
By the Gronwall inequality, we get
\begin{align}\label{ghg1}
	\|M\widetilde{u}(t)\|^2_{\H}&\leq e^{-\sigma_0t}\|M\widetilde{u}(0)\|^2_{\H}+C\int_{0}^{t}e^{-\sigma_0(t-s)}e^{-2\sigma_0s}\|\widetilde{u}(0)\|^2_{\H}\dd s\notag\\
	&\leq Ce^{-\sigma_0t}\|\widetilde{u}(0)\|_{\mathcal{D}(M)}^2.
\end{align}	
Combining \eqref{ghg} and \eqref{ghg1} yields
$$\|\widetilde{u}(t)\|_{\mathcal{D}(M)}\leq Ce^{-\sigma_0t/2}\|\widetilde{u}(0)\|_{\mathcal{D}(M)}.$$

{\em Step 2. Estimate of $\,\|\widetilde{u}(t_{n})-\widetilde{u}^{n}\|_{\H}$}.

Let $\widetilde{e}_n:=\widetilde{u}(t_{n})-\widetilde{u}^{n}$, then
\begin{align}\label{qmx}
	\widetilde{e}_{n+1}-e^{-\sigma\Delta t}\widetilde{e}_n
	&=\big[\widetilde{u}(t_{n+1})-e^{-\sigma\Delta t}\widetilde{u}(t_n)\big]-\big[\widetilde{u}^{n+1}-e^{-\sigma\Delta t}\widetilde{u}^n\big]\notag\\
	&:=\xi^{n+1}+\frac{\Delta t}{2}M\big(\widetilde{e}_{n+1}+e^{-\sigma\Delta t}\widetilde{e}_n\big),
\end{align}
where $\xi^{n+1}:=\widetilde{u}(t_{n+1})-e^{-\sigma\Delta t}\widetilde{u}(t_n)-\frac{\Delta t}{2}M\big(\widetilde{u}(t_{n+1})+e^{-\sigma\Delta t}\widetilde{u}(t_n)\big)$.

Applying $\langle\cdot,\widetilde{e}_{n+1}+e^{-\sigma\Delta t}\widetilde{e}^n\rangle_{\mathbb{H}}$ on both sides of \eqref{qmx} yields
$$\|\widetilde{e}_{n+1}\|^2_{\mathbb{H}}=\|e^{-\sigma\Delta t}\widetilde{e}_n\|^2_{\mathbb{H}}+\langle\xi^{n+1},\widetilde{e}_{n+1}\rangle_{\mathbb{H}}+\langle\xi^{n+1},e^{-\sigma\Delta t}\widetilde{e}_n\rangle_{\mathbb{H}}.$$
Using the fact that $\widetilde{u}(t_{n+1})=e^{-\sigma\Delta t}\widetilde{u}(t_n)+\int_{t_n}^{t_{n+1}}e^{-\sigma(t_{n+1}-s)}M\widetilde{u}(s)\dd s$, we obtain
\begin{align*}
	\langle\xi^{n+1},\widetilde{e}_{n+1}\rangle_{\H}&=\frac{1}{2}\int_{t_n}^{t_{n+1}}\int_{s}^{t_{n+1}}\langle e^{-\sigma(t_{n+1}-r)}M\widetilde{u}(r),M\widetilde{e}_{n+1}\rangle_{\H}\dd r\dd s\\
	&\quad-\int_{t_n}^{t_{n+1}}\langle \widetilde{u}(s),R_{\sigma}^{t_{n+1}-s}\widetilde{e}_{n+1}\rangle_{\H}\dd s\\
	&\quad-\frac{1}{2}\int_{t_n}^{t_{n+1}}\int_{t_n}^{s}\langle e^{-\sigma(t_{n+1}-r)}M\widetilde{u}(r),M\widetilde{e}_{n+1}\rangle_{\H}\dd s\\
	&\leq \frac{1}{2}\int_{t_n}^{t_{n+1}}\int_{s}^{t_{n+1}}\|e^{-\sigma(t_{n+1}-r)}M\widetilde{u}(r)\|_{\H}\|M\widetilde{e}_{n+1}\|_{\H}\dd r\dd s\\
	&\quad+ \frac{1}{2}\int_{t_n}^{t_{n+1}}\int_{t_n}^{s}\|e^{-\sigma(t_{n+1}-r)}M\widetilde{u}(r)\|_{\H}\|M\widetilde{e}_{n+1}\|_{\H}\dd r\dd s\\
	&\quad+2e^{-\sigma_0\Delta t}\|\nabla\sigma\|_{L^\infty(D)^3}\int_{t_n}^{t_{n+1}}(t_{n+1}-s)\|\widetilde{u}(s)\|_{\H}\|\widetilde{e}_{n+1}\|_{\H}\dd s\\
	&\leq Ce^{-\sigma_0n\Delta t}\Delta t^2\|\widetilde{u}^0\|^2_{\mathcal{D}(M)}.
\end{align*}
Similarly, we have
$\langle\xi^{n+1},e^{-\sigma\Delta t}\widetilde{e}_n\rangle_{\H}\leq C\Delta t^2\|\widetilde{u}^0\|^2_{\mathcal{D}(M)}e^{-\sigma_0n\Delta t}.$
Thus,
\begin{align*}
	\|\widetilde{e}_n\|^2_{\H}
	&\leq e^{-2n\Delta t\sigma_0}\|\widetilde{e}_0\|^2_{\H}+C\Delta t^2e^{-\sigma_0(n-1)\Delta t}\|\widetilde{u}^0\|^2_{\mathcal{D}(M)}\frac{1-e^{-2\sigma_0n\Delta t}}{1-e^{-2\sigma_0\Delta t}}\\
	&\leq Ce^{-n\Delta t\sigma_0}\Delta t\|\widetilde{u}^0\|^2_{\mathcal{D}(M)},
\end{align*}
which gives the assertion.

(3) It follows from \eqref{system_1} that
\begin{align*}
	\|(\hat{S}(t)-I)v\|_{\H}&=\|\widetilde{u}(t)-v\|_{\H}\leq \int_{0}^{t}\|(M-\sigma I)\widetilde{u}(s)\|_{\H}\dd s\\
	&\leq (1+\|\sigma\|_{L^\infty(D)})\int_{0}^{t}\|\widetilde{u}(s)\|_{\mathcal{D}(M)}\dd s\leq Ct\|v\|_{\mathcal{D}(M)}\quad\forall\,v\in \mathcal{D}(M),
\end{align*}
which leads to the first assertion of (3).

For any $t\in [0,\Delta t]$, we combine \eqref{system_1}, \eqref{system_2} and {\em Step 1} to get
\begin{align*}
	&\big\|(e^{t(M-\sigma I)}-e^{-\sigma\Delta t})v\big\|_{\H}=\|\widetilde{u}(t)-\overline{u}(\Delta t)\|_{\H}\\
	&=\bigg\|\int_{0}^{t}M\widetilde{u}(s)\dd s+\int_{0}^{t}\sigma(\overline{u}(s)-\widetilde{u}(s))\dd s+\int_{t}^{\Delta t}\sigma\overline{u}(s)\dd s\bigg\|_{\H}\\
	&\leq C\Delta t\|v\|_{\mathcal{D}(M)},
\end{align*}
which yields the second assertion of (3).

(4) For the proof, see \cite[Lemma 3.5, Lemma 5.2]{CCC2021}.

(5) From (1)--(4), it holds that
\begin{align*}
	&\|\hat{S}(t_n-r)-(\hat{S}_{\Delta t})^{n-k-1}T_{\Delta t}\|_{\mathcal{L}(\mathcal{D}(M),\H)}\\
	&\leq\|\hat{S}(t_n-t_{k+1})(\hat{S}(t_{k+1}-r)-I)\|_{\mathcal{L}(\mathcal{D}(M),\H)}+\|\hat{S}(t_n-t_{k+1})-(\hat{S}_{\Delta t})^{n-k-1}\|_{\mathcal{L}(\mathcal{D}(M),\H)}\\
	&\qquad+\|(\hat{S}_{\Delta t})^{n-k-1}(I-T_{\Delta t})\|_{\mathcal{L}(\mathcal{D}(M),\H)}\\
	&\leq Ce^{-\sigma_0(n-k-1)\Delta t/2}\Delta t^{1/2}.
\end{align*}
Thus we finish the proof.

\section{The Proof of Lemma \ref{ooq}}
\begin{proof}
	(1) For the proof, see \cite[Lemma 5.2]{CCC2021}.\\
	(2) Let $q^{n+1}:=(\widehat{S}_{\Delta t})^{n+1}v$ with $q^n:=((q_1^n)^{\top},(q_2^n)^{\top})^{\top}$ and $q^0=v$. Then 
	\begin{align}\label{ujb}
		q^{n+1}-e^{-\sigma_0\Delta t}q^{n}=\frac{\Delta t}{2}\big(Mq^{n+1}+M(e^{-\sigma_0\Delta t}q^{n})\big),
	\end{align}
	which implies 
	\begin{align*}
		\nabla\cdot q_1^{n+1}=e^{-\sigma_0\Delta t}\nabla\cdot q^{n}_1,\qquad\nabla\cdot q^{n+1}_2=e^{-\sigma_0\Delta t}\nabla\cdot q_2^{n}.
	\end{align*}
	Hence we obtain 
	$$\|\nabla\cdot q^n_1\|^2_{L^2(D)}+\|\nabla\cdot q^n_2\|^2_{L^2(D)}=e^{-2n\sigma_0\Delta t}\big(\|\nabla\cdot v_1\|^2_{L^2(D)}+\|\nabla\cdot v_2\|^2_{L^2(D)}\big).$$
	Combining \eqref{esx} and using the fact that $f\in H({\rm curl},D)\cap H({\rm div},D)$ belongs to $H^1(D)^3$ if $\mathbf{n}\times f|_{\partial D}=0$ or $\mathbf{n}\cdot f|_{\partial D}=0$, we have
	\begin{align*}
		\|q^n\|_{H^1(D)^6}\leq Ce^{-\sigma_0n\Delta t/2}\|v\|_{H^1(D)^6}.
	\end{align*}
	(3) Let $p^n_h:=\big(\widehat{S}_{h,\Delta t}\big)^n\pi_hv$, that is, 
	\begin{equation*}
		\left\{
		\begin{split}
			&p_h^n=\big(I-\frac{\Delta t}{2}M_h\big)^{-1}\big(I+\frac{\Delta t}{2}M_h\big)e^{-\sigma_0\Delta t}p_h^{n-1}\\
			&p_h^0=\pi_hv,
		\end{split}		
		\right.
	\end{equation*}
	which yields
	\begin{align}\label{dhg}
		p^{n+1}_h-e^{-\sigma_0\Delta t}p_h^n=\frac{\Delta t}{2}\big(M_hp^{n+1}_h+e^{-\sigma_0\Delta t}M_hp^n_h\big).
	\end{align}
	By applying $\pi_h$ on both sides of \eqref{ujb} and using \cite[Proposition 4.4(i)]{CCC2021}, we obtain
	\begin{align}\label{rtg}
		\pi_hq^{n+1}-e^{\sigma_0\Delta t}\pi_hq^n=\frac{\Delta t}{2}\big(M_hq^{n+1}+e^{-\sigma_0\Delta t}M_hq^n\big).
	\end{align}
	
	Let $e_h^n:=p^n_h-\pi_hq^n$ and $e_\pi^n:=\pi_hq^n-q^n$. Subtracting \eqref{rtg} from \eqref{dhg} leads to
	\begin{equation}\label{wew}
		e^{n+1}_h-e^{-\sigma_0\Delta t}e^n_h=\frac{\Delta t}{2}\Big[M_h\big(e^{n+1}_h+e^{-\sigma_0\Delta t}e^n_h\big)+M_h\big(e^{n+1}_{\pi}+e^{-\sigma_0\Delta t}e^n_{\pi}\big)\Big].
	\end{equation}		
	We apply $\langle\cdot,e^{n+1}_h+e^{-\sigma_0\Delta t}e^n_h\rangle_{\H}$ on both sides of \eqref{wew} to obtain
	\begin{align*}
		\|e^{n+1}_h\|^2_{\H}-e^{-2\sigma_0\Delta t}\|e^n_h\|^2_{\H}&=\frac{\Delta t}{2}\langle M_h(e^{n+1}_h+e^{-\sigma_0\Delta t}e^n_h),e^{n+1}_h+e^{-\sigma_0\Delta t}e^n_h\rangle_{\H}\\
		&\quad+\frac{\Delta t}{2}\langle M_h(e^{n+1}_\pi+e^{-\sigma_0\Delta t}e^n_\pi),e^{n+1}_h+e^{-\sigma_0\Delta t}e^n_h\rangle_{\H}.
	\end{align*}	
	It follows from \cite[Proposition 4.4(iii)]{CCC2021} that
	\begin{align*}
		&\langle M_h(e^{n+1}_\pi+e^{-\sigma_0\Delta t}e^n_{\pi}),e^{n+1}_h+e^{-\sigma_0\Delta t}e^n_h\rangle_{\H}\\
		&=\frac{1}{2}\sum_{F\in\mathcal{G}_h^{\text{int}}}\bigg(\Big\langle (e^{n+1}_{\pi,\Hh}+e^{-\sigma_0\Delta t}e^n_{\pi,\Hh})|_{K_F}+(e^{n+1}_{\pi,\Hh}+e^{-\sigma_0\Delta t}e^n_{\pi,\Hh})|_{K}\\
		&\qquad\qquad-\mathbf{n}_F\times[[e^{n+1}_{\pi,\Ee}+e^{-\sigma_0\Delta t}e^n_{\pi,\Ee}]]_F,\mathbf{n}_F\times[[e^{n+1}_{h,\Ee}+e^{-\sigma_0\Delta t}e^n_{h,\Ee}]]_F\Big\rangle_{L^2(F)^3}\bigg)\\
		&\quad-\frac{1}{2}\sum_{F\in\mathcal{G}_h^{\text{int}}}\bigg(\Big\langle (e^{n+1}_{\pi,\Ee}+e^{-\sigma_0\Delta t}e^n_{\pi,\Ee})|_{K_F}+(e^{n+1}_{\pi,\Ee}+e^{-\sigma_0\Delta t}e^n_{\pi,\Ee})|_{K}\\
		&\qquad\qquad+\mathbf{n}_F\times[[e^{n+1}_{\pi,\Hh}+e^{-\sigma_0\Delta t}e^n_{\pi,\Hh}]]_F,\mathbf{n}_F\times[[e^{n+1}_{h,\Hh}+e^{-\sigma_0\Delta t}e^n_{h,\Hh}]]_F\Big\rangle_{L^2(F)^3}\bigg)\\
		&\quad-\sum_{F\in\mathcal{G}_h^{\text{ext}}}\bigg(\langle (e^{n+1}_{\pi,\Hh}+e^{-\sigma_0\Delta t}e^n_{\pi,\Hh}),\mathbf{n}_F\times(e^{n+1}_{h,\Ee}+e^{-\sigma_0\Delta t}e^n_{h,\Ee})\rangle_{L^2(F)^3}\\
		&\qquad\qquad+\langle \mathbf{n}_F\times(e^{n+1}_{\pi,\Ee}+e^{-\sigma_0\Delta t}e^n_{\pi,\Ee}),\mathbf{n}_F\times(e^{n+1}_{h,\Ee}+e^{-\sigma_0\Delta t}e^n_{h,\Ee})\rangle_{L^2(F)^3}\bigg),
	\end{align*}
	where $e^{n}_\pi=((e^n_{\pi,\Ee})^\top,(e^n_{\pi,\Hh})^\top)^\top$ and $e^{n}_h=((e^n_{h,\Ee})^\top,(e^n_{h,\Hh})^\top)^\top$.
	Then the Cauchy--Schwarz and Young inequalities lead to
	\begin{align*}
		&\langle M_h(e^{n+1}_\pi+e^{-\sigma_0\Delta t}e^n_{\pi}),e^{n+1}_h+e^{-\sigma_0\Delta t}e^n_h\rangle_{\H}\\
		&\leq \frac{1}{2}\sum_{F\in\mathcal{G}_h^{\text{ext}}}\|\mathbf{n}_F\times(e^{n+1}_{h,\Ee}+e^{-\sigma_0\Delta t}e^n_{h,\Ee})\|^2_{L^2(F)^3}\\
		&\quad+\frac{1}{4}\sum_{F\in\mathcal{G}_h^{\text{int}}}\Big(\|\mathbf{n}_F\times[[e^{n+1}_{h,\Hh}+e^{-\sigma_0\Delta t}e^n_{h,\Hh}]]_F\|^2_{L^2(F)^3}\\
		&\qquad\qquad\qquad\qquad+\|\mathbf{n}_F\times[[e^{n+1}_{h,\Ee}+e^{-\sigma_0\Delta t}e^n_{h,\Ee}]]_F\|^2_{L^2(F)^3}\Big)\\
		&\quad+\frac{1}{4}\sum_{F\in\mathcal{G}_h^{\text{int}}}\Big(\|(e^{n+1}_{\pi,\Hh}+e^{-\sigma_0\Delta t}e^n_{\pi,\Hh})|_{K_F}+(e^{n+1}_{\pi,\Hh}+e^{-\sigma_0\Delta t}e^n_{\pi,\Hh})|_{K}\\
		&\qquad\qquad\qquad\qquad
		-\mathbf{n}_F\times[[e^{n+1}_{\pi,\Ee}+e^{-\sigma_0\Delta t}e^n_{\pi,\Ee}]]_F\|^2_{L^2(F)^3}\Big)\\
		&\quad+\frac{1}{4}\sum_{F\in\mathcal{G}_h^{\text{int}}}\Big(\|(e^{n+1}_{\pi,\Ee}+e^{-\sigma_0\Delta t}e^n_{\pi,\Ee})|_{K_F}+(e^{n+1}_{\pi,\Ee}+e^{-\sigma_0\Delta t}e^n_{\pi,\Ee})|_{K}\\
		&\qquad\qquad\qquad\qquad+\mathbf{n}_F\times[[e^{n+1}_{\pi,\Hh}+e^{-\sigma_0\Delta t}e^n_{\pi,\Hh}]]_F\|^2_{L^2(F)^3}\Big)\\
		&\quad+\sum_{F\in\mathcal{G}_h^{\text{ext}}}\Big(\|e^{n+1}_{\pi,\Hh}+e^{-\sigma_0\Delta t}e^n_{\pi,\Hh}\|^2_{L^2(F)^3}+\|\mathbf{n}_F\times(e^{n+1}_{\pi,\Ee}+e^{-\sigma_0\Delta t}e^n_{\pi,\Ee})\|^2_{L^2(F)^3}\Big)\\
		&\leq -\frac{1}{2}\langle M_h(e^{n+1}_h+e^{-\sigma_0\Delta t}e^n_h),e^{n+1}_h+e^{-\sigma_0\Delta t}e^n_h\rangle_{\H}+Che^{-\sigma_0n\Delta t}\|v\|^2_{H^1(D)^6},
	\end{align*}
	where in the last step we use \cite[Proposition 4.4(ii)]{CCC2021}, \eqref{wdx} and the assertion (2). Hence, we have
	\begin{align*}
		\|e^{n+1}_h\|^2_{\H}&\leq e^{-2\sigma_0\Delta t}\|e^n_h\|_{\H}^2+Che^{-\sigma_0n\Delta t}\Delta t\|v\|^2_{H^1(D)^6}
		\leq Che^{-\sigma_0n\Delta t}\|v\|^2_{H^1(D)^6}.
	\end{align*}	
	(4) Let $p(t_n):=\widehat{S}(t_n)v$. Notice that
	\begin{align*}
		\|p(t_n)-p^n_h\|_{\H}&=\|(\widehat{S}(t_n)-(S_{h,\Delta t})^n\pi_h)v\|_{\H}\\
		&\leq \|(\widehat{S}(t_n)-(\widehat{S}_{\Delta t})^n)v\|_{\H}+\|(I-\pi_h)(\widehat{S}_{\Delta t})^nv\|_{\H}\\
		&\quad+\|(\pi_h(\widehat{S}_{\Delta t})^n-(\widehat{S}_{h,\Delta t})^n\pi_h)v\|_{\H}\\
		&\leq Ce^{-\sigma_0n\Delta t/2}\big(\Delta t^{1/2}+h^{1/2}\big)
	\end{align*}	
	due to Lemma \ref{opee}(2), \eqref{touyingwucha} and the assertion (3).\\	
	(5) We know that
	\begin{align*}
		\widehat{S}(t_n-r)-(\widehat{S}_{h,\Delta t})^{n-k-1}T_{h,\Delta t}&= \widehat{S}(t_n-t_{k+1})(\widehat{S}(t_{k+1}-r)-I)\\
		&\quad+\Big(\widehat{S}(t_n-t_{k+1})-(\widehat{S}_{h,\Delta t})^{n-k-1}\pi_h\Big)\\
		&\quad+(\widehat{S}_{h,\Delta t})^{n-k-1}(\pi_h-T_{h,\Delta t}).
	\end{align*}	
	By Lemma \ref{opee}(1)(3), one gets
	\begin{align*}
		&\|\widehat{S}(t_n-t_{k+1})(\widehat{S}(t_{k+1}-r)-I)v\|_{\H}\\
		&\leq \|\widehat{S}(t_n-t_{k+1})\|_{\mathcal{L}(\H,\H)}\|\widehat{S}(t_{k+1}-r)-I\|_{\mathcal{L}(\mathcal{D}(M),\H)}\|v\|_{\mathcal{D}(M)}\\
		&\leq C\Delta te^{-\sigma_0(n-k-1)\Delta t}\|v\|_{H^1(D)^6},
	\end{align*}	
	which by the assertion (4) leads to
	\begin{align*}
		\|(\widehat{S}(t_n-t_{k+1})-(\widehat{S}_{h,\Delta t})^{n-k-1}\pi_h)v\|_{\H}\leq Ce^{-\sigma_0(n-k-1)\Delta t/2}\big(\Delta t^{\frac{1}{2}}+h^{\frac{1}{2}}\big)\|v\|_{H^1(D)^6}.
	\end{align*}	
	
	Let $\widetilde{v}:=\big(I-(I-\frac{\Delta t}{2}M_h)^{-1}\big)\pi_hv$, namely,
	\begin{align}\label{rhg}
		\widetilde{v}-\frac{\Delta t}{2}M_h\widetilde{v}=-\frac{\Delta t}{2}M_h\pi_hv.
	\end{align}
	We apply $\langle\cdot,\widetilde{v}\rangle_{\H}$ on both sides of \eqref{rhg} to obtain
	\begin{align*}
		\|\widetilde{v}\|_{\H}^2\leq -\frac{\Delta t}{2}\langle M_h\pi_hv,\widetilde{v}\rangle_{\H}=-\frac{\Delta t}{2}\langle M_h\pi_hv,\pi_h\widetilde{v}\rangle_{\H}\leq \frac{\Delta t}{2}\Big(\frac{1}{2}\|\widetilde{v}\|_{\H}^2+C\|v\|^2_{H^1(D)^6}\Big)
	\end{align*}	
	due to \cite[Lemma A.4]{IM2015}. For sufficient small $\Delta t$, we have
	$$\|\widetilde{v}\|_{\H}\leq C\Delta t^{\frac{1}{2}}\|v\|_{H^1(D)^6}.$$
	Hence, it holds that
	\begin{align*}
		\|(\widehat{S}(t_n-r)-(\widehat{S}_{h,\Delta t})^{n-k-1}T_{h,\Delta t})v\|_{\H}\leq Ce^{-\sigma_0(n-k-1)\Delta t/2}\big(\Delta t^{\frac{1}{2}}+h^{\frac{1}{2}}\big)\|v\|_{H^1(D)^6}.
	\end{align*}
\end{proof}
\bibliographystyle{plain}
\bibliography{maxwell.bib}

\begin{thebibliography}{10}

\bibitem{CCC2021}
C.~Chen.
\newblock A symplectic discontinuous {G}alerkin full discretization for
  stochastic {M}axwell equations.
\newblock {\em SIAM J. Numer. Anal.}, 59(4):2197--2217, 2021.

\bibitem{CH2016}
C.~Chen and J.~Hong.
\newblock Symplectic {R}unge-{K}utta semidiscretization for stochastic
  {S}chr\"odinger equation.
\newblock {\em SIAM J. Numer. Anal.}, 54(4):2569--2593, 2016.

\bibitem{CHJ2019a}
C.~Chen, J.~Hong, and L.~Ji.
\newblock Mean-square convergence of a semidiscrete scheme for stochastic
  {M}axwell equations.
\newblock {\em SIAM J. Numer. Anal.}, 57(2):728--750, 2019.

\bibitem{CHJ2019b}
C.~Chen, J.~Hong, and L.~Ji.
\newblock Runge-{K}utta semidiscretizations for stochastic {M}axwell equations
  with additive noise.
\newblock {\em SIAM J. Numer. Anal.}, 57(2):702--727, 2019.

\bibitem{CHZ2016}
C.~Chen, J.~Hong, and L.~Zhang.
\newblock Preservation of physical properties of stochastic {M}axwell equations
  with additive noise via stochastic multi-symplectic methods.
\newblock {\em J. Comput. Phys.}, 306:500--519, 2016.

\bibitem{CHCS2020}
D.~Cohen, J.~Cui, J.~Hong, and L.~Sun.
\newblock Exponential integrators for stochastic {M}axwell's equations driven
  by {I}t\^{o} noise.
\newblock {\em J. Comput. Phys.}, 410:109382, 21, 2020.

\bibitem{CHS2021}
J.~Cui, J.~Hong, and L.~Sun.
\newblock Weak convergence and invariant measure of a full discretization for
  parabolic {SPDE}s with non-globally {L}ipschitz coefficients.
\newblock {\em Stochastic Process. Appl.}, 134:55--93, 2021.

\bibitem{DaPrato2008}
G.~Da~Prato.
\newblock {\em An introduction to infinite-dimensional analysis}.
\newblock Universitext. Springer-Verlag, Berlin, 2006.

\bibitem{hairer2006ergodic}
M.~Hairer.
\newblock Ergodic properties of {M}arkov processes.
\newblock {\em https://www.hairer.org/notes/Markov.pdf}, 2006.

\bibitem{MartingHairer2011}
M.~Hairer, J.~C. Mattingly, and M.~Scheutzow.
\newblock Asymptotic coupling and a general form of {H}arris' theorem with
  applications to stochastic delay equations.
\newblock {\em Probab. Theory Related Fields}, 149(1-2):223--259, 2011.

\bibitem{IM2015}
M.~Hochbruck and T.~Pa\v{z}ur.
\newblock Implicit {R}unge-{K}utta methods and discontinuous {G}alerkin
  discretizations for linear {M}axwell's equations.
\newblock {\em SIAM J. Numer. Anal.}, 53(1):485--507, 2015.

\bibitem{HJZ2014}
J.~Hong, L.~Ji, and L.~Zhang.
\newblock A stochastic multi-symplectic scheme for stochastic {M}axwell
  equations with additive noise.
\newblock {\em J. Comput. Phys.}, 268:255--268, 2014.

\bibitem{HJZC2017}
J.~Hong, L.~Ji, L.~Zhang, and J.~Cai.
\newblock An energy-conserving method for stochastic {M}axwell equations with
  multiplicative noise.
\newblock {\em J. Comput. Phys.}, 351:216--229, 2017.

\bibitem{HW2019}
J.~Hong and X.~Wang.
\newblock {\em Invariant measures for stochastic nonlinear {S}chr\"{o}dinger
  equations, Numerical approximations and symplectic structures}, volume 2251
  of {\em Lecture Notes in Mathematics}.
\newblock Springer, Singapore, 2019.

\bibitem{JCL2019}
C.~Jiang, J.~Cui, and Y.~Wang.
\newblock A conformal energy-conserved method for {M}axwell's equations with
  perfectly matched layers.
\newblock {\em Commun. Comput. Phys.}, 25(1):84--106, 2019.

\bibitem{jiang2013stochastic}
S.~Jiang, L.~Wang, and J.~Hong.
\newblock Stochastic multi-symplectic integrator for stochastic nonlinear
  {S}chr\"{o}dinger equation.
\newblock {\em Commun. Comput. Phys.}, 14(2):393--411, 2013.

\bibitem{RKT1989}
S.~M. Rytov, Yu.~A. Kravtsov, and V.~I. Tatarski\u{\i}.
\newblock {\em Principles of statistical radiophysics. 3. Elements of random
  fields}.
\newblock Springer-Verlag, Berlin, 1989.

\bibitem{stochasticKGequation}
M.~Song, X.~Qian, T.~Shen, and S.~Song.
\newblock Stochastic conformal schemes for damped stochastic {K}lein--{G}ordon
  equation with additive noise.
\newblock {\em J. Comput. Phys.}, 411:109300, 20, 2020.

\bibitem{SHL2016}
H.~Su and S.~Li.
\newblock Energy/dissipation-preserving {B}irkhoffian multi-symplectic methods
  for {M}axwell's equations with dissipation terms.
\newblock {\em J. Comput. Phys.}, 311:213--240, 2016.

\bibitem{sun2022discontinuous}
J.~Sun, C.~Shu, and Y.~Xing.
\newblock Discontinuous {G}alerkin methods for stochastic {M}axwell equations
  with multiplicative noise.
\newblock {\em arXiv:2204.09632}.

\bibitem{SQW2022}
J.~Sun, C.~Shu, and Y.~Xing.
\newblock Multi-symplectic discontinuous {G}alerkin methods for the stochastic
  {M}axwell equations with additive noise.
\newblock {\em J. Comput. Phys.}, 461:111199, 2022.

\bibitem{optimaltransport}
C.~Villani.
\newblock {\em Optimal transport, Old and new}, volume 338 of {\em Grundlehren
  der mathematischen Wissenschaften [Fundamental Principles of Mathematical
  Sciences]}.
\newblock Springer-Verlag, Berlin, 2009.

\end{thebibliography}

\end{document}